\documentclass[final,floating,letterpaper]{siamltex1213}
\usepackage{}
\usepackage{latexsym,amsmath,amssymb}
\allowdisplaybreaks[4]

\usepackage{comment}
\usepackage[compress]{cite}
\usepackage{algorithm}
\usepackage{algorithmic}

\algsetup{indent=2em}
\usepackage{mathbbold}
\usepackage{bbm}
\usepackage{amsfonts}
\usepackage[T1]{fontenc}
\usepackage{latexsym,amssymb,amsmath,amsfonts}
\usepackage{graphicx}
\usepackage{subfigure}
\usepackage{epsf}
\usepackage{epstopdf}
\usepackage{amssymb}
\usepackage{hyperref}

\usepackage{bm}% bold math
\newtheorem{remark}[theorem]{Remark}

\renewcommand{\theequation}{\arabic{section}.\arabic{equation}}

\newcommand{\R}{{\mathbb R}}
\newcommand{\Z}{{\mathbb Z}}

\newcommand{\C}{{\mathbb C}}

\newcommand{\be}{\begin{eqnarray}}
\newcommand{\ben}{\begin{eqnarray*}}
\newcommand{\en}{\end{eqnarray}}
\newcommand{\enn}{\end{eqnarray*}}
\newcommand{\ba}{\backslash}
\newcommand{\pa}{\partial}

\newcommand{\ov}{\overline}

\newcommand{\ggrad}{{\rm grad\,}}        %%%%%%%%%%%%%%%%%%%%%%%%%%%%%%%%%%%%%%%
\newcommand{\ddiv}{{\rm div\,}}

\newcommand{\g}{\gamma}
\newcommand{\G}{\Gamma}
\newcommand{\eps}{\epsilon}

\newcommand{\Om}{\Omega}
\newcommand{\om}{\omega}

\newcommand{\la}{\lambda}

\newcommand{\wi}{\widetilde}

\newcommand{\hx}{\hat{x}}
\title{Inverse elastic scattering problems with phaseless far field data}

\author{Xia Ji\thanks{LSEC, Academy of Mathematics and Systems Science, Chinese Academy of Sciences, Beijing 100190, China ({\tt jixia@lsec.cc.ac.cn}).}
\and
Xiaodong Liu\thanks{Institute of Applied Mathematics, Academy of Mathematics and Systems Science, Chinese Academy of Sciences, 100190 Beijing, China({\tt xdliu@amt.ac.cn}).}
}
\begin{document}

\maketitle

\begin{abstract}
This paper is concerned with uniqueness, phase retrieval and shape reconstruction methods for inverse elastic scattering problems with phaseless far field data.
The phaseless far field data is closely related to the outward energy flux, which is easily measured in practice.  Systematically, we study two basic models, i.e., inverse scattering of plane waves by rigid bodies and inverse scattering of sources with compact support. For both models, we show that the phaseless far field data is invariant under translation of the underlying scattering objects, which implies that the location of the objects can not be uniquely recovered by the data.
To solve this problem, we consider simultaneously the incident point sources with one fixed source point and at most three scattering strengths. With this technique, we establish some uniqueness results for source scattering problem with multi-frequency phaseless far field data. Furthermore, a fast and stable phase retrieval approach is proposed based on a simple geometric result which provides a stable reconstruction of a point in the plane from three distances to given points. Difficulties arise for inverse scattering by rigid bodies due to the additional unknown far field pattern of the point sources. To overcome this difficulty, we introduce an artificial rigid body into the system and show that the underlying rigid bodies can be uniquely determined by the corresponding phaseless far field data at a fixed frequency. Noting that the far field pattern of the scattered field corresponding to point sources is very small if the source point is far away from the scatterers, we propose an appropriate phase retrieval method for obstacle scattering problems, without using the artificial rigid body. Finally, we propose several sampling methods for shape reconstruction with phaseless far field data directly. For inverse obstacle scattering problems, two different direct sampling methods are proposed with data at a fixed frequency. For inverse source scattering problems, we introduce two direct sampling methods for source supports with sparse multi-frequency data.  The phase retrieval techniques are also combined with the classical sampling methods for the shape reconstructions.
Extended numerical examples in two dimensions are conducted with noisy data, and the results further verify the effectiveness and robustness of
the proposed phase retrieval techniques and sampling methods.

\vspace{.2in} {\bf Keywords:}
Elastic scattering, phaseless far field data, uniqueness, phase retrieval, direct sampling methods.

\vspace{.2in} {\bf AMS subject classifications:}
35P25, 45Q05, 78A46, 74B05

\end{abstract}
\section{Introduction}
The inverse wave scattering problems are to determine the unknown objects using the prescribed radiated wave fields.
These problems are mainly motivated by a lot of practical applications. For example, waves are transmitted through the earth to detect oil or gas and to study the earth's geological structure, waves are also used on the human body for medical diagnosis and therapy.
The measurement of the scattering of waves has been developed into an effective engineering technique for the nondestructive evaluation of materials for detection of potentially dangerous defects as cracks.
A survey on the state of the art of the mathematical theory and numerical approaches for inverse time harmonic acoustic and electromagnetic scattering problems can be found in the standard monograph \cite{CK}.
The inverse elastic wave scattering is more challenging due to the coupling of compressional waves and shear waves that propagate at different speeds.
We refer to \cite{GintidesSini, Hahner98, HahnerHsiao, HuKirschSini} for the uniqueness and stability theory. Further, several numerical approaches have been developed for shape or source reconstruction \cite{Arens2001,AlvesKress2002,BaoChenLi,BaoHuSunYin,ChenHuang,GuzinaChikichev,HuKirschSini,HuLiLiu,HuLiLiuSun,HuLiLiuWang,HuLiLiuWang,JiLiuXi,LiWangWangZhao,TYG, Sevroglou}.
For the readers interested in a more comprehensive treatment of the direct and inverse elastic scattering problems, we suggest consulting \cite{BonnetConstantinescu, DassiosKiriakiPolyzos, Kupradze1965, Kupradze1979}  on this subject.

The two well known difficulties of the inverse scattering problems are nonlinearity and ill-posedness. Actually, in many cases of practical interest, the third difficulty is incomplete data, i.e., only partial information can be measured directly. There are two cases of incomplete data. The first one is the limited-aperture data, where the measurements are only available in limited directions or positions. Limited-aperture data can present a severe challenge for the reconstructions \cite{JiLiuXi, LiWangWangZhao}. Recently, some data recovery techniques have been developed to obtain the full-aperture data by combining data and models \cite{JiLiuXi,LiuSun17}.
%Unfortunately, the first kind of integral equations have to be solved, and thus these techniques are ill-posed.
The second one is the phaseless data due to the fact that the phase information is difficult or even impossible to be captured.
Many works with phased far field patterns rely heavily on the fact that, by Rellich's lemma, the radiating waves and their far field patterns are one-to-one.
Unfortunately, the corresponding result is not available for the modulus of the far field patterns. Actually, it is well known that the modulus of the far field pattern is invariant under the translation of underlying objects \cite{KR97,KR99}. Thus, the location of the underlying objects can not be uniquely determined. We refer to \cite{AmmariChowZou,ChengHuang-phaseless,DZG,Klibanov14,Klibanov17,KlibanovRomanov17,KlibanovRomanov17-SIAM,KR97,KR99,LZball,XZZ18,XZZ-SIAM,ZG,ZhangGuoLiLiu,ZZ18,ZZ-ip17,ZZ-jcp17}
for related works on shape reconstruction for phaseless inverse acoustic scattering problems.

In this work, we consider the elastic scattering problems with phaseless far field data. To our best knowledge, this is the first work on uniqueness, phase retrieval and sampling methods for inverse elastic scattering problem with phaseless far field data. We refer to a recent manuscript \cite{DLL} on iterative methods with a reference ball for determining a rigid body using phaseless far field data.
By considering simultaneously the scattering of point sources, we introduce a fast and simple phase retrieval method, and then propose some direct sampling methods with phaseless far field data for shape reconstruction.
This work is a nontrivial extension of the results in \cite{JiLiuZhang-obstacle,JiLiuZhang-source} for the inverse acoustic scattering problem of the Helmholtz equation to the inverse elastic scattering of the Navier equation. The elastic wave equation is more challenging because of the coexistence of compressional waves and shear waves propagating at different speeds. The phaseless acoustic far field data for point sources/scatterers is just the modulus of the strength, which is a constant \cite{JiLiuZhang-obstacle,JiLiuZhang-source}.
However, the phaseless elastic far field data for point sources/scatterers changes for different observation directions, and even vanishes at some special directions. Hence more sophisticated modification is required.

This paper is organized as follows.
In the next section, we introduce the two basic scattering models, i.e., scattering of plane waves by rigid bodies and scattering of sources with compact support.
We show that the phaseless far field data is invariant under translation of the underlying scattering objects, which implies that the location of the underlying scattering objects can not be uniquely determined by the phaseless far field data. To overcome this difficulty, we consider simultaneously the scattering of point sources.
Section \ref{sec:uniqueness} is devoted to uniqueness results with phaseless far field data. Some uniqueness results for inverse source scattering problem have been established using multi-frequency phaseless far field data. For inverse scattering by rigid bodies, with the help of an artificial rigid body, we prove some uniqueness results using phaseless far field data at a fixed frequency.
In Section \ref{sec:phaseretrieval+Shapereconstruction}, we introduce a phase retrieval technique and propose some direct sampling methods using phaseless far field data. Only one kind of phaseless far field data is needed in the numerical schemes, and the artificial rigid body is avoided.
Extended numerical simulations in two dimensions are presented in the last section to indicate the efficiency and robustness of the proposed methods.

In this paper, we aim at conciseness, so that we restrict ourselves to the two dimensional case. Three dimensional case can be dealt with similarly after some modifications.
Our methods proposed are independent of physical properties of the scatterers. However, for simplicity, we restrict our presentation to the case where the scatterer is a rigid body in the obstacle scattering problem. The other cases such as scattering by a cavity or by a penetrable inhomogeneous medium of compact support can be treated analogously.

\section{Elastic scattering problems}
This section is devoted to address the elastic scattering problems.
We begin with the notations used throughout this paper. All vectors will be denoted in bold script.
For a vector $\mathbf{x}:=(x_1, x_2)^{T}\in\R^2$, we introduce the two unit vectors $\hat{\mathbf{x}}:=\mathbf{x}/|\mathbf{x}|$ and $\hat{\mathbf{x}}^{\perp}$ obtained by
rotating $\hat{\mathbf{x}}$ anticlockwise by $\pi/2$.
For simplicity, we write $\pa_i$ for the usual partial derivative $\frac{\pa}{\pa x_i},\,i=1,2$.
Then, in addition to the usual differential operators $\ggrad:=(\pa_1, \pa_2)^T$ and $\ddiv:=(\pa_1, \pa_2)$,
we define two auxiliary differential operators $\ggrad^{\perp}$ and $\ddiv^{\perp}$ by
$\ggrad^{\perp}:= (-\pa_2, \pa_1)^T$ and $\ddiv^{\perp}:= (-\pa_2, \pa_1)$, respectively. It is easy to deduce the differential indentities $\ddiv^{\perp}\ggrad = \ddiv\ggrad^{\perp}=0$.
In general, we shall use the notations $\ggrad_{\mathbf{y}}$ and $\ddiv_{\mathbf{y}}$ to denote gradient and divergence with respect to $\mathbf{y}$, and $ds(\mathbf{y})$ and $d\mathbf{y}$ to remind the reader that the appropriate integrals are with respect to $\mathbf{y}$.
Denote by $\mathbb{S}:=\{\mathbf{x}\in \R^2:|\mathbf{x}|=1\}$ the unit circle in $\R^2$.
Let $\Om\subset\R^2$ be an open and bounded Lipschitz domain such that the exterior $\R^2\ba\ov{\Om}$ of $\Om$ is connected.
Here and throughout the paper we denote $\ov{\Om}$ the closure of the set $\Om$. A confusion with the complex conjugate $\ov{z}$ of $z\in\C$ is not expected.
Let $\omega>0$ be the circular frequency, $\lambda$ and $\mu$ be Lam$\acute{\mbox{e}}$ constants satisfying $\mu>0, \,2\mu+\lambda>0$.
Furthermore, denote by
\ben
k_p:=\om/\sqrt{\la+2\mu}\quad\mbox{and}\quad k_s:=\om/\sqrt{\mu}
\enn
the compressional and shear wave number, respectively.

The two basic problems in classical scattering theory are the scattering of elastic waves by bounded scatterers and of external forces with compact support.
In this section we will present the details on these two problems in Subsections \ref{ESP-O} and \ref{ESP-F}, respectively.
We show that the corresponding phaseless elastic far field data is invariant under translation of the sources/scatterers. Thus, location of the underlying objects can not be uniquely determined from the phaseless elastic far field data. To solve this problem, we consider simultaneously the scattering of point sources.

\subsection{Scattering of plane waves by scatterers}\label{ESP-O}
The first problem is the scattering of elastic waves by a bounded scatterer $\Om$.
As incident fields $\mathbf{u}^{in}$, plane waves are of special interest. The time harmonic elastic plane wave with incident direction $\mathbf{d}\in \mathbb{S}$ is given by
\be\label{plane-waves}
\mathbf{u}^{in}=a_{p}\mathbf{u}_{p}^{in}+a_{s}\mathbf{u}_{s}^{in},\quad a_p,\,a_s\in \C,
\en
where $\mathbf{u}_{p}^{in}:=\mathbf{d}e^{ik_{p}\mathbf{x}\cdot \mathbf{d}}$ is a plane compressional wave and $\mathbf{u}_{s}^{in}:=\mathbf{d}^{\perp}e^{ik_{s}\mathbf{x}\cdot \mathbf{d}}$ is a plane shear wave, respectively.

The propagation of time-harmonic elastic wave equation in an isotropic homogeneous media outside $\Om$ is governed by the reduced Navier equation
\be\label{elastic-O}
\Delta^{\ast} \mathbf{u}_{\Om}+\omega^2 \mathbf{u}_{\Om}=0\quad \mbox{in}\,\,\R^2\ba\ov{\Om}, \quad \Delta^{\ast}:= \mu\Delta +(\lambda+\mu) \ggrad\ddiv,
\en
where $\mathbf{u}_{\Om}$ denotes the total displacement field. For a rigid body $\Om$, the total displacement field $\mathbf{u}_{\Om}$ satisfies the first (Dirichlet) boundary condition
\be\label{Dirichlet}
\mathbf{u}_{\Om}=0\quad\mbox{on}\,\pa \Om.
\en
Straightforward calculations show that the incident plane waves \eqref{plane-waves} is an entire solution of the Navier equation \eqref{elastic-O}.
The scatterer $\Om$ gives rise to a scattered field $\mathbf{u}^{sc}_{\Om}=\mathbf{u}_{\Om}-\mathbf{u}^{in}$, which is also a solution of the Navier equation \eqref{elastic-O}.
In the sequel, for the scattering of an incident plane wave ${\bf u}^{in}_m,m=p,s $, by a rigid body $\Om$, we denote the corresponding scattered field  by $\mathbf{u}^{sc}_{\Om,m}(\mathbf{x}, \mathbf{d}),m=p,s$.
It is well known that the scattered field has a decomposition in the form
\ben
\mathbf{u}^{sc}_{\Om,m}=\mathbf{u}^{sc}_{\Om,mp}+\mathbf{u}^{sc}_{\Om,ms}\quad\mbox{in}\,\,\R^2\ba\ov{\Om},
\enn
where
\ben
\mathbf{u}^{sc}_{\Om,mp}:=-\frac{1}{k^2_p}\ggrad\ddiv\mathbf{u}^{sc}_{\Om,m} \quad\mbox{and}\quad \mathbf{u}^{sc}_{\Om,ms}:=-\frac{1}{k^2_s}\ggrad^{\perp}\ddiv^{\perp}\mathbf{u}^{sc}_{\Om,m},\quad m=p,s,
\enn
are known as the compressional (longitudinal) and shear (transversal) parts of $\mathbf{u}^{sc}_{\Om,m}$, respectively.
It is clear that
\ben
\Delta \mathbf{u}^{sc}_{\Om,mp} +k_p^2\mathbf{u}^{sc}_{\Om,mp}=0,\quad \ddiv^{\perp}\mathbf{u}^{sc}_{\Om,mp}=0 \quad \mbox{in}\,\,\R^2\ba\ov{\Om},\,m=p,s
\enn
and
\ben
\Delta \mathbf{u}^{sc}_{\Om,ms} +k_s^2\mathbf{u}^{sc}_{\Om,ms}=0,\quad \ddiv\mathbf{u}^{sc}_{\Om,ms}=0 \quad \mbox{in}\,\,\R^2\ba\ov{\Om},\,m=p,s.
\enn
To characterize outgoing waves, the scattered fields are required to satisfy the Kupradze's radiation conditions
\be\label{KupradzeRC-O}
\frac{\partial \mathbf{u}_{\Om,mp}^{sc}}{\partial r}-ik_p\mathbf{u}_{\Om,mp}^{sc}=o(r^{-\frac{1}{2}}),\quad
\frac{\partial \mathbf{u}_{\Om,ms}^{sc}}{\partial r}-ik_s\mathbf{u}_{\Om,ms}^{sc}=o(r^{-\frac{1}{2}}),\quad m=p,s
\en
uniformly in all directions $\hat{\mathbf{x}}\in \mathbb{S}$ as $r:=|\mathbf{x}|\rightarrow\infty$.
For the unique solvability of the scattering problems \eqref{elastic-O}-\eqref{KupradzeRC-O} in the space $[H^{1}_{loc}(\R^2\ba\ov{\Om})]^{2}$ we refer to
Kupradze \cite{Kupradze1965} and Li, Wang et al. \cite{LiWangWangZhao}.

It is well known that every radiating solution $\mathbf{u}^{sc}_{\Om,m}$ to the Navier equation has an asymptotic behaviour of the form
\be\label{usasymptotic}
\mathbf{u}^{sc}_{\Om,mp}(\mathbf{x})&=&\frac{k_p}{\om}\frac{e^{i\pi/4}}{\sqrt{8\pi\om}}\frac{e^{ik_p|\mathbf{x}|}}{\sqrt{|{\mathbf{x}}|}}u^\infty_{\Om,mp}(\hat{\mathbf{x}})\hat{\mathbf{x}}+O(|\mathbf{x}|^{-3/2}),\quad \mathbf{|x|}\to \infty,\cr
\mathbf{u}^{sc}_{\Om,ms}(\mathbf{x})&=&\frac{k_s}{\om}\frac{e^{i\pi/4}}{\sqrt{8\pi\om}}\frac{e^{ik_s|\mathbf{x}|}}{\sqrt{|{\mathbf{x}}|}}u^\infty_{\Om,ms}(\hat{\mathbf{x}}){\mathbf{x}}^{\perp}+O(|\mathbf{x}|^{-3/2}), \quad \mathbf{|x|}\to \infty
\en
uniformly in all direction $\mathbf{\hx}\in \mathbb{S}$.
The functions $u^\infty_{\Om,mp}$ and $u^\infty_{\Om,ms}$, known as the compressional and shear far field pattern of $\mathbf{u}^{sc}_{\Om,m}$, respectively, are complex valued analytic functions on $\mathbb{S}$.
The far field patterns admit the following representations (see e.g. (2.29)-(2.30) in \cite{JiLiuXi})
\be
\label{upinfty2}u^\infty_{\Om,mn}\mathbf{(\hat{x},d)}&=& \int_{\pa\Om}\Big\{\mathbf{u}_{n}^{in}(\mathbf{y,-\hat{x}})\cdot\mathbb{T}_{\mathbf{\nu(y)}}\mathbf{u}^{sc}_{\Om,m}\mathbf{(y,d)}\cr
&&\, -[\mathbb{T}_{\mathbf{\nu(y)}}\mathbf{u}_{n}^{in}(\mathbf{y,-\hat{x}})]\cdot\mathbf{u}^{sc}_{\Om,m}\mathbf{(y,d)}\Big\}ds(\mathbf{y}),\quad \mathbf{\hx,d}\in \mathbb{S}, \,m,n=p,s.\quad
\en
Here, for a curve $\G$, $\mathbb{T}_{\nu}$ is the surface traction operator defined by
\ben
\mathbb{T}_{\nu}:=2\mu\nu\cdot\ggrad+\la\nu\ddiv-\mu\nu^{\perp}\ddiv^{\perp}
\enn
in terms of the exterior unit normal vector $\nu$ on $\G$.

Throughout this paper, we will denote the pair of far field patterns
\ben
[u^\infty_{\Om,mp}(\hat{\mathbf{x}}, \mathbf{d});\, u^\infty_{\Om,ms}(\hat{\mathbf{x}}, \mathbf{d})]
\enn
by $\mathbf{u}^\infty_{\Om,m}(\hat{\mathbf{x}}, \mathbf{d}),\,m=p,s$,
indicating the dependence on the observation direction $\hat{\mathbf{x}}\in \mathbb{S}$ and
the incident direction $\mathbf{d}\in \mathbb{S}$.
In our recent work \cite{JiLiuXi}, we have shown that the far field patterns satisfy the reciprocity relation
\be\label{ReciprocityRelations}
u^{\infty}_{\Om, mn}({\mathbf{\hat{x},d}})=u^{\infty}_{\Om, nm}({\mathbf{-d, -\hat{x}}}), \mathbf{\hat{x},d}\in \mathbb{S},\,m,n=p,s.
\en
The classical inverse elastic scattering problem is to determine the scatterer $\Om$ from partial or full far field data $\mathbf{u}^\infty_{\Om,m}(\hat{\mathbf{x}}, \mathbf{d})$, $\mathbf{\hx,d}\in \mathbb{S}, \,m=p,s$.
A wealth of theory and numerical results for such an inverse problem is now available. We refer to \cite{HahnerHsiao, HuKirschSini,GintidesSini} for the uniqueness results. The numerical methods can be found in \cite{AlvesKress2002,Arens2001,HuKirschSini,JiLiuXi,Sevroglou}.

In many applications we have only the modulus of the far field data, while the phase information is difficult to be captured.
Actually, the phaseless far field data is closely related to the outward energy flux. In the case of time harmonic case, the outward energy flux $J_{mr},m=p,s,$ through a circle $\pa B_r$ with radius $r$ centered at the origin is given by
\ben
J_{mr}:=-4\om\Im\int_{\pa B_r}\mathbf{u}^{sc}_{\Om,m}\cdot \ov{\mathbb{T}_{\nu}\mathbf{u}^{sc}_{\Om,m}}ds, \quad m=p,s.
\enn
Here, we take $r$ large enough such that $\Om$ is contained in the interior of $\pa B_r$.
We refer to the corresponding concept for acoustic wave in \cite{Kr1}. \\
\begin{theorem}
\ben
J_{mr}=\int_{\mathbb{S}}\left\{k_p|u^{\infty}_{\Om,mp}|^2+k_s|u^{\infty}_{\Om,ms}|^2\right\}ds,\ m=p,s.
\enn
\end{theorem}
\begin{proof}
Let $B_R$ be a disk with radius $R>r$ centered at the origin.
We now apply Betti's formula \cite{Arens2001} in $B_R\ba\ov{B_r}$, with the help of the Navier equation \eqref{elastic-O}, for $m=p,s$, to obtain
\be\label{r-R}
&&\int_{\pa B_r}\mathbf{u}^{sc}_{\Om,m}\cdot \ov{\mathbb{T}_{\nu}\mathbf{u}^{sc}_{\Om,m}}ds\cr
&=&\int_{\pa B_R}\mathbf{u}^{sc}_{\Om,m}\cdot \ov{\mathbb{T}_{\nu}\mathbf{u}^{sc}_{\Om,m}}ds
  -\int_{B_R\ba\ov{B_r}}\Big\{\mathbf{u}^{sc}_{\Om,m}\cdot \Delta^{\ast}\ov{\mathbf{u}^{sc}_{\Om,m}} +\mathcal {E}(\mathbf{u}^{sc}_{\Om,m},\ov{\mathbf{u}^{sc}_{\Om,m}})\Big\}d\mathbf{x}\cr
&=&\int_{\pa B_R}\mathbf{u}^{sc}_{\Om,m}\cdot \ov{\mathbb{T}_{\nu}\mathbf{u}^{sc}_{\Om,m}}ds
  +\int_{B_R\ba\ov{B_r}}\Big\{\om^2\mathbf{u}^{sc}_{\Om,m}\cdot \ov{\mathbf{u}^{sc}_{\Om,m}} -\mathcal {E}(\mathbf{u}^{sc}_{\Om,m},\ov{\mathbf{u}^{sc}_{\Om,m}})\Big\}d\mathbf{x}.
\en
Here, for any two smooth vector functions $\mathbf{v}=(v_1,v_2)^T$ and $\mathbf{w}=(w_1,w_2)^T$,
\ben
\mathcal {E}(\mathbf{v},\mathbf{w}):=\la \ddiv\mathbf{v} \ddiv\mathbf{w}-\mu\ddiv^{\perp}\mathbf{v}\ddiv^{\perp}\mathbf{w}+2\mu\sum_{i,j=1}^{2}\pa_{j}v_i\pa_{j}w_i.
\enn
Taking the imaginary part of the last equation \eqref{r-R} yields
\ben
J_{mr}=-4\om\Im\int_{\pa B_R}\mathbf{u}^{sc}_{\Om,m}\cdot \ov{\mathbb{T}_{\nu}\mathbf{u}^{sc}_{\Om,m}}ds,\ m=p,s.
\enn
By straightforward calculations, we have the asymptotic behaviour \cite{Arens2001} for $m=p,s$
\be\label{Tusasymptotic}
\mathbb{T}_{\mathbf{\hat{x}}}\mathbf{u}^{sc}_{\Om,m}(\mathbf{x})
&=&i\om\frac{e^{i\pi/4}}{\sqrt{8\pi\om}}\frac{e^{ik_p|\mathbf{x}|}}{\sqrt{|{\mathbf{x}}|}}u^\infty_{\Om,mp}(\hat{\mathbf{x}})\hat{\mathbf{x}}\cr
&&+i\om\frac{e^{i\pi/4}}{\sqrt{8\pi\om}}\frac{e^{ik_s|\mathbf{x}|}}{\sqrt{|{\mathbf{x}}|}}u^\infty_{\Om,ms}(\hat{\mathbf{x}}){\mathbf{x}}^{\perp}+O(|\mathbf{x}|^{-3/2}), \quad \mathbf{x}\to \infty.
\en
The proof is now completed by letting $R\rightarrow\infty$ and using \eqref{usasymptotic} and \eqref{Tusasymptotic}.
\end{proof}

We are interested in the inverse problems with phaseless data $|u^{\infty}_{\Om,mn}|,m,n=p,s$.
Unfortunately, the following translation invariance property for the phaseless far field data implies that it is impossible to determine the location of the underlying obstacle, even multiple directions and multiple frequencies are considered.\\

\begin{theorem}\label{TI-O}
Let $\Om_{\mathbf{h}}:=\{\mathbf{y}\in\R^2: \mathbf{y}=\mathbf{z+h}:\;z\in \Om\}$ be the shifted rigid body with a fixed vector $\mathbf{h}\in \R^{2}$.
Then, for any fixed circular frequency $\om>0$,
\be\label{Trans-Invariance}
|u^{\infty}_{\Om_{\mathbf{h}},mn}(\hat{\mathbf{x}},\mathbf{d})| = |u^{\infty}_{\Om,mn}(\hat{\mathbf{x}},\mathbf{d})|, \quad \hat{\mathbf{x}},\mathbf{d}\in\mathbb{S},\,m,n=p,s.
\en
\end{theorem}
\begin{proof}
We firstly consider the case with $m=p$ and $n=p$, i.e., the compressional far field pattern corresponding to incident plane compressional wave.
For the shifted rigid body $\Om_{\bf h}$, we have
\be\label{u=u}
\mathbf{u}^{sc}_{\Om_{\mathbf{h}},p}(\mathbf{y},\mathbf{d})
&=&-\mathbf{d}e^{ik_p\mathbf{y}\cdot\mathbf{d}}\cr
&=&-e^{ik_p\mathbf{h}\cdot\mathbf{d}}\mathbf{d}e^{ik_p\mathbf{(y-h)}\cdot\mathbf{d}}\cr
&=&e^{ik_p\mathbf{h}\cdot\mathbf{d}}\mathbf{u}^{sc}_{\Om,p}(\mathbf{y-h},\mathbf{d}),\qquad \mathbf{y}\in\pa \Om_{\bf h},\mathbf{d}\in\mathbb{S}.
\en
Note that the radiating solution is completely determined by the boundary value and the fact that $e^{ik_p\mathbf{h}\cdot\mathbf{d}}$ is a constant, we have
\ben
\mathbf{u}^{sc}_{\Om_{\mathbf{h}},p}(\mathbf{y},\mathbf{d})=e^{ik_p\mathbf{h}\cdot\mathbf{d}}\mathbf{u}^{sc}_{\Om,p}(\mathbf{y-h},\mathbf{d}),\quad \mathbf{y}\in\R^2\ba\ov{\Om_{\bf h}},\mathbf{d}\in\mathbb{S}.
\enn
Consequently,
\be\label{Tu=Tu}
\mathbb{T}_{\mathbf{\nu(y)}}\mathbf{u}^{sc}_{\Om_{\mathbf{h}},p}\mathbf{(y,d)}
=e^{ik_p\mathbf{h}\cdot\mathbf{d}}\mathbb{T}_{\mathbf{\nu(y)}}\mathbf{u}^{sc}_{\Om,p}\mathbf{(y-h,d)},\quad\mathbf{y}\in\pa\Om_{\bf h},\mathbf{d}\in\mathbb{S}.
\en
With the representations \eqref{upinfty2}, \eqref{u=u}-\eqref{Tu=Tu}, we have
\be\label{uinf=uinfpp}
u^\infty_{\Om_{\mathbf{h}},pp}\mathbf{(\hat{x},d)}
&=& \int_{\pa\Om_{\mathbf{h}}}\Big\{-\mathbf{\hat{x}}e^{-ik_p\mathbf{y\cdot \hat{x}}}\cdot\mathbb{T}_{\mathbf{\nu(y)}}\mathbf{u}^{sc}_{\Om_{\mathbf{h}},p}\mathbf{(y,d)}\cr
&&\qquad\qquad\qquad\quad +[\mathbb{T}_{\mathbf{\nu(y)}} \mathbf{\hat{x}}e^{-ik_p\mathbf{y\cdot \hat{x}}}]\cdot\mathbf{u}^{sc}_{\Om_{\mathbf{h}},p}\mathbf{(y,d)}\Big\}ds(\mathbf{y})\cr
&=& e^{ik_p\mathbf{h\cdot (d-\hat{x})}}\int_{\pa\Om}\Big\{-\mathbf{\hat{x}}e^{-ik_p\mathbf{z\cdot \hat{x}}}\cdot\mathbb{T}_{\mathbf{\nu(z)}} \mathbf{u}^{sc}_{\Om,p}\mathbf{(z,d)}\cr
&&\qquad\qquad\qquad\quad +[\mathbb{T}_{\mathbf{\nu(z)}} \mathbf{\hat{x}}e^{-ik_p\mathbf{z\cdot \hat{x}}}]\cdot \mathbf{u}^{sc}_{\Om,p}\mathbf{(z,d)}\Big\}ds(\mathbf{z})\cr
&=&  e^{ik_p\mathbf{h\cdot (d-\hat{x})}} u^\infty_{\Om,pp}\mathbf{(\hat{x},d)}, \quad \hat{\mathbf{x}},\mathbf{d}\in\mathbb{S}.
\en
Similarly, we have
\be
\label{uinf=uinfps}
u^\infty_{\Om_{\mathbf{h}},ps}\mathbf{(\hat{x},d)} &=& e^{-ik_s\mathbf{h\cdot\hat{x}}}e^{ik_p\mathbf{h\cdot d}} u^\infty_{\Om,ps}\mathbf{(\hat{x},d)}, \quad \hat{\mathbf{x}},\mathbf{d}\in\mathbb{S},\\
\label{uinf=uinfsp}
u^\infty_{\Om_{\mathbf{h}},sp}\mathbf{(\hat{x},d)} &=& e^{-ik_p\mathbf{h\cdot\hat{x}}}e^{ik_s\mathbf{h\cdot d}} u^\infty_{\Om,sp}\mathbf{(\hat{x},d)}, \quad \hat{\mathbf{x}},\mathbf{d}\in\mathbb{S},\\
\label{uinf=uinfss}
u^\infty_{\Om_{\mathbf{h}},ss}\mathbf{(\hat{x},d)} &=& e^{ik_s\mathbf{h\cdot (d-\hat{x})}} u^\infty_{\Om,ss}\mathbf{(\hat{x},d)}, \quad \hat{\mathbf{x}},\mathbf{d}\in\mathbb{S}.
\en
Finally, \eqref{Trans-Invariance}  follows by taking modulus on both sides of the above four equalities \eqref{uinf=uinfpp}-\eqref{uinf=uinfss}.
\end{proof}

Recently, a simple phase retrieval technique has been proposed in \cite{JiLiuZhang-obstacle} for acoustic obstacle scattering problems. The basic idea is to add three point like scatterers located at $\mathbf{z}\in\R^2\ba\ov{\Om}$ with different scattering strengths $\tau_j\in\C, j=1,2,3$, into the scattering system. The corresponding far field patterns are given by
\ben
u^{\infty}_j=\tau_je^{ik\mathbf{z\cdot(d-\hat{x})}}, j=1,2,3,
\enn
where $k$ is the wave number.
A key step of the proposed phase retrieval technique is to choose three different far field data $u_j^{\infty}$. This is possible by choosing $\tau_j,j=1,2,3$, such that they are not collinear in the complex plane.
Similarly, for the elastic scattering of plane compressional wave $\mathbf{u}^{in}_{p}$ by point like scatterer located at $\mathbf{z}\in\R^2\ba\ov{\Om}$ with strengths $\tau_j\in\C, j=1,2,3$, the corresponding compressional far field pattern is given by
\ben
u^{\infty}_{pp,j}=\tau_je^{ik_p\mathbf{z\cdot(d-\hat{x})}}(\mathbf{d\cdot\hat{x}}), j=1,2,3.
\enn
Clearly, difficulties arise due to the term $(\mathbf{d\cdot\hat{x}})$. In particular, the three elastic compressional far field patterns $u^{\infty}_{pp,j},j=1,2,3$, vanish if $\mathbf{d\cdot\hat{x}}=0$. In practice, $\mathbf{d\cdot\hat{x}}$ is very small for many cases, which makes the phase retrieval quite unstable.

In this work, instead of adding point like scatterers, we introduce point sources into the scattering system. In other words, we also consider the scattering of point sources.
Denote by $\mathbf{z}\in \R^2\ba\ov{\Om},$ $\tau\in\C$ and $\mathbf{q}\in \mathbb{S}$ the location, scattering strength and polarization of the point source, respectively.
The incident point source $\mathbf{v}^{in}$ is given by
\ben
\mathbf{v}^{in}(\mathbf{x,z,q},\tau):= \tau \mathbf{\Phi(x,z)q},
\enn
where $\mathbf{\Phi}$ is the Green's tensor \cite{Arens2001} of the Navier equation in $\R^2$ given by
\ben
\mathbf{\Phi(x,y)}
&:=&\frac{i}{4\mu}H_0^{(1)}(k_s\vert \mathbf{x-y} \vert)\mathbb{I}\cr
&&+\frac{i}{4\om^2}\grad_{\mathbf{x}}\grad_{\mathbf{x}}^{T}(H_0^{(1)}(k_s\vert\mathbf{x-y} \vert)-H_0^{(1)}(k_p\vert\mathbf{x-y}\vert)), \quad \mathbf{x\neq y}
\enn
in terms of the identity matrix $\mathbb{I}$ and the Hankel function $H^{(1)}_0$ of the first kind of order zero.
Note that $\mathbf{\Phi(x,y)q}$ is a radiating elastic scattered field and the corresponding far field patterns are given by \cite{JiLiuXi}
\be
\label{Phip}\mathbf{\Phi}^\infty_{p}(\mathbf{\hx, y, q})&=&  e^{-ik_p\mathbf{\hx\cdot y}}(\mathbf{q}\cdot\mathbf{\hx}),\quad\mathbf{\hx,q}\in \mathbb{S},\, \mathbf{y}\in\R^2,\\
\label{Phis}\mathbf{\Phi}^\infty_{s}(\mathbf{\hx, y, q})&=&  e^{-ik_s\mathbf{\hx\cdot y}}(\mathbf{q}\cdot\mathbf{\hx}^{\perp}),\quad \mathbf{\hx,q}\in \mathbb{S},\,\mathbf{y}\in\R^2.
\en
Denote by $\mathbf{v}^{sc}_{\Om}(\mathbf{x,z,q},\tau)$ and $\mathbf{v}^{\infty}_{\Om}(\mathbf{\hx,z,q},\tau)$ the scattered field and far field pattern, respectively, due to the scattering of point source $\mathbf{v}^{in}(\mathbf{x,z,q},\tau)$ by the rigid body $\Om$.

For all $\mathbf{x}\in\R^2\ba\ov{\Om\cup\{\mathbf{z}\}},\,\mathbf{d,q}\in\mathbb{S},\,\tau\in\C$, define
\ben
\mathbf{w}^{sc}_{\Om\cup\{\mathbf{z}\}}(\mathbf{x,d,q},\tau):=\mathbf{u}^{sc}_{\Om}(\mathbf{x,d})+\mathbf{v}^{sc}_{\Om}(\mathbf{x,z,q},\tau)+\mathbf{v}^{in}(\mathbf{x,z,q},\tau).
\enn
Then $\mathbf{w}^{sc}_{\Om\cup\{\mathbf{z}\}}$ is a radiating solution to
\be\label{usOMz0}
\Delta^{\ast} \mathbf{w}^{sc}_{\Om\cup\{\bf z\}}+\omega^2 \mathbf{w}^{sc}_{\Om\cup\{\mathbf{z}\}}=-\tau \delta_{\mathbf{z}}\mathbf{q}\quad \mbox{in}\,\,\R^2\ba\ov{\Om\cup\{\mathbf{z}\}},
\en
where $\delta_{\mathbf{z}}$ is the Dirac delta function at the point $\mathbf{z}$, and
\be
\label{Dirichlet2}
\mathbf{w}^{sc}_{\Om\cup\{\mathbf{z}\}} = - \mathbf{u}^{in} \quad\mbox{on}\,\,\pa \Om.
\en
%We denote the corresponding elastic far field pattern by
%\ben
%\Big[w^{\infty}_{\Om\cup\{\mathbf{z}\},p}(\mathbf{\hat{x},d,q},\tau);\, w^{\infty}_{\Om\cup\{\mathbf{z}\},s}(\mathbf{\hat{x},d,q},\tau)\Big],\quad \mathbf{\hat{x}, d,q}\in\mathbb{S},\tau\in\C.
%\enn
Then, by linearity, for all $\mathbf{\hat{x}, d,q}\in\mathbb{S},\tau\in\C$, denote respectively by
\be\label{wmpdef}
w^{\infty}_{\Om\cup\{\mathbf{z}\},mp}(\mathbf{\hat{x},d,q},\tau)=u^{\infty}_{\Om,mp}(\mathbf{\hat{x},d})+v^{\infty}_{\Om,p}(\mathbf{\hx,z,q},\tau)+\tau e^{-ik_p\mathbf{\hx\cdot z}}(\mathbf{q}\cdot\mathbf{\hx})
\en
the compressional far field pattern corresponding to the incident field $\mathbf{u}^{in}_{m}, m=p,s$,
and
\be\label{wspdef}
w^{\infty}_{\Om\cup\{\mathbf{z}\},ms}(\mathbf{\hat{x},d,q},\tau)=u^{\infty}_{\Om,ms}(\mathbf{\hat{x},d})+v^{\infty}_{\Om,s}(\mathbf{\hx,z,q},\tau)+\tau e^{-ik_s\mathbf{\hx\cdot z}}(\mathbf{q}\cdot\mathbf{\hx}^{\perp})
\en
the shear far field pattern corresponding to $\mathbf{u}^{in}_{m}, m=p,s$.

The following Theorem \ref{weakinteraction} implies that the far field pattern $\mathbf{v}^{\infty}_{\Om}$ corresponding to the point source $\mathbf{v}^{in}$ is quite small
if the distance $\rho:=dist (z, \Om)$ is large enough. To prove this, we recall the elastic single-layer boundary operator $\mathbf{S}: [H^{-1/2}(\pa\Om)]^2 \rightarrow [H^{1/2}(\pa\Om)]^2$, given by
\be\label{singlelayerpotential}
(\mathbf{S}\mathbf{\bbphi})(\mathbf{x}):=\int_{\pa\Om}\mathbf{\Phi(x,y)\bbphi(y)}ds(\mathbf{y}),\quad \mathbf{x}\in\pa\Om
\en
and the elastic double-layer boundary operator $\mathbf{K}: [H^{1/2}(\pa\Om)]^2 \rightarrow [H^{1/2}(\pa\Om)]^2 $ by
\ben
\mathbf{(K\bbphi)(x)}:=\int_{\pa\Om}(\mathbb{T}_{\mathbf{\nu(y)}}\mathbf{\Phi(x,y)})^{T}\bbphi(\mathbf{y})ds(\mathbf{y}),\quad \mathbf{x}\in \pa\Om.
\enn
The mapping properties of $\mathbf{S}$ and $\mathbf{K}$ are intensively studied in \cite{HahnerHsiao,Mclean}.\\

\begin{theorem}\label{weakinteraction}
Let ${\bf z}$ be a point outside $\Om$ such that the distance $\rho:=dist (\mathbf{z}, \Om)$ is large enough. Then we have
\ben
v^{\infty}_{\Om,m}(\mathbf{\hx,z,q},\tau) = O\left(\frac{1}{\sqrt{\rho}}\right),\quad \rho\rightarrow\infty,\,\mathbf{\hat{x}, q}\in\mathbb{S}, \tau\in\C, \,m=p,s.
\enn
\end{theorem}
\begin{proof}
We seek a solution in the form
\be\label{vs}
\mathbf{v}^{sc}_{\Om}(\mathbf{x,z,q},\tau)
=\int_{\pa\Om}[(\mathbb{T}_{\mathbf{\nu(y)}}\mathbf{\Phi(x,y)})^{T}+i\mathbf{\Phi(x,y)}]\bbphi_{\Om}(\mathbf{y,z,q},\tau)ds(\mathbf{y})\qquad
\en
for all $\mathbf{x}\in\R^2\ba\ov{\Om\cup\{\mathbf{z}\}}, \mathbf{q}\in\mathbb{S}, \tau\in\C$
with a density $\bbphi_{\Om}\in [H^{1/2}(\pa\Om)]^2$.
From the jump relation of the double layer potential \cite{HahnerHsiao,Mclean}, we see that the representation $\mathbf{v}^{sc}_{\Om}$
given in \eqref{vs} solves the exterior Dirichlet boundary problem provided the density is a solution of
the integral equation
\ben
\Big(I/2+\mathbf{K}+i \mathbf{S}\Big)\bbphi_{\Om} = -\mathbf{v}^{in}=-\tau\mathbf{\Phi(\cdot,z)q}\quad \mbox{on}\,\,\pa\Om.
\enn
Note that $I/2+\mathbf{K}+i \mathbf{S}$ is bijective and the inverse $(I/2+\mathbf{K}+i \mathbf{S})^{-1}: [H^{1/2}(\pa\Om)]^2\rightarrow [H^{1/2}(\pa\Om)]^2$ is bounded \cite{HahnerHsiao}. Therefore
\be\label{ppphi}
\bbphi_{\Om} = -(I/2+\mathbf{K}+i \mathbf{S})^{-1}(\tau\mathbf{\Phi(\cdot,z)q}).
\en
A straightforward calculation shows that the Green tensor $\mathbf{\Phi}$ satisfies Sommerfeld's finiteness condition
\ben
\mathbf{\Phi(x,y)q}=O\left(|x-y|^{-\frac{1}{2}}\right), \quad |x-y|\rightarrow\infty.
\enn
This implies that
\ben
\bbphi_{\Om}(\mathbf{y,z,q},\tau) =  O\left(\frac{1}{\sqrt{\rho}}\right), \quad \rho\rightarrow\infty
\enn
for all $\mathbf{y}\in\pa\Om, \mathbf{q}\in\mathbb{S}$. Inserting this into \eqref{vs}, we find that
\ben
\mathbf{v}^{sc}_{\Om}(\mathbf{x,z,q},\tau) =  O\left(\frac{1}{\sqrt{\rho}}\right), \quad \rho\rightarrow\infty
\enn
for all $\mathbf{x}\in\R^2\ba\ov{\Om\cup\{\mathbf{z}\}}, \mathbf{q}\in\mathbb{S}, \tau\in\C$.
The proof is completed by letting $|\mathbf{x}|\rightarrow \infty$ in \eqref{vs} and using \eqref{Phip}-\eqref{Phis}.
\end{proof}

With the previous analysis, we are now well prepared for studying the following elastic inverse obstacle scattering problem.

{\bf Phaseless-IP1}: {\em Determine the location and shape of the scatterer $\Om$ from the following phaseless far field data
\ben
\Big| w^{\infty}_{\Om\cup\{\mathbf{z}\},mn}(\mathbf{\hat{x},d,q},\tau) \Big|, \quad \mathbf{\hat{x}, d,q}\in\mathbb{S},\, \tau\in\C,\, m,n=p, s.
\enn
}

\subsection{Scattering of sources} \label{ESP-F}
In this subsection, we study the scattering of sources with compact support. Denote by $\mathbf{F}=\mathbf{F(x)}\in [L^2(\R^2)]^2$ the external force with compact support $\Om$.
For the source scattering problem, we assume that
\ben
\om\in \mathbb{W}:=(\om_{min}, \om_{max}),
\enn
where $\om_{min}<\om_{max}$ are two fixed circular frequencies.
For the scattering of a source $\mathbf{F}$ we denote the scattered field by $\mathbf{u}^{sc}_{\mathbf{F}}(\mathbf{x},\om)$, and the corresponding far field pattern by
${\bf u}^{\infty}_{\mathbf{F}}(\mathbf{\hx},\om)=[u^{\infty}_{\mathbf{F},p}(\mathbf{\hx},\om),\,u^{\infty}_{\mathbf{F},s}(\mathbf{\hx},\om)]$.

The propagation of time-harmonic elastic wave in an isotropic homogeneous media is governed by the reduced Navier equation
\be\label{elastic-F}
\Delta^{\ast} \mathbf{u}^{sc}_{\mathbf{F}}+\omega^2 \mathbf{u}^{sc}_{\mathbf{F}}=\mathbf{F}\quad \mbox{in}\,\,\R^2.
\en
Again, the scattered field $\mathbf{u}^{sc}_{\mathbf{F}}$ admits the decomposition
\ben
\mathbf{u}^{sc}_{\mathbf{F}}=\mathbf{u}^{sc}_{\mathbf{F},p}+\mathbf{u}^{sc}_{\mathbf{F},s}:=-\frac{1}{k^2_p}\ggrad\ddiv\mathbf{u}^{sc}_{\mathbf{F}}-\frac{1}{k^2_s}\ggrad^{\perp}\ddiv^{\perp}\mathbf{u}^{sc}_{\mathbf{F}}\quad
\mbox{in}\, \R^2\ba\ov{D}
\enn
and satisfies the Kupradze's radiation conditions
\be\label{KupradzeRC-F}
\frac{\partial \mathbf{u}_{\mathbf{F},p}^{sc}}{\partial r}-ik_p\mathbf{u}_{\mathbf{F},p}^{sc}=o(r^{-\frac{1}{2}}),\qquad
\frac{\partial \mathbf{u}_{\mathbf{F},s}^{sc}}{\partial r}-ik_s\mathbf{u}_{\mathbf{F},s}^{sc}=o(r^{-\frac{1}{2}}),
\en
uniformly in all directions $\hat{\mathbf{x}}\in \mathbb{S}$ as $r:=|\mathbf{x}|\rightarrow\infty$.

The radiating solution $\mathbf{u}^{sc}_{\mathbf{F}}$ to the scattering problem \eqref{elastic-F}-\eqref{KupradzeRC-F} takes the form
\be\label{us}
\mathbf{u}^{sc}_{\mathbf{F}}(\mathbf{x},\om)=\int_{\R^2}\mathbf{\Phi(x,y)F(y)}d\mathbf{y},\quad \mathbf{x}\in\R^2,\,\om\in \mathbb{W}.
\en
By \eqref{Phip}-\eqref{Phis}, the corresponding compressional and shear far field patterns of $\mathbf{u}^{sc}_{\mathbf{F}}$ are given by
\be\label{uinfFp}
u^{\infty}_{\mathbf{F},p}(\mathbf{\hx},\om)=\int_{\R^2}e^{-ik_p\mathbf{\hx\cdot y}}\mathbf{\hx\cdot F(y)}d\mathbf{y},\quad \mathbf{\hx}\in\mathbb{S},\,\om\in \mathbb{W}
\en
and
\be\label{uinfFs}
u^{\infty}_{\mathbf{F},s}(\mathbf{\hx},\om)=\int_{\R^2}e^{-ik_s\mathbf{\hx\cdot y}}\mathbf{\hx^{\perp}\cdot F(y)}d\mathbf{y},\quad \mathbf{\hx}\in\mathbb{S},\,\om\in \mathbb{W},
\en
respectively.

The analogous result of Theorem \ref{TI-O} is formulated in the following theorem.
\begin{theorem}\label{TI-F}
Let $\mathbf{F_{h}(y)}:=\mathbf{F(y+h)}, \mathbf{y}\in \R^2$ be the shifted source with a fixed vector $\mathbf{h}\in \R^2$.
For the scattering of the source $\mathbf{F_{h}}$, we denote the scattered field by $\mathbf{u}^{sc}_{\mathbf{F_h}}$ and its far field pattern by $(u^{\infty}_{\mathbf{F_h},p}, u^{\infty}_{\mathbf{F_h},s})$. Then we have the following translation invariance property
\be\label{Trans-Invariance-F}
|u^{\infty}_{\mathbf{F_{h}},m}(\hat{\mathbf{x}},\om)| = |u^{\infty}_{\mathbf{F},m}(\hat{\mathbf{x}},\om)|, \quad \hat{\mathbf{x}}\in\mathbb{S},\,\om\in \mathbb{W},\,m=p,s.
\en
\end{theorem}
\begin{proof}
By the far field representations \eqref{uinfFp}-\eqref{uinfFs}, we find
\be\label{uFp}
u^{\infty}_{\mathbf{F_{h}},p}(\hat{\mathbf{x}},\om)=e^{ik_p\mathbf{h\cdot\hx}}u^{\infty}_{\mathbf{F},p}(\hat{\mathbf{x}},\om), \quad \hat{\mathbf{x}}\in\mathbb{S},\,\om\in \mathbb{W}
\en
and
\be\label{uFs}
u^{\infty}_{\mathbf{F_{h}},s}(\hat{\mathbf{x}},\om)=e^{ik_s\mathbf{h\cdot\hx}}u^{\infty}_{\mathbf{F},s}(\hat{\mathbf{x}},\om), \quad \hat{\mathbf{x}}\in\mathbb{S},\,\om\in \mathbb{W}.
\en
Then the statement follows by taking the modulus on both sides of \eqref{uFp}-\eqref{uFs}.
\end{proof}

Theorem \ref{TI-F} implies that the phaseless far field data are invariant under the translation of the source. Thus the location of the source can be not uniquely determined.
Clearly, from \eqref{uinfFp}-\eqref{uinfFs}, we have
\ben
|u^{\infty}_{c\mathbf{F},m}(\hat{\mathbf{x}},\om)|=|u^{\infty}_{\mathbf{F},m}(\hat{\mathbf{x}},\om)|, \quad \hat{\mathbf{x}}\in\mathbb{S},\,\om\in \mathbb{W},\,m=p,s
\enn
for any constant $c\in\C$ with $|c|=1$. Therefore, the source can not be uniquely determined from the phaseless far field data $|u^{\infty}_{\mathbf{F},p}(\hat{\mathbf{x}},\om)|$
and $|u^{\infty}_{\mathbf{F},s}(\hat{\mathbf{x}},\om)|, \quad \hat{\mathbf{x}}\in\mathbb{S},\,\om\in \mathbb{W}$, even the location $\Om$ of the source is known in advance.

To solve the above mentioned problems, we introduce again point sources into the scattering system. Let $\mathbf{z}\in \R^2\ba\ov{\Om},$ $\tau\in\C$ and $\mathbf{q}\in \mathbb{S}$ be the location, scattering strength and polarization of the point source, respectively.
%Different to the obstacle scattering problems in the previous subsection, the elastic source scattering problem is linear with respect to the source.
For the scattering of combined sources, the scattered field $\mathbf{u}^{sc}_{\mathbf{F\cup\{z}\}}:=\mathbf{u}^{sc}_{\mathbf{F\cup\{z}\}}(\mathbf{x,q},\om,\tau)$ is given by
\be\label{usFz0}
\mathbf{u}^{sc}_{\mathbf{F}\cup\{\mathbf{z}\}}(\mathbf{x,q},\om,\tau)=
\mathbf{u}^{sc}_{\mathbf{F}}(\mathbf{x},\om)+\tau \mathbf{\Phi(x,z)}\mathbf{q},\quad \mathbf{x}\in\R^2, \mathbf{q}\in\mathbb{S}, \,\om\in \mathbb{W},\,\tau\in\C\qquad
\en
and the corresponding compressional and shear far field patterns are given by
\be\label{uinfFpz0}
u^{\infty}_{\mathbf{F}\cup\{\mathbf{z}\},p}(\mathbf{\hx,q},\om,\tau)=
u^{\infty}_{\mathbf{F},p}(\mathbf{\hx},\om)+\tau e^{-ik_p\mathbf{\hx\cdot z}}\mathbf{q\cdot\hx},\quad \mathbf{\hx, q}\in\mathbb{S}, \,\om\in \mathbb{W},\,\tau\in\C \qquad
\en
and
\be\label{uinfFsz0}
u^{\infty}_{\mathbf{F}\cup\{\mathbf{z}\},s}(\mathbf{\hx,q},\om,\tau)=
u^{\infty}_{\mathbf{F},s}(\mathbf{\hx},\om)+\tau e^{-ik_s\mathbf{\hx\cdot z}}\mathbf{q\cdot\hx^{\perp}},\quad \mathbf{\hx, q}\in\mathbb{S}, \,\om\in \mathbb{W},\,\tau\in\C, \qquad
\en
respectively.

Finally, the corresponding inverse source problem of our interest is as follows.

{\bf Phaseless-IP2}: {\em Determine the source $\mathbf{F}$ from the following phaseless far field data
\ben
\Big| u^{\infty}_{\mathbf{F}\cup\{\mathbf{z}\},m}(\mathbf{\hat{x},q},\om,\tau) \Big|, \quad \mathbf{\hat{x}, q}\in\mathbb{S},\,\om\in\mathbb{W},\,\tau\in\C,\, m=p, s.
\enn
}

\section{Uniqueness}\label{sec:uniqueness}
Using the results of the previous section, this section is devoted to study uniqueness for the inverse elastic scattering problems with phaseless far field data.

\subsection{Uniqueness for sources}
Firstly, we establish the uniqueness results for {\bf Phaseless-IP2}, i.e., the inverse source scattering problems.
Let $B_R$ be a disk with radius $R$ large enough such that $\Om$ is contained inside. In a recent paper \cite{BaoLiZhao}, the authors show that the source $\mathbf{F}$ can be uniquely determined by the scattered field $\mathbf{u}^{sc}_{\bf F}(\mathbf{x},\om), \mathbf{x}\in\pa B_R, \om\in \mathbb{W}$. By Rellich's lemma \cite{CK}, the compressional scattered field $\mathbf{u}^{sc}_{\mathbf{F},p}$ and the shear scattered field $\mathbf{u}^{sc}_{\mathbf{F},s}$ are uniquely determined by the corresponding compressional far field $u^{\infty}_{\mathbf{F},p}$ and shear far field $u^{\infty}_{\mathbf{F},s}$, respectively. Therefore, we immediately deduce the following uniqueness result with phased far field data.\\

\begin{theorem}\label{uni-full}
The source $\mathbf{F}$ is uniquely determined by the multi-frequency far field data sets
\be\label{Ap}
\mathcal {A}_{p}:=\big\{u^{\infty}_{\mathbf{F},p}(\mathbf{\hx}, \om):\,\mathbf{\hx}\in \mathbb{S},\,\om\in \mathbb{W}\big\}
\en
and
\be\label{As}
\mathcal {A}_{s}:=\big\{u^{\infty}_{\mathbf{F},s}(\mathbf{\hx}, \om):\,\mathbf{\hx}\in \mathbb{S},\,\om\in \mathbb{W}\big\}.
\en
\end{theorem}

We want to remark that by analyticity the far field pattern can be determined on the whole unit circle $\mathbb{S}$ using partial value on any circular arc of $\mathbb{S}$. Thus the previous theorem also holds if the far field data are given on any arc. In the sequel we set
\be\label{Sq}
\mathbb{S}_{\mathbf{q}}:=\{\mathbf{\hx}\in\mathbb{S}\,|\,\mathbf{q\cdot\hx}\geq 1/2\}\quad\mbox{and}\quad \mathbb{S}^{\perp}_{\mathbf{q}}:=\{\mathbf{\hx}\in\mathbb{S}\,|\,\mathbf{q\cdot\hx^{\perp}}\geq 1/2\}
\en
for some fixed $\mathbf{q}\in \mathbb{S}$.
Define
\ben
\mathcal{Q}:=\{\mathbf{q}_1,\mathbf{q}_2,\mathbf{q}_3\},
\enn
where
\ben
\mathbf{q}_1:=\left(\cos\frac{\pi}{4}, \sin\frac{\pi}{4}\right)^{T},\quad \mathbf{q}_2:=\left(\cos\frac{11\pi}{12}, \sin\frac{11\pi}{12}\right)^{T},\quad \mathbf{q}_3:=\left(\cos\frac{19\pi}{12}, \sin\frac{19\pi}{12}\right)^{T}.
\enn
It follows that
\ben
\mathbb{S}=\bigcup_{\mathbf{q}\in\mathcal{Q}}\mathbb{S}_{\mathbf{q}}=\bigcup_{\mathbf{q}\in\mathcal{Q}}\mathbb{S}^{\perp}_{\mathbf{q}}.
\enn\\

\begin{theorem}\label{uni-phaseless012}
Define $\tau_1:=|\tau_1|e^{i\alpha}\in\C\ba\{0\}$ with the principal argument $\alpha\in [0,2\pi)$, and let $\mathbf{z_0}$ and $\mathbf{z_1}$ be two different points in $\R^2\ba\ov{\Om}$.
For fixed polarization $\mathbf{q}\in \mathbb{S}$ of the point sources, define two circular arcs $\mathbb{S}_{\mathbf{q}}$ and $\mathbb{S}^{\perp}_{\mathbf{q}}$ as in \eqref{Sq}.
Then the source $\mathbf{F}$ is uniquely determined by the two phaseless far field data sets
\be\label{Bp}
\mathcal {B}_{p}:=\big\{|u^{\infty}_{\mathbf{F\cup\{z\}},p}(\mathbf{\hx, q}, \om, \tau)|:\,\mathbf{\hx}\in \mathbb{S}_{\mathbf{q}},\,\om\in \mathbb{W},\, \tau\in \mathcal \{0, \tau_1\}, \mathbf{z}\in \{\mathbf{z_0,z_1}\}\big\}
\en
and
\be\label{Bs}
\mathcal {B}_{s}:=\big\{|u^{\infty}_{\mathbf{F\cup\{z\}},s}(\mathbf{\hx, q}, \om, \tau)|:\,\mathbf{\hx}\in \mathbb{S}^{\perp}_{\mathbf{q}},\,\om\in \mathbb{W},\, \tau\in \mathcal \{0, \tau_1\}, \mathbf{z}\in \{\mathbf{z_0,z_1}\}\big\}.
\en
\end{theorem}
\begin{proof}
We first show that the compressional far field data set $\mathcal{A}_p$ given in \eqref{Ap} is uniquely determined by the phaseless far field data set $\mathcal{B}_p$.
To accomplish this, we observe that from the representation \eqref{uinfFpz0} it follows that
\ben
&&|u^{\infty}_{\mathbf{F\cup\{z\}}, p}(\mathbf{\hx, q}, \om, \tau_1)|^2\cr
&=&|u^{\infty}_{\mathbf{F}, p}(\mathbf{\hx}, \om)+\tau_1 e^{-ik_p\mathbf{\hx\cdot z}}\mathbf{q\cdot\hx}|^2\cr
&=&|u^{\infty}_{\mathbf{F}, p}(\mathbf{\hx}, \om)|^2+2\mathbf{q\cdot\hx}\Re\big(u^{\infty}_{\mathbf{F}, p}(\mathbf{\hx}, \om)\ov{\tau_1 e^{-ik_p\mathbf{\hx\cdot z}}}\big)+|\tau_1|^2,\quad \mathbf{\hx}\in\mathbb{S}_{\mathbf{q}},\,\om\in \mathbb{W},\,\mathbf{z\in \{z_0,z_1\}}.
\enn
From this, since $\mathbf{q\cdot\hx}\geq1/2$ for $\mathbf{\hx}\in\mathbb{S}_{\mathbf{q}}$, we deduce that $\Re\big(u^{\infty}_{\mathbf{F}, p}(\mathbf{\hx}, \om)\ov{\tau_1 e^{-ik_p\mathbf{\hx\cdot z}}}\big)$ can be uniquely determined for all $\mathbf{\hx}\in\mathbb{S}_{\mathbf{q}},\,\om\in \mathbb{W},\,\mathbf{z\in \{z_0,z_1\}}$. We rewrite $u^{\infty}_{\mathbf{F}, p}(\mathbf{\hx}, \om)$ and $\tau_1 e^{-ik_p\mathbf{\hx\cdot z}}$ in the form
\ben
u^{\infty}_{\mathbf{F}, p}(\mathbf{\hx}, \om)=|u^{\infty}_{\mathbf{F}, p}(\mathbf{\hx}, \om)|e^{i\phi(\mathbf{\hx},\om)} \quad\mbox{and}\quad \tau_1 e^{-ik_p\mathbf{\hx\cdot z}}=|\tau_1|e^{i(\alpha-k_p\mathbf{\hx\cdot z})},
\enn
respectively, for all $\mathbf{\hx}\in\mathbb{S}_{\mathbf{q}},\,\om\in \mathbb{W},\,\mathbf{z\in \{z_0,z_1\}}$.
Here, $\phi(\mathbf{\hx},\om)\in [0,2\pi)$ is the principal argument of $u^{\infty}_{\mathbf{F}, p}(\mathbf{\hx}, \om)$.
We claim that the principal argument $\phi(\mathbf{\hx},\om)$ is uniquely determined.
Indeed, assume that there are two arguments $\phi_1(\mathbf{\hx},\om)$ and $\phi_2(\mathbf{\hx},\om)$.
Define
\ben
\mathbb{S}_{\mathbf{q},0}({\bf z_0,z_1},\om):=\{{\bf\hx}\in \mathbb{S}_{\mathbf{q}}:\, u^{\infty}_{\mathbf{F}}(\mathbf{\hx}, \om)\neq 0\quad\,\mbox{and}\quad\, \mathbf{\hx\cdot(z_0-z_1)}\neq 0\}.
\enn
It is easy to verify that $\mathbb{S_\mathbf{q}}\ba\mathbb{S}_{\mathbf{q},0}({\bf z_0,z_1},\om)$ has Lebegue meassure zero.
Since $|u^{\infty}_{\mathbf{F}, p}(\mathbf{\hx}, \om)|$ is given in $\mathcal {B}_p$ and $\Re\big(u^{\infty}_{\mathbf{F}, p}(\mathbf{\hx}, \om)\ov{\tau_1 e^{-ik_p\mathbf{\hx\cdot z}}}\big)$ can be uniquely determined for all $\mathbf{\hx}\in\mathbb{S}_{\mathbf{q}},\,\om\in \mathbb{W},\,\mathbf{z\in \{z_0,z_1\}}$,
we have that for all $\mathbf{\hx}\in\mathbb{S}_{\mathbf{q},0}({\bf z_0,z_1},\om),\,\om\in \mathbb{W},\,\mathbf{z\in \{z_0,z_1\}}$,
\ben
\cos[\phi_1(\mathbf{\hx},\om)-(\alpha-k_p\mathbf{\hx\cdot z})]=\cos[\phi_2(\mathbf{\hx},\om)-(\alpha-k_p\mathbf{\hx\cdot z})],
\enn
and furthermore, we conclude that
\be\label{case1}
\phi_1(\mathbf{\hx},\om)-(\alpha-k_p\mathbf{\hx\cdot z})=\phi_2(\mathbf{\hx},\om)-(\alpha-k_p\mathbf{\hx\cdot z})+2l\pi,\quad \mbox{for\,\,some}\, l\in\Z,\quad
\en
or
\be\label{case2}
\phi_1(\mathbf{\hx},\om)-(\alpha-k_p\mathbf{\hx\cdot z})=-[\phi_2(\mathbf{\hx},\om)-(\alpha-k_p\mathbf{\hx\cdot z})]+2l'\pi,\ \mbox{for\,\,some}\, l'\in\Z.\quad
\en

We now show that the case \eqref{case2} does not hold. Actually, \eqref{case2} implies that
\be\label{phiplus}
\phi_1(\mathbf{\hx},\om)+\phi_2(\mathbf{\hx},\om)-2l'\pi=2[\alpha-k_p\mathbf{\hx\cdot z}],\quad \mathbf{z\in \{z_0, z_1\}}, \,\mbox{for\,\,some}\,l'\in\Z.
\en
The left hand side of \eqref{phiplus} is independent of ${\bf z}$. However, the right hand side of \eqref{phiplus} changes for different $\mathbf{z\in \{z_0, z_1\}}$. This leads to a contradiction, and thus \eqref{case2} does not hold.

For the case when \eqref{case1} holds, we have
\ben
\phi_1(\mathbf{\hx},\om)-\phi_2(\mathbf{\hx},\om)=2l\pi,\quad \mbox{for\,\,some}\,l\in\Z.
\enn
Noting that $\phi_1(\mathbf{\hx},\om),\phi_2(\mathbf{\hx},\om)\in [0,2\pi)$, we have $\phi_1(\mathbf{\hx},\om)-\phi_2(\mathbf{\hx},\om)\in (-2\pi,2\pi)$, and thus $l=0$, i.e.,
\ben
\phi_1(\mathbf{\hx},\om)=\phi_2(\mathbf{\hx},\om), \quad\forall\,\mathbf{\hx}\in\mathbb{S}_{\mathbf{q},0}({\bf z_0,z_1},\om),\,\om\in \mathbb{W}.
\enn
This further implies that $u^{\infty}_{\mathbf{F}, p}(\mathbf{\hx}, \om)$ is uniquely determined for all $\mathbf{\hx}\in\mathbb{S}_{\mathbf{q},0}({\bf z_0,z_1},\om),\,\om\in \mathbb{W}$ and also for $\mathbf{\hx}\in\mathbb{S},\,\om\in \mathbb{W}$ by analytic continuation.

Similarly, the shear far field data set $\mathcal{A}_s$ given in \eqref{As} is uniquely determined by the set $\mathcal{B}_s$.
The conclusion of the theorem now follows from Theorem \ref{uni-full}.
\end{proof}\\

Theorem \ref{uni-phaseless01} below shows that, under further assumptions on the source, uniqueness can even be established with the help of a single artificial point source.\\

\begin{theorem}\label{uni-phaseless01}
Let $\mathbf{z}\in \R^2\ba\ov{\Om},$ $\tau_1\in\R\ba\{0\}$ and $\mathbf{q}\in \mathbb{S}$ be the location, scattering strength and polarization of the artificial point source, respectively. Define again two circular arcs $\mathbb{S}_{\mathbf{q}}$ and $\mathbb{S}^{\perp}_{\mathbf{q}}$ as in \eqref{Sq}.
Assume further that
\be\label{Rdiv>0}
\Re(\ddiv\mathbf{F(y)})\geq0 \quad\mbox{and}\quad \exists\, \mathbf{y_0}\in\R^2\,\, s.t.\,\, \Re(\ddiv\mathbf{F(y_0)})>0.
\en
Then the source $\mathbf{F}$ is uniquely determined by the phaseless data set
\ben
\mathcal {C}_{p}:=\big\{|u^{\infty}_{\mathbf{F\cup\{z\}},p}(\mathbf{\hx, q}, \om, \tau)|:\,\mathbf{\hx}\in \mathbb{S}_{\mathbf{q}},\,\om\in \mathbb{W},\, \tau\in \mathcal \{0, \tau_1\}\big\}
\enn
and
\ben
\mathcal {C}_{s}:=\big\{|u^{\infty}_{\mathbf{F\cup\{z\}},s}(\mathbf{\hx, q}, \om, \tau)|:\,\mathbf{\hx}\in \mathbb{S}^{\perp}_{\mathbf{q}},\,\om\in \mathbb{W},\, \tau\in \mathcal \{0, \tau_1\}\big\}.
\enn
\end{theorem}
\begin{proof}
Following the arguments in Theorem \ref{uni-phaseless012}, noting that $\tau_1\in\R\ba\{0\}$, we deduce that
\ben
\Re\big(u^{\infty}_{\mathbf{F},p}(\mathbf{\hx}, \om)e^{ik_p\mathbf{\hx\cdot z}}\big),\quad \hx \in \mathbb{S}_{\mathbf{q}},\,\om\in\mathbb{W}
\enn
is uniquely determined by the data set $\mathcal{C}_p$.
For the fixed point $\mathbf{z}\in \R^{2}\ba\ov{\Om}$, define $\mathbf{F_{z}(y)}:=\mathbf{F(y+z)}, \mathbf{y}\in \R^2$. Then by the translation relation \eqref{uFp}, we have
\ben
u^{\infty}_{\mathbf{F_{z}},p}(\hat{\mathbf{x}},\om)=e^{ik_p\mathbf{\hx\cdot z}}u^{\infty}_{\mathbf{F},p}(\hat{\mathbf{x}},\om), \quad \hat{\mathbf{x}}\in\mathbb{S}_{\mathbf{q}},\,\om\in \mathbb{W}.
\enn
Therefore,
\ben
|u^{\infty}_{\mathbf{F_{z}},p}(\hat{\mathbf{x}},\om)|=|u^{\infty}_{\mathbf{F},p}(\hat{\mathbf{x}},\om)|, \quad \hat{\mathbf{x}}\in\mathbb{S}_{\mathbf{q}},\,\om\in \mathbb{W}
\enn
and $\Re\big(u^{\infty}_{\mathbf{F_{z}},p}(\hat{\mathbf{x}},\om)\big)$ is also uniquely determined for all $\hx\in \mathbb{S}_{\mathbf{q}},\,\om\in\mathbb{W}$.
We rewrite $u^{\infty}_{\mathbf{F_{z}},p}(\hat{\mathbf{x}},\om)$ in the form
\be\label{uinfphi}
u^{\infty}_{\mathbf{F_{z}},p}(\hat{\mathbf{x}},\om)=|u^{\infty}_{\mathbf{F},p}(\hat{\mathbf{x}},\om)|e^{i\phi(\mathbf{\hx},\om)},\quad \hat{\mathbf{x}}\in\mathbb{S}_{\mathbf{q}},\,\om\in \mathbb{W}
\en
with some analytic function $\phi$ such that $\phi(\mathbf{\hx},\om)\in [0,2\pi)$. We claim that $\phi$ is uniquely determined.
Assume on the contrary that there are two functions $\phi_1$ and $\phi_2$.
To simplify the notations, for any $\om\in\mathbb{W}$, we define
\ben
\mathbb{S}_{\mathbf{q}}(\om):=\{\mathbf{\hx}\in \mathbb{S}_{\mathbf{q}}: u^{\infty}_{\mathbf{F},p}(\hat{\mathbf{x}},\om)\neq 0\}, \quad\om\in\mathbb{W}.
\enn
Then by analyticity of $u^{\infty}_{\mathbf{F},p}(\hat{\mathbf{x}},\om)$, we deduce that the set $\mathbb{S}_{\mathbf{q}}\ba\mathbb{S}_{\mathbf{q}}(\om)$ has Lebesgue measure zero.
Since the phaseless far field pattern $|u^{\infty}_{\mathbf{F},p}(\hat{\mathbf{x}},\om)|$ is given and $\Re\big(u^{\infty}_{\mathbf{F_{z}},p}(\hat{\mathbf{x}},\om)\big)$ is obtained for all $\mathbf{\hx}\in \mathbb{S}_{\bf q}, \om\in\mathbb{W}$, we have
\be\label{cos1cos2}
\cos[\phi_1(\hat{\mathbf{x}},\om)]=\cos[\phi_2(\hat{\mathbf{x}},\om)], \quad \mathbf{\hx}\in \mathbb{S}_{\mathbf{q}}(\om), \,\om\in\mathbb{W}.
\en
This implies that for any fixed $\mathbf{\hx}\in \mathbb{S}_{\mathbf{q}}(\om), \,\om\in\mathbb{W}$, $\phi_1(\hat{\mathbf{x}},\om)=\phi_2(\hat{\mathbf{x}},\om)$ or $\phi_1(\hat{\mathbf{x}},\om)=2\pi-\phi_2(\hat{\mathbf{x}},\om)$  since $\phi_1(\hat{\mathbf{x}},\om), \phi_2(\hat{\mathbf{x}},\om)\in [0,2\pi)$.

We claim that
\be\label{phi1phi2}
\phi_1(\hat{\mathbf{x}},\om)\equiv\phi_2(\hat{\mathbf{x}},\om), \quad \mathbf{\hx}\in \mathbb{S}_{\mathbf{q}}(\om), \,\om\in\mathbb{W},
\en
or
\be\label{phi1-phi2}
\phi_1(\hat{\mathbf{x}},\om)\equiv2\pi-\phi_2(\hat{\mathbf{x}},\om), \quad \mathbf{\hx}\in \mathbb{S}_{\mathbf{q}}(\om), \,\om\in\mathbb{W}.
\en
For the special case $\phi_1(\hat{\mathbf{x}},\om)\equiv\pi$ for all $\mathbf{\hx}\in \mathbb{S}_{\mathbf{q}}(\om), \,\om\in\mathbb{W}$, we have $\phi_2\equiv\phi_1=\pi$ by \eqref{cos1cos2}. Otherwise, without loss of generality, there exists a direction $\mathbf{\hx}_1\in \mathbb{S}_{\mathbf{q}}(\om_1)$ for some frequency $\om_1\in\mathbb{W}$ such that
$\phi_1(\mathbf{\hx}_1,\om_1)>\pi$. By analyticity of the function $\phi_1$, there exists a neighbourhood $\mathbb{U}(\mathbf{\hx}_1)\subset \mathbb{S}_{\mathbf{q}}(\om_1)$ of $\mathbf{\hx}_1$ such that
\be\label{phi1>pi}
\phi_1(\mathbf{\hx},\om_1)>\pi,\quad\forall \mathbf{\hx}\in \mathbb{U}(\mathbf{\hx}_1).
\en
If $\phi_2(\mathbf{\hx}_1,\om_1)=\phi_1(\mathbf{\hx}_1,\om_1)>\pi$, then $\phi_2(\mathbf{\hx},\om_1)=\phi_1(\mathbf{\hx},\om_1),\quad\forall \mathbf{\hx}\in \mathbb{U}(\mathbf{\hx}_1)$.
Otherwise, there exists a  direction $\mathbf{\hx}_2\in \mathbb{U}(\mathbf{\hx}_1)$ such that $\phi_2(\mathbf{\hx}_2,\om_1)=2\pi-\phi_1(\mathbf{\hx}_2,\om_1)$.
From \eqref{phi1-phi2}-\eqref{phi1>pi}, we have $\phi_2(\mathbf{\hx}_2,\om_1)=2\pi-\phi_1(\mathbf{\hx}_2,\om_1)<\pi$. By analyticity of $\phi_2$ in $\mathbb{U}(\mathbf{\hx}_1)$,
there exists a direction $\mathbf{\hx}_0\in \mathbb{U}(\mathbf{\hx}_1)$ such that $\phi_2(\mathbf{\hx}_0,\om_1)=\pi$. Thus $\phi_1(\mathbf{\hx}_0,\om_1)=\phi_2(\mathbf{\hx}_0,\om_1)=\pi$
by \eqref{cos1cos2}. However, this leads to a contradiction to \eqref{phi1>pi}.

We show that \eqref{phi1-phi2} does not hold. Let $\mathbf{F}_1$ and $\mathbf{F}_2$ be sources corresponding to $\phi_1$ and $\phi_2$, respectively. With the help of unique continuation, from \eqref{uinfphi} and \eqref{phi1-phi2} we have
\ben
u^{\infty}_{\mathbf{F}_1,p}(\mathbf{\hx}, \om)= \ov{u^{\infty}_{\mathbf{F}_2,p}(\mathbf{\hx}, \om)}, \quad \mathbf{\hx}\in \mathbb{S}_{\mathbf{q}}(\om), \,\om\in\mathbb{W}.
\enn
By the representation \eqref{uinfFp}, noting that the source terms $\mathbf{F}_1$ and $\mathbf{F}_2$ are supported in $\Om$ and using Gauss divergence theorem, we have
\be\label{f1f2}
\int_{\R^2}e^{-ik_p\mathbf{\hx\cdot y}}\ddiv\mathbf{F_1(y)}d\mathbf{y}
&=&-\int_{\R^2}\grad_{\mathbf{y}}e^{-ik_p\mathbf{\hx\cdot y}}\cdot\mathbf{F_1(y)}d\mathbf{y}\cr
&=&ik_p\int_{\R^2}e^{-ik_p\mathbf{\hx\cdot y}}\mathbf{\hx\cdot F_1(y)}d\mathbf{y}\cr
&=&ik_p\ov{\int_{\R^2}e^{-ik_p\mathbf{\hx\cdot y}}\mathbf{\hx\cdot F_2(y)}d\mathbf{y}}\cr
&=&\int_{\R^2}\grad_{\mathbf{y}}e^{ik_p\mathbf{\hx\cdot y}}\cdot\ov{\mathbf{F_2(y)}}d\mathbf{y}\cr
&=&-\int_{\R^2}e^{ik_p\mathbf{\hx\cdot y}}\ddiv_{\mathbf{y}}\ov{\mathbf{F_2(y)}}d\mathbf{y}\cr
&=& \int_{\R^2}e^{-ik_p\mathbf{\hx\cdot y}}\ddiv_{\mathbf{y}}\mathbf{\ov{F_2(-y)}}d\mathbf{y}, \quad \mathbf\hx\in \mathbb{S}_{\mathbf{q}}, \om\in\mathbb{W}.
\en
Clearly, both sides are analytic on the frequency $\om$ and the observation direction $\mathbf{\hx}$ and thus \eqref{f1f2} holds for all $\mathbf{\hx}\in \mathbb{S},\, \om\in [0,\infty)$.
By the Fourier integral theorem, we conclude that
\be\label{S1S2}
\ddiv\mathbf{F_1(y)}=\ddiv_{\mathbf{y}}\mathbf{\ov{F_2(-y)}}=-\ddiv_{\mathbf{y^\ast}}\mathbf{\ov{F_2(y^{\ast})}}, \quad \mathbf{y^{\ast}}:=-\mathbf{y},\quad \mathbf{y}\in \R^2.
\en
Our assumption \eqref{Rdiv>0} implies that
\ben
\Re(\ddiv\mathbf{F_1(y_0)}) >0,\quad  \Re(\ddiv_{\mathbf{y^\ast}}\mathbf{\ov{F_2(y^{\ast})}})\geq 0, \,\forall \mathbf{y}\in \R^2.
\enn
This is a contradiction to \eqref{S1S2}. Thus \eqref{phi1-phi2} does not hold, and we deduce that \eqref{phi1phi2}, i.e.,
$\phi_1=\phi_2$. Thus $u^{\infty}_{\mathbf{F_{-z}},p}(\hat{\mathbf{x}},\om)$ and $u^{\infty}_{\mathbf{F},p}(\hat{\mathbf{x}},\om)$ are uniquely determined by the data set $\mathcal{C}_p$ for all $\hat{\mathbf{x}}\in\mathbb{S}_{\mathbf{q}},\,\om\in \mathbb{W}$.

Similarly, we deduce that the shear far field pattern $u^{\infty}_{\mathbf{F},s}(\hat{\mathbf{x}},\om)$ is also uniquely determined by the data set $\mathcal{C}_s$ for all $\hat{\mathbf{x}}\in\mathbb{S}_{\mathbf{q}}^{\perp},\,\om\in \mathbb{W}$.
The conclusion of the theorem now follows again from Theorem \ref{uni-full}.
\end{proof}

Theorem \ref{uni-phaseless012} also holds if the assumption \eqref{Rdiv>0} is replaced by
\be\label{Rdiv<0}
\Re(\ddiv\mathbf{F(y)})\leq0 \quad\mbox{and}\quad \exists\, \mathbf{y_0}\in\R^2\,\, s.t.\,\, \Re(\ddiv\mathbf{F(y_0)})<0.
\en
We think such an assumption can be removed. But this need different techniques.

In Theorems \ref{uni-phaseless012} and \ref{uni-phaseless01}, the phaseless data set includes the data with scattering strength $\tau=0$.
For diversity, we now prove a uniqueness theorem with three different scattering strengths.
Define
\ben
\mathcal {T}:= \{\tau_1,\tau_2,\tau_3\},
\enn
where $\tau_1, \tau_2, \tau_3\in\C$ are three different scattering strengths such that $\tau_2-\tau_1$ and $\tau_3-\tau_1$ are linearly independent.\\

\begin{theorem}\label{uni-phaseless123}
Fix $\mathbf{z}\in \R^2\ba\ov{\Om}$  and $\mathbf{q}\in \mathbb{S}$. Define again two circular arcs $\mathbb{S}_{\mathbf{q}}$ and $\mathbb{S}^{\perp}_{\mathbf{q}}$ as in \eqref{Sq}.
Then the source $\mathbf{F}$ is uniquely determined by the phaseless data set
\ben
\mathcal {D}_{p}:=\big\{|u^{\infty}_{\mathbf{F\cup\{z\}},p}(\mathbf{\hx, q}, \om, \tau)|:\,\mathbf{\hx}\in \mathbb{S}_{\mathbf{q}},\,\om\in \mathbb{W},\, \tau\in \mathcal {T}\big\}
\enn
and
\ben
\mathcal {D}_{s}:=\big\{|u^{\infty}_{\mathbf{F\cup\{z\}},s}(\mathbf{\hx, q}, \om, \tau)|:\,\mathbf{\hx}\in \mathbb{S}^{\perp}_{\mathbf{q}},\,\om\in \mathbb{W},\, \tau\in\mathcal {T}\big\}.
\enn
\end{theorem}
\begin{proof}
Recalling the representation \eqref{uinfFpz0} and noting that $\mathbf{\hx\cdot q}\geq 1/2$ for $\mathbf{\hx}\in\mathbb{S}_{\mathbf{q}}$, we deduce that
\ben
&&\frac{1}{\mathbf{\hx\cdot q}}\Big|u^{\infty}_{\mathbf{F}\cup\{\mathbf{z}\},p}(\mathbf{\hx,q},\om,\tau)\Big|\cr
&=&\frac{1}{\mathbf{\hx\cdot q}}\Big|u^{\infty}_{\mathbf{F},p}(\mathbf{\hx},\om)+\tau e^{-ik_p\mathbf{\hx\cdot z}}\mathbf{\hx\cdot q}\Big|\cr
&=&\Big|\frac{u^{\infty}_{\mathbf{F},p}(\mathbf{\hx},\om)e^{ik_p\mathbf{\hx\cdot z}}}{\mathbf{\hx\cdot q}}+\tau\Big|,\quad \mathbf{\hx}\in \mathbb{S}_{\mathbf{q}},\,\om\in \mathbb{W},\, \tau\in \mathcal {T}.
\enn
Define $z_j:=-\tau_{j},\,j=1,2,3$. Using Lemma \ref{phaseretrieval} we find that $\frac{u^{\infty}_{\mathbf{F},p}(\mathbf{\hx},\om)e^{ik_p\mathbf{\hx\cdot z}}}{\mathbf{\hx\cdot q}}$ is uniquely determined for all $\mathbf{\hx}\in \mathbb{S}_{\mathbf{q}},\,\om\in \mathbb{W}$. By analyticity, we also deduce that $u^{\infty}_{\mathbf{F},p}(\mathbf{\hx},\om)$ is uniquely determined for all $\mathbf{\hx}\in \mathbb{S},\,\om\in \mathbb{W}$.

Similarly, the shear far field pattern $u^{\infty}_{\mathbf{F},s}(\mathbf{\hx},\om)$ is uniquely determined for all $\mathbf{\hx}\in \mathbb{S},\,\om\in \mathbb{W}$ from the data set $\mathcal{D}_s$. Then the proof is finished by Theorem \ref{uni-full}.
\end{proof}

In the previous uniqueness results, we determine the source from both the compressional and shear far field data. A challenging open problem is whether one of the compressional and shear far field data completely determines the source.
We are also interested in broadband sparse measurements, i.e., the data can only be measured in finitely many observation directions,
\ben
\{\mathbf{\hx_1, \hx_2, \cdots, \hx_M}\}=:\Theta\subset \mathbb{S}.
\enn
In particular, we are interested in what information can be obtained using multi-frequency data with a single observation direction.
For a bounded domain $\Om$, the $\mathbf{\hx}$-strip hull of $\Om$ for a single observation direction $\mathbf{\hx} \in \,\Theta$ is defined by
\ben
S_{\Om}(\mathbf{\hx}):=  \{ \mathbf{y}\in \mathbb \R^{2}\; | \; \inf_{\mathbf{z}\in \Om}\mathbf{z}\cdot \mathbf{\hx} \leq \mathbf{y}\cdot \mathbf{\hx}\leq \sup_{\mathbf{z}\in \Om}\mathbf{z}\cdot \mathbf{\hx}\},
\enn
which is the smallest strip (region between two parallel lines) with normals in the directions $\pm \mathbf{\hx}$ that contains $\overline{\Om}$.
Let
\ben
\Pi_{\alpha}:=\{\mathbf{y}\in \R^2 |\, \mathbf{y}\cdot\mathbf{\hx}+\alpha=0\},\quad \alpha\in\R
\enn
be a line with normal $\mathbf{\hx}$. Define
\be
f_{\mathbf{\hx}}(\alpha):=\int_{\Pi_\alpha}\mathbf{\hx}\cdot \mathbf{F}(\mathbf{y})ds(\mathbf{y}),\quad \alpha\in\R.
\en
The following theorem gives a uniqueness result on determine the strip by using only compressional far field data or shear far field data at a fixed observation direction.
The proof follows the related results for inverse acoustic source scattering problems \cite{AlaHuLiuSun}.

\begin{theorem}\label{uni-strip}
For any fixed $\mathbf{\hx_{0}}\in\mathbb{S}$, choose a polarization $\mathbf{q_0}\in\mathbb{S}$ such that $\mathbf{q_0\cdot\hx_0}\neq 0$.
If the set
\be\label{set2}
\{\alpha\in\R|\, \Pi_\alpha\subset S_\Om(\mathbf{\hx}), f_{\mathbf{\hx_0}}(\alpha)=0\}
\en
has Lebesgue measure zero, then the strip $S_{\Om}(\mathbf{\hx})$ of the source support $\Om$ can be uniquely determined by
the following phaseless data set
\ben
\mathcal {E}_{p}:=\big\{|u^{\infty}_{\mathbf{F\cup\{z\}},p}(\mathbf{\hx_0, q_0}, \om, \tau)|:\,\om\in \mathbb{W},\, \tau\in \mathcal {T}\big\}.
\enn
\end{theorem}
\begin{proof}
Essentially, the same arguments as in the proof of Theorem \ref{uni-phaseless123} show that the compressional far field pattern $u^{\infty}_{\mathbf{F},p}(\mathbf{\hx_0}, \om)$
is uniquely determined by $\mathcal {E}_{p}$ for all $\om\in\mathbb{W}$ and for also all $\om\in\R$ by analyticity.
Recall the compressional far field pattern representation \eqref{uinfFp}, we have
\ben
u^{\infty}_{\mathbf{F},p}(\mathbf{\hx_0},\om)
&=&\int_{\R^2}e^{-ik_p\mathbf{\hx_0\cdot y}}\mathbf{\hx_0\cdot F(y)}d\mathbf{y}\cr
&=&\int_{\R}e^{ik_p\alpha}\int_{\Pi_{\alpha}}\mathbf{\hx_0\cdot F(y)}ds(\mathbf{y})d\alpha\cr
&=&\int_{\R}e^{ik_p\alpha}f_{\mathbf{\hx_0}}(\alpha)d\alpha, \quad \om\in\R.
\enn
Clearly, the compressional far field pattern $u^{\infty}_{\mathbf{F},p}(\mathbf{\hx_0},\om)$ is just the inverse Fourier transform of $f_{\mathbf{\hx_0}}$. This implies that $f_{\mathbf{\hx_0}}$ is uniquely determined by $u^{\infty}_{\mathbf{F},p}(\mathbf{\hx_0},\om), \om\in\mathbb{W}$. Under the assumption that the set in \eqref{set2} has Lebesgue measure zero, it is seen that
\ben
S_{\Om}(\mathbf{\hx})=\ov{\bigcup_{\alpha\in\R}\{\Pi_{\alpha} |\, f_{\mathbf{\hx_0}}(\alpha)\neq 0\}}
\enn
which implies that the strip $S_{\Om}(\mathbf{\hx})$ is uniquely determined by $f_{\mathbf{\hx_0}}$, and also by $u^{\infty}_{\mathbf{F},p}(\mathbf{\hx_0},\om)$, $\om\in\mathbb{W}$. The proof is complete.
\end{proof}

We end this subsection by the following remarks:
\begin{itemize}{\em
  \item Unique determination of the strip can also be established by using the phaseless shear far field data
\ben
\mathcal {E}_{s}:=\big\{|u^{\infty}_{\mathbf{F\cup\{z\}},s}(\mathbf{\hx_0, q_0}, \om, \tau)|:\,\om\in \mathbb{W},\, \tau\in \mathcal {T}\big\},
\enn
where the polarization $\mathbf{q_0}$ should be chosen such that $\mathbf{\hx_0}^{\perp}\cdot\mathbf{q_0}\neq 0$.
  \item Theorem \ref{uni-strip} is not true in general if the set in \eqref{set2} has positive Lebesgue measure.
  For example, for $\mathbf{y}=(y_1,y_2)^{T}\in\R^2$,
  we consider
  \be\label{f1}
  \mathbf{F}_1(\mathbf{y})=\left\{
         \begin{array}{ll}
           (1,0)^{T}, & \hbox{$y_1\in (-1,1), y_2\in [1,2)$;} \\
           (y_1, 0)^{T}, & \hbox{$y_1\in (-1,1), y_2\in (-1,1)$;} \\
           (1,0)^{T}, & \hbox{$y_1\in (-1,1), y_2\in (-2,-1]$;} \\
           (0,0)^{T}, & \hbox{otherwise.}
         \end{array}
       \right.
  \en
  and
  \be\label{f2}
  \mathbf{F}_2(\mathbf{y})=\left\{
         \begin{array}{ll}
           (1,0)^{T}, & \hbox{$y_1\in (-1,1), y_2\in [1,2)$;} \\
           (1,0)^{T}, & \hbox{$y_1\in (-1,1), y_2\in (-2,-1]$;}\\
           (0,0)^{T}, & \hbox{otherwise.}
         \end{array}
       \right.
  \en
  Then $\mathbf{F}_1\in [L^2(\R^2)]^2$ has compact support in $D_1=[-1,1]\times[-2,2]$, and $\mathbf{F}_2\in [L^2(\R^2)]^2$ has compact support in $D_2:=D_2^{(1)}\cup D_2^{(2)}$, where $D_2^{(1)}=[-1,1]\times[-2,-1]$ and $D_2^{(2)}=[-1,1]\times[1,2]$.
  Straightforward calculations show that the corresponding compressional far field patterns corresponding
  to these two different sources coincide for all frequencies at the fixed observation direction $\mathbf{\hx}=(1,0)^{T}$.
  \item A by-product of Theorem \ref{uni-strip} is that one may determine a convex hull of the source support $\Om$ by using two or more observation directions, which is verified in Section \ref{Numericals} numerically.}
\end{itemize}\quad\\

\subsection{Uniqueness for obstacles}
Difficulties arise for the uniqueness results for {\bf Phaseless-IP1} because of the additional far field pattern $v^{\infty}_{\Om}$ due to point sources, i.e., scattering effects of point sources can not be ignored, even they are very weak if the source is far away from the obstacle (see Theorem \ref{weakinteraction}).
Inspired by the recent works \cite{ZG,XZZ18}, we introduce an artificial rigid body $D\subset\R^2\ba\ov{\Om\cup{\{\mathbf{z}\}}}$ into the scattering system.
The reference ball technique dates back to \cite{LLZ}, where such a technique is used to avoid eigenvalues and choose a cut-off value for the linear sampling method with phased far field patterns. We also refer to \cite{KirschLiuIP14avoid} and \cite{QinLiu} for application of the reference ball techniques in phased inverse scattering problems.
We guess that such a body $D$ can be removed. However, this requires other techniques. We hope to report this elsewhere in the future.

To simplify notations, for fixed $\mathbf{z}$ and $\mathbf{q}$, we set
\ben
v^{\infty}_{D\cup\Om,p}(\mathbf{\hx}):=v^{\infty}_{D\cup\Om,p}(\mathbf{\hx,z,q},1)+ e^{-ik_p\mathbf{\hx\cdot z}}(\mathbf{q\cdot\hx}),\quad \mathbf{\hx}\in \mathbb{S}_{\bf q},\\
v^{\infty}_{D\cup\Om,s}(\mathbf{\hx}):=v^{\infty}_{D\cup\Om,s}(\mathbf{\hx,z,q},1)+ e^{-ik_s\mathbf{\hx\cdot z}}(\mathbf{q\cdot\hx^{\perp}}),\quad \mathbf{\hx}\in \mathbb{S}^{\perp}_{\bf q}.
\enn

\begin{theorem}\label{uni-body}
%Assume that $\Om\subset B_R$ for some disk $B_R$ with radius $R$ centered at the origin. We choose a rigid body $D\subset\R^2\ba\ov{B_R}$.
For any fixed frequency $\om>0$, source polarization $\bf q\in\mathcal {Q}$ and source point $\mathbf{z}\in \R^2\ba\ov{D\cup\Om}$, the scatterer $\Om$ is uniquely determined by the following phaseless data
\be
&&\label{wmp1}\Big|w^{\infty}_{D\cup\Om\cup\{\mathbf{z}\},mp}(\mathbf{\hx,d,q},\tau)\Big|, \quad \mathbf{\hx}\in \mathbb{S}_{\bf q},  \,\mathbf{d}\in\mathbb{S},\,\tau\in\mathcal {T}, \,m=p,s,\\
&&\label{wsp1}\Big|w^{\infty}_{D\cup\Om\cup\{\mathbf{z}\},ms}(\mathbf{\hx,d,q},\tau)\Big|, \quad \mathbf{\hx}\in \mathbb{S}^{\perp}_{\bf q}, \,\mathbf{d}\in\mathbb{S},\,\tau\in\mathcal {T},\,m=p,s,\\
&&\label{vp1}\Big|v^{\infty}_{D\cup\Om,p}(\mathbf{\hx})\Big|, \quad \mathbf{\hx}\in \mathbb{S}_{\bf q}, \\
&&\label{vs1}\Big|v^{\infty}_{D\cup\Om,s}(\mathbf{\hx})\Big|, \quad \mathbf{\hx}\in \mathbb{S}^{\perp}_{\bf q}.
\en
\end{theorem}
\begin{proof}
Define two sets
\ben
\mathbb{S}_0:=\Big\{\mathbf{\hx}\in \mathbb{S}_{\bf q}: \, \Big|v^{\infty}_{D\cup\Om,p}(\mathbf{\hx})\Big|=0\Big\}
\quad\mbox{and}\quad
\mathbb{S}^\perp_{0}:=\Big\{\mathbf{\hx}\in \mathbb{S}^{\perp}_{\bf q}: \, \Big|v^{\infty}_{D\cup\Om,s}(\mathbf{\hx})\Big|=0\Big\}.
\enn
We then claim that these two sets $\mathbb{S}_0$ and $\mathbb{S}^\perp_{0}$ have Lebesgue measure zero. For the case when the distance $\rho:=dist (\mathbf{z}, D\cup\Om)$ is large enough,
this is obvious with Theorem \ref{weakinteraction}. For the general case, we show only that $\mathbb{S}_0$ has Lebesgue measure zero, the other one can be proved analogously. On the contrary, assume that $\mathbb{S}_0$ has positive Lebesgue measure. By the analyticity of the far field pattern, we conclude that
$\mathbb{S}_0=\mathbb{S}_{\bf q}$. This implies that $v^{\infty}_{D\cup\Om,p}(\mathbf{\hx,z,q},1)=- e^{-ik_p\mathbf{\hx\cdot z}}(\mathbf{q}\cdot\mathbf{\hx})$ for $\mathbf{\hx}\in \mathbb{S}_{\bf q}$, and also for
all $\mathbf{\hx}\in \mathbb{S}$ by analyticity again.
Note that by \eqref{Phip}, $e^{-ik_p\mathbf{\hx\cdot z}}(\mathbf{q}\cdot\mathbf{\hx})$ is just the compressional far field pattern of the scattered field
$\Phi(\cdot, \mathbf{z})\mathbf{q}$. We conclude by Rellich's Lemma \cite{CK} and analytic continuation that
\be\label{vpPhip}
\mathbf{v}^{sc}_{D\cup\Om,p}(\mathbf{x,z,q},1) = \frac{1}{k^2_p}\ggrad\ddiv [\Phi(\mathbf{x,z})\mathbf{q}], \quad \mathbf{x}\in\R^2\ba\ov{D\cup\Om\cup\{\bf z\}}.
\en
This contradicts the fact that the scattered field $\mathbf{v}^{sc}_{D\cup\Om,p}(\mathbf{\cdot,z,q},1)$ is analytic in $\R^2\ba\ov{D\cup\Om}$ since the right hand side of
\eqref{vpPhip} is singular at $\mathbf{x=z}$.

Recall the representation \eqref{wmpdef}, we have for all $\mathbf{\hx}\in\mathbb{S}_{\bf q}\ba\mathbb{S}_{0},\,\mathbf{d}\in\mathbb{S},\,\tau\in\mathcal{T}$,
\ben
&&\Big|w^{\infty}_{D\cup\Om\cup\{z\},mp}(\mathbf{\hx,d,q},\tau)\Big|\cr
&=&\Big|u^{\infty}_{D\cup\Om,mp}(\mathbf{\hx,d})+\tau v^{\infty}_{D\cup\Om,p}(\mathbf{\hx})\Big|\cr
&=&\Big|v^{\infty}_{D\cup\Om,p}(\mathbf{\hx})\Big| \Big|\frac{u^{\infty}_{D\cup\Om,mp}(\mathbf{\hx,d})}{v^{\infty}_{D\cup\Om,p}(\mathbf{\hx})}+\tau\Big|,\ m=p,s.
\enn
By Lemma \ref{phaseretrieval}, we deduce that
\ben
\frac{u^{\infty}_{D\cup\Om,mp}(\mathbf{\hx,d})}{v^{\infty}_{D\cup\Om,p}(\mathbf{\hx})}, \quad \mathbf{\hx}\in\mathbb{S}_{\bf q}\ba\mathbb{S}_{0}, \,\mathbf{d}\in\mathbb{S},\, m=p,s
\enn
is uniquely determined by the phaseless data given in \eqref{wmp1} and \eqref{vp1}.
Note that $|v^{\infty}_{D\cup\Om,p}(\mathbf{\hx})|$ is given in \eqref{vp1}, we thus further obtain the phaseless data $|u^{\infty}_{D\cup\Om,mp}(\mathbf{\hx,d})$, $m=p,s$ uniquely for all
$\mathbf{\hx, d}\in\mathbb{S}$ by analyticity.
Similarly,
\ben
|u^{\infty}_{D\cup\Om,ms}(\mathbf{\hx,d})|\quad\mbox{and}\quad
\frac{u^{\infty}_{D\cup\Om,ms}(\mathbf{\hx,d})}{v^{\infty}_{D\cup\Om,s}(\mathbf{\hx})}, \quad \mathbf{\hx}\in\mathbb{S}^{\perp}_{\bf q}\ba\mathbb{S}^{\perp}_{0}, \,\mathbf{d}\in\mathbb{S},\, m=p,s
\enn
are uniquely determined by the phaseless data given in \eqref{wsp1} and \eqref{vs1}, then $|u^{\infty}_{D\cup\Om,ms}(\mathbf{\hx,d})|$, $m=p,s$ are obtained for all $\mathbf{\hx, d}\in\mathbb{S}$ uniquely.
Writing
\ben
&&u^{\infty}_{D\cup\Om,mn}(\mathbf{\hx,d}):=|u^{\infty}_{D\cup\Om,mn}(\mathbf{\hx,d})|e^{i\alpha_{mn}(\mathbf{\hx,d})},\quad \mathbf{\hx,d}\in\mathbb{S},\,m,n=p,s,\cr
&&v^{\infty}_{D\cup\Om,n}(\mathbf{\hx})=|v^{\infty}_{D\cup\Om,n}(\mathbf{\hx})|e^{i\beta_{n}(\mathbf{\hx})}, \quad \mathbf{\hx,d}\in\mathbb{S},\,n=p,s,
\enn
where the phase functions $\alpha_{mn}\in [0,2\pi)$ and $\beta_{n}\in [0,2\pi)$ are analytic on $\mathbb{S}$, and from the  previous analysis, we obtain that $e^{i[\alpha_{mn}(\mathbf{\hx,d})-\beta_{n}(\mathbf{\hx})]}$ is uniquely determined for all $\mathbf{\hx}\in \mathbb{S}_{\bf q}\ba\mathbb{S}_{0}$ or $\mathbf{\hx}\in \mathbb{S}^{\perp}_{\bf q}\ba\mathbb{S}^{\perp}_{0}$, $\mathbf{d}\in\mathbb{S}$, and thus for all $\mathbf{\hx,d}\in\mathbb{S}$ by analyticity.
We then claim that $\alpha_{mn}$ and $\beta_{n}$ are uniquely determined for all $m,n=p,s$.
To prove this, we assume there are two sets of phase functions $\alpha^{(i)}_{mn}$ and $\beta^{(i)}_{n}$, $i=1,2$ and $m,n=p,s$. Denote by $\Om_1$ and $\Om_2$ the corresponding rigid bodies. Then we obtain
\ben
e^{i[\alpha^{(1)}_{mn}(\mathbf{\hx,d})-\beta^{(1)}_{n}(\mathbf{\hx})]}=e^{i[\alpha^{(2)}_{mn}(\mathbf{\hx,d})-\beta^{(2)}_{n}(\mathbf{\hx})]}, \quad \mathbf{\hx, d}\in\mathbb{S},\, m,n=p,s,
\enn
that is,
\be\label{alphabeta}
\alpha^{(1)}_{mn}(\mathbf{\hx,d})-\beta^{(1)}_{n}(\mathbf{\hx})=\alpha^{(2)}_{mn}(\mathbf{\hx,d})-\beta^{(2)}_{n}(\mathbf{\hx}) +2l_{mn}\pi,\quad\mathbf{\hx,d}\in\mathbb{S},\, m,n=p,s\quad
\en
for some $l_{mn}\in\{0,\pm 1\}$.
Define
\be\label{gmndef}
\g_{mn}(\mathbf{\hx}):=\beta^{(1)}_{n}(\mathbf{\hx})-\beta^{(2)}_{n}(\mathbf{\hx}) +2l_{mn}\pi,\quad\mathbf{\hx}\in\mathbb{S},\, m,n=p,s.
\en
Then by \eqref{alphabeta} we have
\ben
\g_{mn}(\mathbf{\hx})= \alpha^{(1)}_{mn}(\mathbf{\hx,d})-\alpha^{(2)}_{mn}(\mathbf{\hx,d}),\quad\mathbf{\hx,d}\in\mathbb{S},\, m,n=p,s.
\enn
From this, noting that $|u^{\infty}_{D\cup\Om_1,mn}(\mathbf{\hx,d})|=|u^{\infty}_{D\cup\Om_2,mn}(\mathbf{\hx,d})|$ for $\mathbf{\hx, d}\in\mathbb{S}$, we observe that
\be\label{uinf12g}
u^{\infty}_{D\cup\Om_1,mn}(\mathbf{\hx,d})
&=&|u^{\infty}_{D\cup\Om_1,mn}(\mathbf{\hx,d})|e^{i\alpha^{(1)}_{mn}(\mathbf{\hx,d})}\cr
&=&|u^{\infty}_{D\cup\Om_2,mn}(\mathbf{\hx,d})|e^{i\alpha^{(2)}_{mn}(\mathbf{\hx,d})}e^{i\g_{mn}(\mathbf{\hx})}\cr
&=&u^{\infty}_{D\cup\Om_2,mn}(\mathbf{\hx,d})e^{i\g_{mn}(\mathbf{\hx})},\quad \mathbf{\hx, d}\in\mathbb{S},\,m,n=p,s.
\en
Interchanging the roles of $\mathbf{\hx}$ and $-\mathbf{d}$ gives
\be\label{uinf12g2}
u^{\infty}_{D\cup\Om_1,mn}(\mathbf{-d, -\hx})
=u^{\infty}_{D\cup\Om_2,mn}(\mathbf{-d,-\hx})e^{i\g_{mn}(-\mathbf{d})},\quad \mathbf{\hx, d}\in\mathbb{S},\,m,n=p,s.\quad
\en
Using the reciprocity relation \eqref{ReciprocityRelations}, from \eqref{uinf12g}, we see that
\be\label{uinf12g3}
u^{\infty}_{D\cup\Om_1,nm}(\mathbf{-d,-\hx})=u^{\infty}_{D\cup\Om_2,nm}(\mathbf{-d,-\hx})e^{i\g_{mn}(\mathbf{\hx})},\quad \mathbf{\hx, d}\in\mathbb{S},\,m,n=p,s.
\en
Comparing \eqref{uinf12g2} and \eqref{uinf12g3}, we deduce that
\ben
\g_{nm}(-\mathbf{d})=\g_{mn}(\mathbf{\hx}), \quad \mathbf{\hx, d}\in\mathbb{S},\,m,n=p,s.
\enn
Thus, $\g_{mn}=\g_{nm}$ are constants independent of the directions $\mathbf{\hx, d}\in\mathbb{S}$.
Recall the definition \eqref{gmndef} of $\g_{mn}$, we find that $\g_{pm}-\g_{sm}=2(l_{pm}-l_{sm})\pi$ for two constants $l_{pm}, l_{sm}\in \{0,\pm 1\}$, $m=p,s$.
Thus $e^{i\g_{mn}}=e^{i\g}, m,n=p,s$ for some fixed constant $\g$.

Let $G$ be the unbounded connected component of the complement of $\Om_1\cup\Om_2$.
We apply Rellich Lemma to conclude that
\be\label{us1us2}
\mathbf{u}^{sc}_{D\cup\Om_1,nm}(\mathbf{y},\mathbf{-\hx})=\mathbf{u}^{sc}_{D\cup\Om_2,nm}(\mathbf{y},\mathbf{-\hx})e^{i\g},\quad
\mathbf{y}\in G\ba\ov{D}, \,\mathbf{\hx}\in\mathbb{S},\, m,n=p,s.\qquad
\en
Since $D\cap\Om_i=\emptyset,\,i=1,2$, we have
\ben
D\subset G \quad\mbox{or}\quad D\subset\R^2\ba\ov{G}.
\enn
If $D\subset\R^2\ba\ov{G}$, there exists a point $\mathbf{y}_0\in\pa G$ such that $\mathbf{y}_0\in\pa\Om_1\cap\pa\Om_2$. For the case of $D\subset G$, we choose any point $\mathbf{y}_1\in\pa D$.
Then we define
\ben
\mathbf{y^\ast}:=\left\{
                   \begin{array}{ll}
                     \mathbf{y}_0, & \hbox{$D\subset\R^2\ba\ov{G}$;} \\
                     \mathbf{y}_1, & \hbox{$D\subset G$.}
                   \end{array}
                 \right.
\enn
Using the Dirichlet boundary conditions, we have
\be\label{u12=0}
\sum_{n\in\{p,s\}}\mathbf{u}^{sc}_{D\cup\Om_1,mn}(\mathbf{y}^\ast,\mathbf{-\hx})
&=&-\mathbf{u}^{in}_{m}(\mathbf{y^\ast,-\hx})\cr
&=&\sum_{n\in\{p,s\}}\mathbf{u}^{sc}_{D\cup\Om_2,mn}(\mathbf{y}^\ast,\mathbf{-\hx}),\quad\mathbf{\hx}\in\mathbb{S},\,m=p,s.
\en
%From the previous two equations we deduce
%\be\label{11}
%&&\mathbf{u}^{sc}_{D\cup\Om_2,mp}(\mathbf{y},\mathbf{-\hx})e^{i\g_{pm}}+\mathbf{u}^{sc}_{D\cup\Om_2,ms}(\mathbf{y},\mathbf{-\hx})e^{i\g_{sm}}\cr
%&=&-\mathbf{u}^{in}_{m}(\mathbf{y,-\hx}),\quad\mathbf{y}\in\pa D, \,\mathbf{\hx}\in\mathbb{S},\,m=p,s.
%\en
%Using the boundary condition again, we have
%\be\label{22}
%&&\mathbf{u}^{sc}_{D\cup\Om_2,mp}(\mathbf{y},\mathbf{-\hx})+\mathbf{u}^{sc}_{D\cup\Om_2,ms}(\mathbf{y},\mathbf{-\hx})\cr
%&=&-\mathbf{u}^{in}_{m}(\mathbf{y,-\hx}),\quad\mathbf{y}\in\pa D, \,\mathbf{\hx}\in\mathbb{S},\,m=p,s.
%\en
Combining this with \eqref{us1us2} we find
%We now subtract \eqref{22} from \eqref{11} to obtain
\be\label{33}
\mathbf{u}^{in}_{m}(\mathbf{y}^\ast,\mathbf{-\hx})[e^{i\g}-1]=0,\quad\mathbf{\hx}\in\mathbb{S},\,m=p,s.
\en
From this and $|\mathbf{u}^{in}_{m}(\mathbf{y^\ast,-\hx})|=1$, it can be deduced that
\ben
e^{i\g}= 1.
\enn
Combining this with \eqref{uinf12g}, we find that
\ben
u^{\infty}_{D\cup\Om_1,mn}(\mathbf{\hx,d}) = u^{\infty}_{D\cup\Om_2,mn}(\mathbf{\hx,d}),\quad \mathbf{\hx, d}\in\mathbb{S},\,m,n=p,s.
\enn
The proof is completed by applying the classical uniqueness result with phased far field pattern \cite{HahnerHsiao}.
%Note that $\mathbf{u}^{sc}_{D\cup\Om_2,mp}(\mathbf{y},\mathbf{-\hx})[e^{i\g_{pm}(\mathbf{\hx})}-1]+\mathbf{u}^{sc}_{D\cup\Om_2,ms}(\mathbf{y},\mathbf{-\hx})[e^{i\g_{sm}(\mathbf{\hx})}-1]$ is a solution of the Navier equation in $\R^2\ba\ov{D\cup\Om_2}$. Since $\om^2$ is not a Dirichlet eigenvalue of $-\Delta^{\ast}$ in $D$, we deduce that
%\ben
%\mathbf{u}^{sc}_{D\cup\Om_2,mp}(\mathbf{y},\mathbf{-\hx})[e^{i\g_{pm}(\mathbf{\hx})}-1]+\mathbf{u}^{sc}_{D\cup\Om_2,ms}(\mathbf{y},\mathbf{-\hx})[e^{i\g_{sm}(\mathbf{\hx})}-1]=0,\quad \mathbf{y}\in D.
%\enn
%Define
%\ben
%\mathbf{u}(\mathbf{y}):=\mathbf{u}^{sc}_{D\cup\Om_2,mp}(\mathbf{y},\mathbf{-\hx})[e^{i\g_{pm}(\mathbf{\hx})}-1],\quad \mathbf{y}\in D
%\enn
%
%and also for $\mathbf{y}\in \R^2\ba\ov{D\cup\Om_2}$ by analyticity.
\end{proof}

\begin{remark}
For the case of $D\subset\R^2\ba\ov{G}$, we have used the Dirichlet boundary conditions on $\pa\Om_i$ in the proof. Such a case can be avoided if we know a priori that the unknown scatterers $\Om_i, \,i=1,2$ are located in some big ball $B_R$ with radius $R$ centered at the origin, and we choose $D$ outside of $B_R$.

Essentially the same arguments as in the proof of Theorem \ref{uni-body} show that the scatterer can be uniquely determined by the phaseless far field patterns with multiple frequencies and one fixed observation direction.

Using Lemma \ref{phaseretrieval} and following the first part of the previous proof of Theorem \ref{uni-body}, for any fixed polarization $\bf q\in\mathcal {Q}$ and source point $\mathbf{z}\in \R^2\ba\ov{\Om}$, we can show that the scatterer can be uniquely determined by the following data
\be
&&\label{wmp2}\Big|w^{\infty}_{\Om\cup\{\mathbf{z}\},mp}(\mathbf{\hx,d,q},\tau)\Big|, \quad \mathbf{\hx}\in \mathbb{S}_{\bf q},  \,\mathbf{d}\in\mathbb{S},\,\tau\in\mathcal {T}, \,m=p,s,\\
&&\label{wsp2}\Big|w^{\infty}_{\Om\cup\{\mathbf{z}\},ms}(\mathbf{\hx,d,q},\tau)\Big|, \quad \mathbf{\hx}\in \mathbb{S}^{\perp}_{\bf q}, \,\mathbf{d}\in\mathbb{S},\,\tau\in\mathcal {T},\,m=p,s,\\
&&\label{vp2}v^{\infty}_{\Om,p}(\mathbf{\hx,z,q},1), \quad \mathbf{\hx}\in \mathbb{S}_{\bf q}, \\
&&\label{vs2}v^{\infty}_{\Om,s}(\mathbf{\hx,z,q},1), \quad \mathbf{\hx}\in \mathbb{S}^{\perp}_{\bf q}.
\en
Here, we remove the artificial rigid body $D$. However, there is a price to pay, that is, we need phased data \eqref{vp2}-\eqref{vs2}.
\end{remark}

Under more assumptions on the artificial rigid body $D$, the following Theorem gives a uniqueness result with less data.

\begin{theorem}\label{uni-body2}
Assume that $\Om\subset B_R$ for some disk $B_R$ with radius $R$ centered at the origin. We choose a convex rigid body $D\subset\R^2\ba\ov{B_R}$ such that $\om^2$ is not a Dirichlet eigenvalue of $-\Delta^{\ast}$ in $D$.
Define $\mathcal{T}_0:=\{0,\tau_1\}$, where $\tau_1\in\R\ba\{0\}$.
For the fixed frequency $\om>0$, source polarization $\bf q\in\mathcal {Q}$ and source point $\mathbf{z}\in \R^2\ba\ov{\Om}$, the scatterer $\Om$ is uniquely determined by the following phaseless data
\be
&&\label{wmp3}\Big|w^{\infty}_{D\cup\Om\cup\{\mathbf{z}\},mp}(\mathbf{\hx,d,q},\tau)\Big|, \quad \mathbf{\hx}\in \mathbb{S}_{\bf q},  \,\mathbf{d}\in\mathbb{S},\,\tau\in\mathcal {T}_0, \,m=p,s,\\
&&\label{wsp3}\Big|w^{\infty}_{D\cup\Om\cup\{\mathbf{z}\},ms}(\mathbf{\hx,d,q},\tau)\Big|, \quad \mathbf{\hx}\in \mathbb{S}^{\perp}_{\bf q}, \,\mathbf{d}\in\mathbb{S},\,\tau\in\mathcal {T}_0,\,m=p,s,\\
&&\label{vp3}\Big|v^{\infty}_{D\cup\Om,p}(\mathbf{\hx})\Big|, \quad \mathbf{\hx}\in \mathbb{S}_{\bf q}, \\
&&\label{vs3}\Big|v^{\infty}_{D\cup\Om,s}(\mathbf{\hx})\Big|, \quad \mathbf{\hx}\in \mathbb{S}^{\perp}_{\bf q}.
\en
\end{theorem}
\begin{proof}
Using \eqref{wmpdef}, we have
\ben
&&\Big|w^{\infty}_{D\cup\Om\cup\{z\},mp}(\mathbf{\hx,d,q},\tau_1)\Big|^2\cr
&=&\Big|u^{\infty}_{D\cup\Om,mp}(\mathbf{\hx,d})+\tau_1 v^{\infty}_{D\cup\Om,p}(\mathbf{\hx})\Big|^2\cr
&=&\Big|u^{\infty}_{D\cup\Om,mp}(\mathbf{\hx,d})\Big|^2+\tau_1^2\Big|v^{\infty}_{D\cup\Om,p}(\mathbf{\hx})\Big|^2\cr
&&+2\tau_1\Re\Big[u^{\infty}_{D\cup\Om,mp}(\mathbf{\hx,d})\ov{v^{\infty}_{D\cup\Om,p}(\mathbf{\hx})}\Big], \quad \mathbf{\hx}\in\mathbb{S}_{\bf q},\,\mathbf{d}\in\mathbb{S},\,m=p,s.
\enn
Thus $\Re\Big[u^{\infty}_{D\cup\Om,mp}(\mathbf{\hx,d})\ov{v^{\infty}_{D\cup\Om,p}(\mathbf{\hx})}\Big],\,\mathbf{\hx}\in\mathbb{S}_{\bf q},\,\mathbf{d}\in\mathbb{S},\,m=p,s$ is uniquely determined by the phaseless data \eqref{wmp3} and \eqref{vp3}.
Similarly, $\Re\Big[u^{\infty}_{D\cup\Om,ms}(\mathbf{\hx,d})\ov{v^{\infty}_{D\cup\Om,s}(\mathbf{\hx})}\Big],\,\mathbf{\hx}\in\mathbb{S}^{\perp}_{\bf q},\,\mathbf{d}\in\mathbb{S},\,m=p,s$ is uniquely determined by the phaseless data \eqref{wsp3} and \eqref{vs3}. Using analyticity, both of them are uniquely determined for $\mathbf{\hx}\in\mathbb{S}$.

Writing
\ben
u^{\infty}_{D\cup\Om,mn}(\mathbf{\hx,d}) &=& |u^{\infty}_{D\cup\Om,mn}(\mathbf{\hx,d})|e^{i\alpha_{mn}(\mathbf{\hx,d})}, \,\mathbf{\hx}\in\mathbb{S},\,\mathbf{d}\in\mathbb{S},\,m,n=p,s\cr
v^{\infty}_{D\cup\Om,n}(\mathbf{\hx}) &=& |v^{\infty}_{D\cup\Om,n}(\mathbf{\hx})|e^{i\beta_{n}(\mathbf{\hx})}, \,\mathbf{\hx}\in\mathbb{S},n=p,s.
\enn
where the phase functions $\alpha_{mn}\in [0,2\pi)$ and $\beta_{n}\in [0,2\pi)$ are analytic on $\mathbb{S}$. Then, by analyticity of $\alpha_{mn}$ and $\beta_n$, the previous analysis shows that
\be\label{cosmn}
\cos[\alpha_{mn}(\mathbf{\hx,d})-\beta_{n}(\mathbf{\hx})], \quad \mathbf{\hx}\in\mathbb{S},\, \mathbf{d}\in\mathbb{S},\,m,n=p,s
\en
are uniquely determined.
We claim that $\alpha_{mn}$ and $\beta_n$ are uniquely determined. Assume on the contrary that there are two sets of $\alpha^{(i)}_{mn}$ and $\beta^{(i)}_n$,
and the corresponding scatterers are $\Om_{i},\,i=1,2$.
Then we have either
\be\label{case11}
\alpha_{mn}^{(1)}(\mathbf{\hx,d})-\beta^{(1)}_{n}(\mathbf{\hx}) = \alpha^{(2)}_{mn}(\mathbf{\hx,d})-\beta^{(2)}_{n}(\mathbf{\hx})+2l_{mn}\pi, \quad \mathbf{\hx}\in\mathbb{S},\, \mathbf{d}\in\mathbb{S}, \qquad
\en
or
\be\label{case22}
\alpha_{mn}^{(1)}(\mathbf{\hx,d})-\beta^{(1)}_{n}(\mathbf{\hx}) = -\alpha^{(2)}_{mn}(\mathbf{\hx,d})+\beta^{(2)}_{n}(\mathbf{\hx})+2l'_{mn}\pi, \quad \mathbf{\hx}\in\mathbb{S},\, \mathbf{d}\in\mathbb{S}, \qquad
\en
for some $l_{mn},l'_{mn}\in\{0,\pm1\},\,m,n=p,s$.
Arguing similarly as in the proof of Theorem \ref{uni-body}, we can obtain from \eqref{case11}-\eqref{case22} that either
\be\label{uinfcase11}
u^{\infty}_{D\cup\Om_1,mn}(\mathbf{\hx, d})=e^{i\xi}u^{\infty}_{D\cup\Om_2,mn}(\mathbf{\hx, d}), \quad \mathbf{\hx}\in\mathbb{S},\, \mathbf{d}\in\mathbb{S}, \,m,n=p,s\qquad
\en
or
\be\label{uinfcase22}
u^{\infty}_{D\cup\Om_1,mn}(\mathbf{\hx, d})=e^{i\eta}\ov{u^{\infty}_{D\cup\Om_2,mn}(\mathbf{\hx, d})}, \quad \mathbf{\hx}\in\mathbb{S},\, \mathbf{d}\in\mathbb{S}, \,m,n=p,s,\qquad
\en
where $\xi$ and $\eta$ are two constants independent of $\mathbf{\hx,d}\in \mathbb{S}$ and $m,n\in \{p,s\}$.
%where $\xi_{mn}$ and $\eta_{mn}$ are constants satisfying
%\ben
%\xi_{mn}=\xi_{nm},\quad \eta_{mn}-\eta_{n,m},\quad m,n=p,s
%\enn
%and
%\ben
%\frac{\xi_{mp}-\xi_{ms}}{2\pi}\in\Z, \quad \frac{\eta_{mp}-\eta_{ms}}{2\pi}\in\Z,\quad m=p,s.
%\enn

Using further Rellich's Lemma, from \eqref{uinfcase11}, we deduce that
\ben
\mathbf{u}^{sc}_{D\cup\Om_1,mn}(\mathbf{x, d})=e^{i\xi}\mathbf{u}^{sc}_{D\cup\Om_2,mn}(\mathbf{x, d}), \quad \mathbf{x}\in G,\, \mathbf{d}\in\mathbb{S}, \,m,n=p,s,
\enn
where $G$ is the unbounded connected component of the compliment of $D\cup\Om_1\cup\Om_2$. The assumption on the artificial rigid body $D$ implies that $\pa D\subset\pa G$ and
\ben
\sum_{n\in\{p,s\}}\mathbf{u}^{sc}_{D\cup\Om_1,mn}(\mathbf{x, d})&=&-\mathbf{u}^{in}_{m}(\mathbf{x,d})\cr
&=&\sum_{n\in\{p,s\}}\mathbf{u}^{sc}_{D\cup\Om_2,mn}(\mathbf{x, d}), \quad \mathbf{x}\in \pa D,\, \mathbf{d}\in\mathbb{S}, \,m=p,s,
\enn
Combining the previous two equalities implies that
\ben
(e^{i\xi}-1)\mathbf{u}^{in}_{m}(\mathbf{x,d})=0, \quad \mathbf{x}\in \pa D,\, \mathbf{d}\in\mathbb{S}, \,m=p,s.
\enn
Thus $e^{i\xi}=1$ by noting the fact that $|\mathbf{u}^{in}_{m}(\mathbf{x,d})|=1$ and consequently
\ben
u^{\infty}_{D\cup\Om_1,mn}(\mathbf{\hx, d})=u^{\infty}_{D\cup\Om_2,mn}(\mathbf{\hx, d}), \quad \mathbf{\hx,d}\in\mathbb{S},\,m,n=p,s.\qquad
\enn
Then the statement of the theorem follows from the classical uniqueness result with phased far field patterns \cite{HahnerHsiao}.

We finally show that \eqref{uinfcase22} does not hold. Denote by
\ben
\mathbf{u}^{sc}_{D\cup\Om_i,m}\mathbf{(x,d)}:=\sum_{n\in\{p,s\}}\mathbf{u}^{sc}_{D\cup\Om_i,mn}\mathbf{(x,d)}, \quad \mathbf{x}\in\R^2\ba\ov{D\cup\Om_i},\,\mathbf{d}\in\mathbb{S}
\enn
the scattered field due to scattering of plane wave $\mathbf{u}^{in}_{m}$ by the scatterer $D\cup\Om_i, i=1,2$.
From \eqref{uinfcase22}, by \eqref{upinfty2}, we have
\be\label{u1qinf=ovu2inf}
&&u^{\infty}_{D\cup\Om_1,mn}(\mathbf{\hx, d})\cr
&=&e^{i\eta}\ov{u^{\infty}_{D\cup\Om_2,mn}(\mathbf{\hx, d})}\cr
&=&-e^{i\eta}\int_{\pa D\cup\pa B_R}\Big\{\mathbf{u}_{n}^{in}(\mathbf{y,\hat{x}})\cdot\ov{\mathbb{T}_{\mathbf{\nu(y)}}\mathbf{u}^{sc}_{D\cup\Om_2,m}\mathbf{(y,d)}}\cr
&&\qquad -[\mathbb{T}_{\mathbf{\nu(y)}}\mathbf{u}_{n}^{in}(\mathbf{y,\hat{x}})]\cdot\ov{\mathbf{u}^{sc}_{D\cup\Om_2,m}\mathbf{(y,d)}}\Big\}ds(\mathbf{y})\cr
&=&e^{i\eta}\int_{\pa \wi{D}\cup\pa B_R}\Big\{\mathbf{u}_{n}^{in}(\mathbf{y,-\hat{x}})\cdot\ov{\mathbb{T}_{\mathbf{\nu(y)}}\mathbf{u}^{sc}_{D\cup\Om_2,m}\mathbf{(-y,d)}}\cr
&&\qquad -[\mathbb{T}_{\mathbf{\nu(y)}}\mathbf{u}_{n}^{in}(\mathbf{y,-\hat{x}})]\cdot\ov{\mathbf{u}^{sc}_{D\cup\Om_2,m}\mathbf{(-y,d)}}\Big\}ds(\mathbf{y}),
\quad \mathbf{\hx,d}\in\mathbb{S},\,m,n=p,s,\qquad
\en
where $\wi{D}:=\{\mathbf{x}\in\R^2: \,-\mathbf{x}\in D\}$. Since $D\subset\R^2\ba\ov{B_R}$ is convex, we have $\wi{D}\subset\R^2\ba\ov{B_R}$ and $\ov{D\cap\wi{D}}=\emptyset$.
Define
\ben
\mathbb{L}_n:=\left\{
                \begin{array}{ll}
                  -\frac{1}{k^2_p}\ggrad\ddiv, & \hbox{$n=p$;} \\
                  -\frac{1}{k^2_s}\ggrad^{\perp}\ddiv^{\perp}, & \hbox{$n=s$.}
                \end{array}
              \right.
\enn
Note that the right hand side of the above equality \eqref{u1qinf=ovu2inf} is the far field pattern of the scattered field
\ben
&&\mathbf{g}^{sc}_{mn}(\mathbf{x,d})\cr
&:=&e^{i\eta}\mathbb{L}_n\int_{\pa \wi{D}\cup\pa B_R}\Big\{\Phi(\mathbf{x,y})\ov{\mathbb{T}_{\mathbf{\nu(y)}}\mathbf{u}^{sc}_{D\cup\Om_2,m}\mathbf{(-y,d)}}\cr
&& -[\mathbb{T}_{\mathbf{\nu(y)}}\Phi(\mathbf{x,y})]^T\ov{\mathbf{u}^{sc}_{D\cup\Om_2,m}\mathbf{(-y,d)}}\Big\}ds(\mathbf{y}),
\quad \mathbf{x}\in \R^2\ba\ov{\wi{D}\cup B_R}, \,\mathbf{d}\in\mathbb{S},\,m,n=p,s.
\enn
Rellich's Lemma gives that
\ben
\mathbf{u}^{sc}_{D\cup\Om_1,mn}(\mathbf{x, d})=\mathbf{g}^{sc}_{mn}(\mathbf{x,d}),\quad \mathbf{x}\in \R^2\ba\ov{D\cup\wi{D}\cup B_R}, \,\mathbf{d}\in\mathbb{S},\,m,n=p,s.
\enn
By the analyticity of $\mathbf{g}^{sc}_{mn}$ in $\R^2\ba\ov{\wi{D}\cup B_R}$, we find that the scattered field $\mathbf{u}^{sc}_{D\cup\Om_1,mn}$, $m,n=p,s$ can be analytically extended into $D$ and satisfies the Navier equation in $D$. Therefore, we deduce that the total field $\mathbf{u}_{D\cup\Om_1,m}:=\mathbf{u}^{in}_m+\mathbf{u}^{sc}_{D\cup\Om_1,m}$ satisfies
\ben
\Delta^\ast \mathbf{u}_{D\cup\Om_1,m} +\om^2 \mathbf{u}_{D\cup\Om_1,m}=0\,\, \mbox{in}\,\, D\quad\mbox{and}\quad  \mathbf{u}_{D\cup\Om_1,m}=0\,\, \mbox{on}\,\, \pa D,\quad \, m=p,s.
\enn
The assumption that $\om^2$ is not a Dirichlet eigenvalue of $-\Delta^\ast$ in $D$ implies that $\mathbf{u}_{D\cup\Om_1,m}=0$ in $D$. From this and the analyticity of $\mathbf{u}_{D\cup\Om_1,m}$ we finally obtain $\mathbf{u}_{D\cup\Om_1,m}=0$ in $\R^2\ba\ov{B_R}$. However this leads to a contradiction to $|\mathbf{u}^{in}_m(\mathbf{x,d})|=1$
for all $\mathbf{x}\in\R^2$ and $\mathbf{u}^{sc}_{D\cup\Om_1,m}(\mathbf{x,d})\rightarrow 0$ as $|\mathbf{x}|\rightarrow \infty$.
\end{proof}

\section{Phase retrieval and shape reconstruction methods}\label{sec:phaseretrieval+Shapereconstruction}
This section devotes to the numerical schemes for shape reconstruction with phaseless far field data. First, we propose a fast and stable phase retrieval approach using a simple geometric structure which provides a stable reconstruction of a point in the plane from three given distances. Then, several sampling methods for shape reconstruction with phaseless far field data are given. For obstacle scattering problems, the shear far field pattern corresponding to incident plane shear wave is considered, two different direct sampling methods are proposed with data at a fixed frequency. For inverse source scattering problems, the shear far field pattern is considered, we introduce two direct sampling methods for source supports with sparse multi-frequency data. Other cases follow similarly. The phase retrieval techniques are also combined with the classical sampling methods for the shape reconstructions.
\subsection{Phase retrieval}\label{PR}
In this subsection, we introduce a phase retrieval method based on the following geometric result \cite{JiLiuZhang-obstacle}.

\begin{lemma}\label{phaseretrieval}
Let $z_j:=x_j+iy_j,\,j=1,2,3,$ be three different complex numbers such that they are not collinear.
Then there is at most one complex number $z\in\C$ with the distances $r_j=|z-z_j|,\,j=1,2,3$. Let further $\eps>0$ and assume that
\ben
|r_j^{\eps}-r_j|\leq\eps, \quad j=1,2,3.
\enn
Here, and throughout the paper, we use the subscript $\eps$ to denote the polluted data.
Then there exists a constant $c>0$ depending on $z_j, j=1,2,3$, such that
\ben
|z^{\eps}-z|\leq c\eps.
\enn
\end{lemma}

Based on Lemma \ref{phaseretrieval}, we have the following stable phase retrieval scheme which can be implemented easily.

\begin{figure}[htbp]
\centering
\includegraphics[width=2.5in]{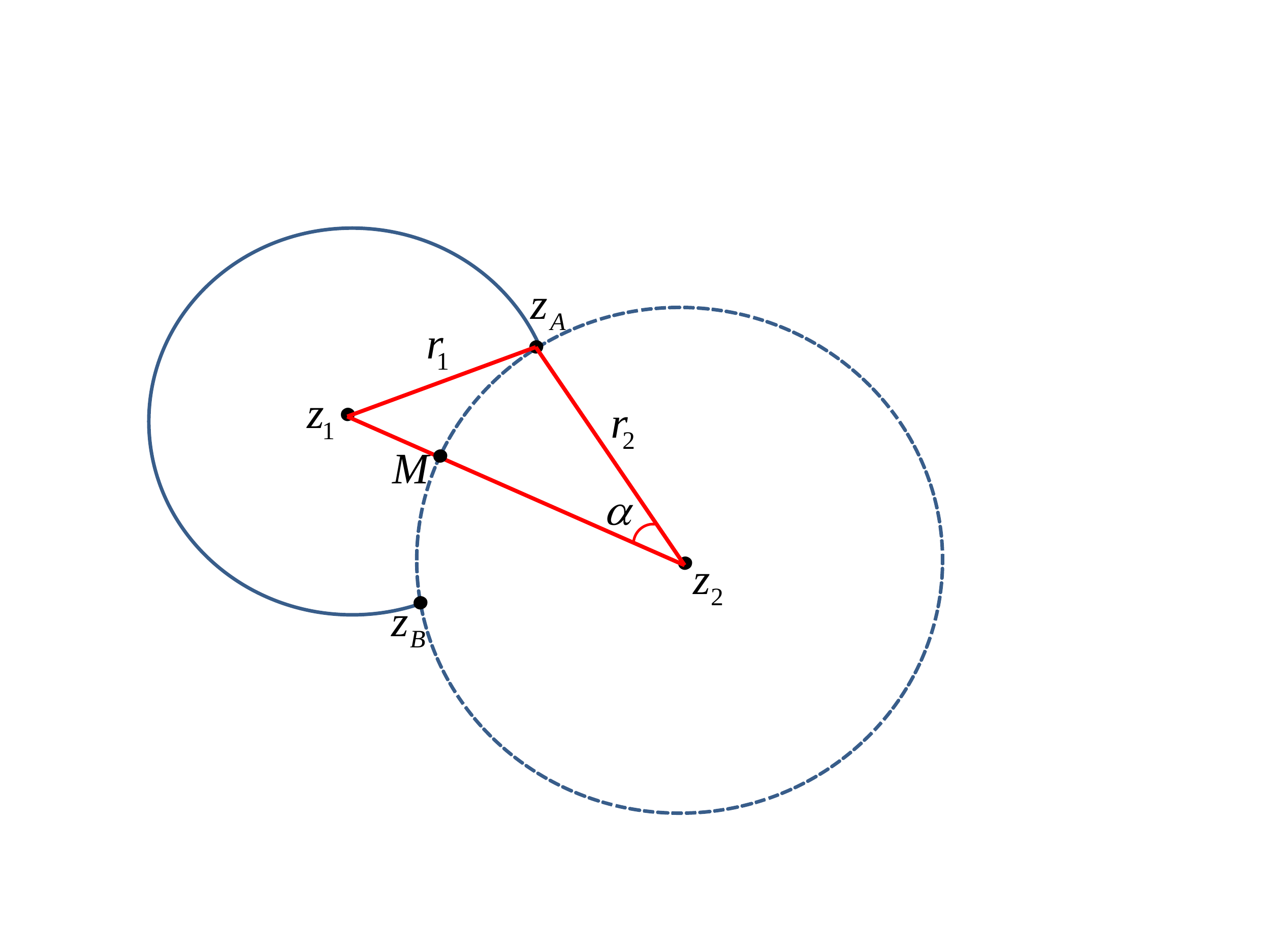}
\caption{Sketch map for phase retrieval scheme.}
\label{twodisks}
\end{figure}

{\bf Phase Retrieval Scheme.}

\begin{itemize}{\em
  \item (1). {\bf\rm Collect the distances $r_j:=|z-z_j|$ with given complex numbers $z_j,\,j=1,2,3$.} If $r_j=0$ for some $j\in\{1,2,3\}$, then $z=z_j$. Otherwise, go to next step.
  \item (2). {\bf\rm Look for the point $M=(x_M, y_M)$.} As shown in Figure \ref{twodisks}, $M$ is the intersection of circle centered at $Z_2$ with radius $r_2$ and the ray $z_2z_1$ with initial point $z_2$. Denote by $d_{1,2}:=|z_1-z_2|$ the distance between $z_1$ and $z_2$, then
       \be\label{xMyM}
       x_M=\frac{r_2}{d_{1,2}}x_1+\frac{d_{1,2}-r_2}{d_{1,2}}x_2,\quad y_M=\frac{r_2}{d_{1,2}}y_1+\frac{d_{1,2}-r_2}{d_{1,2}}y_2,
       \en
  \item (3). {\bf\rm Look for the points $z_A=(x_A, y_A)$ and $z_B=(x_B, y_B)$.} Note that $z_A$ and $z_B$ are just two rotations of $M$ around the point $z_2$. Let $\alpha\in [0,\pi]$ be the angle between rays $z_2z_1$ and $z_2z_A$. Then, by the law of cosines, we have
       \be\label{cosalpha}
       \cos \alpha=\frac{r_2^2+d_{1,2}^2- r_1^2}{2r_2d_{1,2}}.
       \en
       Note that $\alpha\in [0,\pi]$ and $\sin^2 \alpha+\cos^2 \alpha=1$, we deduce that $\sin \alpha=\sqrt{1-\cos^2 \alpha}$.
       Then
       \be
       \label{xA}x_A &=& x_2+\Re\{[(x_M-x_2)+i(y_M-y_2)]e^{-i\alpha}\},\\
       \label{yA}y_A &=& y_2+\Im\{[(x_M-x_2)+i(y_M-y_2)]e^{-i\alpha}\},\\
       \label{xB}x_B &=& x_2+\Re\{[(x_M-x_2)+i(y_M-y_2)]e^{i\alpha}\},\\
       \label{yB}y_B &=& y_2+\Im\{[(x_M-x_2)+i(y_M-y_2)]e^{i\alpha}\}.
       \en
  \item (4). {\bf\rm Determine the point $z$. $z=z_A$ if the distance $|z_Az_3|=r_3$, or else $z=z_B$.}}
\end{itemize}
\quad\\

Lemma \ref{phaseretrieval} can immediately be applied to the elastic source scattering problems. Indeed, we set $z_j=-\tau_j$, where $\tau_j\in\C, j=1,2,3$ are three scattering strengths with different principle arguments.
For any observation direction $\mathbf{\hx}\in\mathbb{S}$, there exists an arc $\mathbb{S}^{\perp}_{\mathbf{q}}$ for some $\mathbf{q}\in\mathcal{Q}$ such that $\mathbf{\hx}\in\mathbb{S}^{\perp}_{\mathbf{q}}$, i.e., $\mathbf{\hx^{\perp}\cdot q}\geq1/2$.
Applying Lemma \ref{phaseretrieval} and following the proof in Theorem \ref{uni-phaseless123}, we obtain the approximate far field pattern
$u^{\infty,\eps}_{\mathbf{F},s}(\mathbf{\hx},\om)$ from the perturbed phaseless far field data set
$\Big\{\big|u^{\infty,\eps}_{\mathbf{F}\cup\{\mathbf{z}\},s}(\mathbf{\hx,q},\tau,\om)\big| : \, \tau\in\mathcal {T}\Big\}$.

\begin{theorem}\label{sourse-stability}
Let $\tau_j\in\C, j=1,2,3$ be three scattering strengths with different principle arguments. For fixed $\mathbf{z}\in\R^2\ba\ov{\Om}$,
assume that
\ben
\left|\big|u^{\infty,\eps}_{\mathbf{F}\cup\{\mathbf{z}\},s}(\mathbf{\hx,q},\tau,\om)\big|-\big|u^{\infty}_{\mathbf{F}\cup\{\mathbf{z}\},s}(\mathbf{\hx,q},\tau,\om)\big|\right|\leq \eps,\quad \mathbf{\hx}\in\mathbb{S}^{\perp}_{\mathbf{q}}, \mathbf{q}\in\mathcal{Q}, \tau\in\mathcal {T}, \om\in\mathbb{W}.
\enn
Then we have
\be\label{Sourse-estimate}
\left|u^{\infty,\eps}_{\mathbf{F},s}(\mathbf{\hx},\om)-u^{\infty}_{\mathbf{F},s}(\mathbf{\hx},\om)\right|\leq c\eps,\quad \mathbf{\hx}\in\mathbb{S}^{\perp}_{\mathbf{q}}, \om\in\mathbb{W}.
\en
for some constant $c>0$ depending only on $\tau_j, j=1,2,3$.
\end{theorem}\\

Difficulties arise for the obstacle scattering problems because of the additional unknown far field pattern $v^{\infty}_{\Om}$ corresponding to the point sources.
For a source point $\mathbf{z}\in\R^2\ba\ov{\Om}$, let $\rho:=dist(\mathbf{z},\Om)$ be the distance from $\mathbf{z}$ to the unknown target $\Om$.
By Theorem \ref{weakinteraction}, $v^{\infty}_{\Om}$ is very weak if $\rho\rightarrow\infty$, and thus
\ben
w^{\infty}_{\Om\cup\{\mathbf{z}\},ss}(\mathbf{\hat{x},d,q},\tau) = u^{\infty}_{\Om,ss}(\mathbf{\hat{x},d}) + \tau e^{-ik_s\mathbf{\hx\cdot z}}\mathbf{q}\cdot\mathbf{\hx}^{\perp} + O\left(\frac{1}{\sqrt{\rho}}\right),\quad \rho\rightarrow\infty
\enn
for all $\mathbf{\hat{x}, d,q}\in\mathbb{S}, \tau\in\C$.
Using the phase retrieval scheme, we wish to approximately reconstruct $u^{\infty}_{\Om,ss}(\mathbf{\hat{x},d})$ from the knowledge of the perturbed phaseless data
$\left|w^{\infty}_{\Om\cup\{\mathbf{z}\},ss}(\mathbf{\hat{x},d,q},\tau)\right|$ with a known error level
\be\label{uinfeps}
\left|\big|w^{\infty,\eps}_{\Om\cup\{\mathbf{z}\},ss}(\mathbf{\hat{x},d,q},\tau)\big|-\big|w^{\infty}_{\Om\cup\{\mathbf{z}\},ss}(\mathbf{\hat{x},d,q},\tau)\big|\right|\leq \eps,\, \mathbf{\hx}\in\mathbb{S}^{\perp}_{\mathbf{q}}, \mathbf{d}\in\mathbb{S}, \mathbf{q}\in\mathcal{Q},\tau\in\mathcal {T}\qquad
\en
uniformly with respect to $\rho>0$.

\begin{theorem}\label{obstacle-stability}
Let $\tau_j\in\C, j=1,2,3$ be three scattering strengths with different principle arguments. Under the measurement error estimate \eqref{uinfeps}, we have
\be\label{uDinfestimate}
\left|u^{\infty,\eps}_{\Om,ss}(\mathbf{\hat{x},d})-u^{\infty}_{\Om,ss}(\mathbf{\hat{x},d})\right|\leq c\eps+O\left(\frac{1}{\sqrt{\rho}}\right), \quad \rho\rightarrow\infty, \,\mathbf{\hx,d}\in\mathbb{S}
\en
for some constant $c>0$ depending only on $\tau_j, j=1,2,3$.
\end{theorem}
\begin{proof}
Define $r_j^{\eps}:=\big|w^{\infty,\eps}_{\Om\cup\{\mathbf{z}\},ss}(\mathbf{\hat{x},d,q},\tau_j)\big|, j=1,2,3$.
By Theorem \ref{weakinteraction}, we have
\be\label{rjeps}
r_j^{\eps}
&=&\left|u^{\infty,\eps}_{\Om,ss}(\mathbf{\hat{x},d}) + \tau_j e^{-ik_s\mathbf{\hx\cdot z}}\mathbf{q}\cdot\mathbf{\hx}^{\perp} + O\left(\frac{1}{\sqrt{\rho}}\right)\right|\cr
&=&\mathbf{q}\cdot\mathbf{\hx}^{\perp}\left|\frac{u^{\infty}_{\Om,ss}(\mathbf{\hat{x},d})e^{ik_s\mathbf{\hx\cdot z}}}{\mathbf{q}\cdot\mathbf{\hx}^{\perp}} + \tau_j + O\left(\frac{1}{\sqrt{\rho}}\right)\right|,\quad \rho\rightarrow \infty
\en
for all $\mathbf{\hx}\in\mathbb{S}^{\perp}_{\mathbf{q}}, \mathbf{d}\in\mathbb{S}, \mathbf{q}\in\mathcal{Q}$.
Here we have used the fact that $\mathbf{q}\cdot\mathbf{\hx}^{\perp}\geq1/2$ for $\mathbf{\hx}\in\mathbb{S}^{\perp}_{\mathbf{q}}$.
Let now
\ben
z_j=-\tau_j,\, j=1,2,3.
\enn
Then the assumption on the strengths implies that the three points $z_j, j=1,2,3$ are not collinear.
Applying Lemma \ref{phaseretrieval}, we have
\ben
&&\left|u^{\infty,\eps}_{\Om,ss}(\mathbf{\hat{x},d})-u^{\infty}_{\Om,ss}(\mathbf{\hat{x},d})+O\left(\frac{1}{\sqrt{\rho}}\right)\right|\cr
&=&\mathbf{q}\cdot\mathbf{\hx}^{\perp}\left|\frac{u^{\infty,\eps}_{\Om,ss}(\mathbf{\hat{x},d})e^{ik_s\mathbf{\hx\cdot z}}}{\mathbf{q}\cdot\mathbf{\hx}^{\perp}}-\frac{u^{\infty}_{\Om,ss}(\mathbf{\hat{x},d})e^{ik_s\mathbf{\hx\cdot z}}}{\mathbf{q}\cdot\mathbf{\hx}^{\perp}}+O\left(\frac{1}{\sqrt{\rho}}\right)\right|\cr
&\leq& c\eps,
\enn
for some constant $c>0$ depending only on $\tau_j, j=1,2,3$. The stability estimate \eqref{uDinfestimate}
now follows by using the triangle inequality.
\end{proof}

Finally, we want to remark that the same estimates in Theorems \ref{sourse-stability} and \ref{obstacle-stability} also hold for the other phaseless far field data with appropriate modification of the observation arcs.

\subsection{Scatterer shape reconstruction}
This subsection is devoted to introduce some direct sampling methods for reconstruction of $\Omega$ by using phaseless shear
far field data $\big|w^{\infty}_{\Omega\cup\{\mathbf{z}\},ss}\big|$. The direct sampling methods proposed
do not need any a priori information of the obstacle.

For some polarization direction $\bf q$, We first introduce two auxiliary functions
\be\label{GA}
G\mathbf {(p,d,q)}&:=&\int_{\mathbb{S}}u^{\infty}_{\Omega,ss}\mathbf{(\hx,d)}e^{ik_{s}\mathbf{\hx\cdot p}}\mathbf{q\cdot \hx ^\perp} ds(\mathbf{\hx}),\quad\,\mathbf{p}\in\R^2,\cr
A\mathbf{(p,q)}&:=&\int_{\mathbb{S}}G\mathbf{(p,d,q)}e^{-ik_s\mathbf{d\cdot\,p}}\mathbf{q\cdot d^\perp} ds(\mathbf{d}),\quad\,\mathbf{p}\in\R^2.
\en
By the well-known Riemann-Lebesgue Lemma, both $G$ and $A$ tend to $0$ as $|\mathbf{p}|\rightarrow\infty$.
In fact, due to the  systematic analysis in \cite{JiLiuXi}, $G$ is a superposition of the Bessel functions.
We thus expect that $G$ (and therefore $A$) decays like Bessel functions as the sampling points away from
the boundary of the scatterer. Then one may look for the scatterers by using the following indicators \cite{JiLiuXi}
with phased far field patterns
\be\label{indicator23}
{\bf{ I_2}(p)=\bigcup_{\mathbf{q}\in\mathcal{Q}}\big|}A{\bf(p,q)\big| \quad\mbox{and}\quad {\bf I_3}(z,d)=\bigcup_{\mathbf{q}\in\mathcal{Q}}\big|}G{\bf (p,d,q)\big|.}
\en
In \cite{JiLiuXi}, it has been shown that the indicator ${\bf I_2}$
has a positive lower bound for sampling points
inside the scatterer, and decays like Bessel functions as the sampling points tend to infinity.
If the size of the scatterer $\Om$ is small enough (compared with the wavelength),
${\bf I_3}$ takes its local maximum at the location of the scatterer.

Consider now the case of phaseless far field measurements.
Using the notations in Theroem 2.3, for all $\mathbf{\hx,d}\in\mathbb{S}$, fixed $\tau_1\in \C\ba\{0\}$ and $\mathbf{z_0}\in\R^2\ba\ov{\Omega}$, we have
\be\label{F}
&&\mathcal {F}_{\bf z_0}(\mathbf{\hx,d,q},\tau_1)\cr
&:=&|w^{\infty}_{\Omega\cup\{\mathbf{z_0}\},ss}(\mathbf{\hx, d, q},\tau_1)|^2-|u^{\infty}_{\Omega,ss}\mathbf{(\hx, d)}|^2-|\tau_1\mathbf{q\cdot \hx^\perp}|^2\cr
&=&\Big|u^{\infty}_{\Omega,ss}\mathbf{(\hx, d)}+\tau_1 e^{-ik_s\mathbf{z_0\cdot\hx}}\mathbf{q\cdot \hx^\perp} + O\Big(\rho^{-1/2}\Big)\Big|^2-|u^{\infty}_{\Omega,ss}\mathbf{(\hx, d)}|^2-|\tau_1\mathbf{q\cdot \hx^\perp}|^2\cr
&=&u^{\infty}_{\Omega,ss}\mathbf{(\hx, d)}\ov{\tau_1}e^{ik_s\mathbf{z_0\cdot\hx}}\mathbf{q\cdot \hx^\perp} +\ov{u^{\infty}_{\Omega,ss}\mathbf{(\hx, d)}}\tau_1 e^{-ik_s\mathbf{z_0\cdot\hx}}\mathbf{q\cdot \hx^\perp} + O\Big(\rho^{-1/2}\Big).
\en
Then we introduce the following
two indicators
\be
\label{Indicator01}{\bf I_{\mathbf{z_0}}}\mathbf{(p,d)}&:=&\sum_{\mathbf{q}\in\mathcal{Q}}\left|\int_{\mathbb{S}} \mathcal {F}_{\bf z_0}(\mathbf{\hx,d,q},\tau_1)\cos[k_s\mathbf{\hx\cdot(p-z_0)}]ds(\mathbf{\hx})\right|^2,\,\mathbf{d}\in\mathbb{S},\,\mathbf{p}\in\R^2,\qquad 
\en
and
\be
\label{Indicator0f}{\bf I_{\mathbf{z_0}}}\mathbf{(p)}&:=&\int_{\mathbb{S}}{\bf I_{\mathbf{z_0}}}\mathbf{(p,d)}ds(\mathbf{d}),\,\quad\mathbf{p}\in\R^2.\quad
\en
Insert \eqref{F} into \eqref{Indicator01}-\eqref{Indicator0f}. Then a straightforward calculation shows that
\ben
&&{\bf I_{\mathbf{z_0}}}\mathbf{(p,d)}=\sum_{\mathbf{q}\in\mathcal{Q}}\left| V_{\mathbf{z_0}}\mathbf{(p,d,q)}+\ov{V_{\mathbf{z_0}}\mathbf{(p,d,q)}} \right|^2+ O\Big(\rho^{-1/2}\Big),\,\mathbf{d}\in\mathbb{S},\,\mathbf{p}\in\R^2,\cr
&&{\bf I_{\mathbf{z_0}}}\mathbf{(p)}=\int_{\mathbb{S}} \sum_{\mathbf{q}\in\mathcal{Q}}\left| V_{\mathbf{z_0}}\mathbf{(p,d,q)}+\ov{V_{\mathbf{z_0}}\mathbf{(p,d,q)}} \right|^2ds(\mathbf{d})+ O\Big(\rho^{-1/2}\Big),\,\mathbf{p}\in\R^2
\enn
with
\ben
V_{\mathbf{z_0}}\mathbf{(p,d,q)}:=\frac{\ov{\tau_1}}{2}\Big[G\mathbf{(z,d,q)}+G\mathbf{(2z_0-p,d,q)}\Big],\quad \,\mathbf{p}\in\R^2, \mathbf{d}\in\mathbb{S}, \mathbf{q}\in\mathcal{Q}.
\enn
Let $\Omega(\mathbf{z_0})$ be the point symmetric domain of $\Omega$ with respect to $\mathbf{z_0}$.
If the size of the scatterer $\Omega$ is small enough, from the properties of $G$, we expect that the indicator
${\bf I_{\mathbf{z_0}} (p, d)}$ takes its local maximum on the locations of $\Omega$ and $\Omega(\mathbf{z_0})$.
For extended scatterer $\Omega$, we expect that
the indicator ${\bf I_{z_0}}$ takes its maximum on or near the boundary $\pa \Omega\cup\pa \Omega(\mathbf{z_0})$.

Note that the indicator ${\bf I_{z_0}(p)}$ (or ${\bf I_{\mathbf{z_0}} (p, d)}$ the case of small scatterers) produces a false scatterer $\Omega(\mathbf{z_0})$.
However, since we have the freedom to choose the point $\mathbf{z_0}$, we can always choose it such that
the false domain $\Omega(\mathbf{z_0})$ located outside our searching domain of interest. One may also overcome this problem by considering another indicator
 ${\bf I_{z_1}(p)}$ (or ${\bf I_{z_1}(p, d)}$ the case of small scatterers)  with $\mathbf{z_1}\in \R^2\ba\ov{\Omega}$ and $\bf z_1\neq z_0$.\\

{\bf Scatterer Reconstruction Scheme One.}
\begin{itemize}{\em
  \item Collect the phaseless data set
       \ben
       \big\{|w^{\infty}_{\Omega\cup\{\mathbf z_0\},ss}(\mathbf{\hx,d,q},\tau)|:\, \mathbf{\hx}\in\mathbb{S}_{q}^{\perp}, \mathbf{q}\in \mathcal{Q}, \mathbf{d}\in\mathbb{S}, \tau\in\{0,\tau_1\}\big\}.
       \enn
  \item Select a sampling region in $\R^{2}$ with a fine mesh $\mathcal {Z}$ containing the scatterer $\Omega$.
  \item Compute the indicator functional ${\bf I_{z_0}(p)}$ (or ${\bf I_{z_0}(p,d)}$ with some fixed $\mathbf{d}\in\mathbb{S}$ in the case of small scatterers) for all sampling points $\mathbf{p}\in\mathcal{Z}$.
  \item Plot the indicator functional ${\bf I_{z_0}(p)}$ (or ${\bf I_{z_0}(p,d)}$  in the case of small scatterers).}
\end{itemize}

Using the Phase Retrieval Scheme proposed in the previous subsection, we can retrieve approximately the phased far field pattern $u_{\Omega,ss}^{\infty}$. Then we have the second scatterer reconstruction algorithm.

{\bf Scatterer Reconstruction Scheme Two.}
\begin{itemize}{\em
  \item Collect the phaseless data set
       \ben
       \big\{|w^{\infty}_{\Omega\cup\{\mathbf z_0\},ss}(\mathbf{\hx,d,q},\tau)|:\, \mathbf{\hx}\in\mathbb{S}_{q}^{\perp}, \mathbf{q}\in \mathcal{Q}, \mathbf{d}\in\mathbb{S}, \tau\in\mathcal{T}\big\}.
       \enn
  \item Use the {\bf Phase Retrieval Scheme} to retrieve approximately the phased far field patterns $u^{\infty}_{\Omega,ss}\mathbf{(\hx, d)}$
  for all $\mathbf{\hx, d}\in\mathbb{S}$.
  \item Select a sampling region in $\R^{2}$ with a fine mesh $\mathcal{Z}$ containing $\Omega$.
  \item Compute the indicator functional ${\bf I_2(p)}$ (or ${\bf I_3(p,d)}$ with fixed $d\in\mathbb{S}$ in the case of small
  scatterers) for all sampling points $\mathbf{p}\in\mathcal {Z}$.
  \item Plot the indicator functional ${\bf I_2(p)}$ (or ${\bf I_3(p,d)}$ in the case of small scatterers).}
\end{itemize}

\subsection{Source support reconstruction}
The uniqueness results discussed before ensure the possibility to reconstruct the unknown objects by stable algorithms. In this section, we investigate the numerical methods for support reconstruction of the source $\bf F$ using phaseless far field data $\big| u^{\infty}_{{\bf F\cup \{z\}},s} \big|$.

Denote by $\Theta$ a finite set with finitely many observation directions as elements. We first introduce an auxiliary function
\be\label{H}
H\mathbf{(p,\hx)}:=\int_{\mathbb{W}}u^{\infty}_{\mathbf{F},s}(\mathbf{\hx},\om)e^{ik_s\mathbf{\hx\cdot p}}d\om, \quad\,\mathbf{p}\in\R^2,\,\mathbf{\hx}\in \Theta.
\en
Clearly,
\be\label{I1behavior1}
H(\mathbf{p}+\alpha\mathbf{\hx^{\perp}, \hx})=H(\mathbf{p,\hx}), \quad \mathbf{p}\in\R^2, \,\alpha\in\R.
\en
This further implies that the functional $H$ has the same value for sampling points in the hyperplane with normal direction $\mathbf{\hx}$.
By the well known Riemann-Lebesgue Lemma, we obtain that $H$ tends to $0$ as $|\mathbf{p}|\rightarrow\infty$.

Recall from \eqref{uinfFs} that the far field pattern has the following representation
\be
\mathbf{u}^{\infty}_{\mathbf{F},s}(\mathbf{\hx},\om)=\int_{\R^2}e^{-ik_s\mathbf{\hx\cdot y}}\mathbf{\hx^{\perp}\cdot F(y)}ds(\mathbf{y}),\quad \mathbf{\hx}\in\mathbb{S},\,\om\in \mathbb{W},
\en
Inserting it into the indicator $H$ defined in \eqref{H}, changing the order of integration, and integrating by parts, we have
\be
H(\mathbf{p,\hx})
&=&\int_{\Omega}\int_{\mathbb{W}}e^{ik_s\mathbf{\hx\cdot (p-y)}} \mathbf{\hx^{\perp}\cdot F(y)}d\om d\mathbf{y} \cr
&=&\int_{\Omega}\frac{S_\mathbf{p}\mathbf{(y,\hx)}}{i\mathbf{\hx\cdot\,(p-y)}}d\mathbf{y},
\en
where $S_\mathbf{p}\in L^{\infty}(\Omega)$ is given by
\ben
S_\mathbf{p}\mathbf{(y,\hx)}:=\mathbf{\hx^{\perp}\cdot F(y)}e^{ik_s\mathbf{\hx\cdot (p-y)}}\Big|^{\om_{max}}_{\om_{min}}.
\enn
This implies that the functional $H$ is a superposition of functions that decays as $1/|\mathbf{\hx\cdot(p-y)}|$ as the sampling point $\mathbf{p}$ goes away from the strip $S_{\Omega}(\mathbf{\hx})$.

For any $\mathbf{\hx}\in \mathbb{S},\,\om\in\mathbb{W},\,\tau\in\C$, we define
\be\label{F1}
\mathcal {K}_{\bf z}(\mathbf{\hx,q},\om,\tau)
&:=&|u^{\infty}_{\bf F\cup{\{z\}},s}(\mathbf{\hx, q},\om,\tau)|^2-|u^{\infty}_{\mathbf{F},s}({\bf\hx}, \om)|^2-|\tau\mathbf{q\cdot \hx^\perp}|^2\cr
&=&|u^{\infty}_{\mathbf{F},s}({\bf\hx}, \om)+\tau e^{-ik_s\mathbf{z\cdot\hx}}\mathbf{q\cdot \hx^\perp}|^2-|u^{\infty}_{\mathbf{F},s}({\bf\hx}, \om)|^2-|\tau\mathbf{q\cdot \hx^\perp}|^2\cr
&=&\Big(u^{\infty}_{\mathbf{F},s}({\bf\hx}, \om)\ov{\tau}e^{ik_s\mathbf{z\cdot\hx}}+\ov{u^{\infty}_{\mathbf{F},s}({\bf\hx}, \om)}\tau e^{-ik_s\mathbf{z\cdot\hx}}\Big)\mathbf{q\cdot \hx^\perp}.
\en
Then, for any fixed $\tau\in \C\ba\{0\}$ and $\mathbf{z_0}\in\R^2\ba\ov{\Omega}$, we introduce the following indicator
\be
\label{Indicator011}{\bf I^{\Theta}_{z_0,S}}(\mathbf{p})&:=&\sum_{\hx\in\Theta,\mathbf{q}\in\mathcal{Q}}\left|\int_{\mathbb{W}} \mathcal {K}_{\bf z_0}(\mathbf{\hx,q},\om,\tau)\cos[k_s\mathbf{\hx\cdot(p-z_0)}]d\om\right|,\mathbf{p}\in\R^2.
\en
Inserting \eqref{F1} into \eqref{Indicator011}, straightforward calculations show that
\ben
{\bf I^{\Theta}_{z_0,S}}(\mathbf{p})=\sum_{\mathbf{\hx}\in\Theta,\mathbf{q}\in\mathcal{Q}}\left| U_{\bf z_0}\mathbf{(p,\hx,q)}+\ov{U_{\bf z_0}\mathbf{(p,\hx,q)}}\right|,\mathbf{p}\in\R^2
\enn
with
\ben
U_{\bf z_0}\mathbf{(p,\hx,q)}:=\frac{\ov{\tau} \mathbf{q\cdot \hx^\perp}}{2}\Big[H(\mathbf{p,\hx})+H(\mathbf{2z_0-p,\hx})\Big],\quad \mathbf{p}\in\R^2, \mathbf{\hx}\in\Theta,\mathbf{q}\in\mathcal{Q}.
\enn
Let $\Omega(\mathbf{z_0})$ be again the point symmetric domain of $\Omega$ with respect to $z_0$.
We expect that the indicator ${\bf I^{\Theta}_{\mathbf{z_0},S}}$ takes its maximum on the locations of $\Omega$ and $\Omega(\mathbf{z_0})$. Similarly to the phaseless obstacle scattering problem, we can
choose $\mathbf{z_0}$ such that the false domain $\Omega(\mathbf{z_0})$ located outside our searching domain of interest or another point $\mathbf{z_1}\in \R^2\ba\ov{\Omega}$ and $\mathbf{z_1\neq z_0}$.
\quad\,\\

{\bf Source Support Reconstruction Scheme One.}
\begin{itemize}{\em
  \item (1). Collect the phaseless data set
         \ben
         \big\{|u^{\infty}_{\mathbf{F\cup\{z_0\}},s}(\mathbf{\hx, q}, \om, \tau)|:\,\mathbf{\hx}\in \Theta\cap\mathbb{S}^{\perp}_{\bf q}, {\mathbf{q}}\in\mathcal{Q},\,\om\in \mathbb{W},\, \tau\in \mathcal \{0, \tau_1\}\}\big\}.
         \enn
  \item (2). Select a sampling region in $\R^{2}$ with a fine mesh $\mathcal {Z}$ containing the source support $\Omega$.
  \item (3). Compute the indicator functional ${\bf I^{\Theta}_{\mathbf{z_0}, S}}(\mathbf{p})$  for all sampling points $\mathbf p\in\mathcal {Z}$.
  \item (4). Plot the indicator functional ${\bf I^{\Theta}_{\mathbf{z_0},S}}(\mathbf{p})$ }.
\end{itemize}\,\quad\\

Using the {\bf Phase Retrieval Scheme} proposed in the previous subsection, we obtain the phased far field pattern $u_{\mathbf{F},s}^{\infty}$. Then we have the second scatterer reconstruction algorithm using the following indicator \cite{AlaHuLiuSun}
\be\label{indicatorM}
{\bf I^{\Theta}_{S}(p)}:=\sum_{\hx\in\Theta}\Big|H\mathbf{(p,\hx)}\Big|,\quad \mathbf{p}\in\R^2.
\en

{\bf Source Support Reconstruction Scheme Two.}
\begin{itemize}{\em
  \item (1). Collect the phaseless data set
        \ben
        \big\{|u^{\infty}_{\mathbf{F\cup{\{z_0\}}},s}(\mathbf{\hx, q}, \tau,\om)|:\,\hx\in\Theta\cap\mathbb{S}^{\perp}_{\bf q}, {\bf q}\in\mathcal{Q},\, \om\in\mathbb{W},\,  \tau\in\mathcal{T}\big\}.
        \enn
  \item (2). Use the {\bf Phase Retrieval Scheme} to obtain the phased far field patterns $u^{\infty}_{\mathbf{F},s}(\mathbf{\hx}, \om)$ for all $\mathbf{\hx}\in\Theta, \, \om\in\mathbb{W}$.
  \item (3). Select a sampling region in $\R^{2}$ with a fine mesh $\mathcal {Z}$ containing $\Omega$.
  \item (4). Compute the indicator functional ${\bf I^{\Theta}_{S}(p)}$ for all sampling points $\mathbf{p}\in\mathcal {Z}$.
  \item (5). Plot the indicator functional ${\bf I^{\Theta}_{S}(p)}$.}
\end{itemize}

\section{Numerical experiments}\label{Numericals}
Now we present a variety of numerical examples in two dimensions to illustrate the applicability
and effectiveness of our sampling methods. In the simulations, we use $0.05$ as the sampling space. If not otherwise stated, we add $10\%$ noise. We take $\tau=1$ for the indicators ${\bf I_{z_0}}$ and ${\bf I_{z_0, S}^{\Theta}}$, while $\tau=\pm 0.5, 0.5i$ are used in the {\bf Phase Retrieval Scheme}.
\subsection{Phaseless inverse scattering problems}
There are totally four groups of numerical tests to be considered, and they are
respectively referred to as {\bf ${\bf I_{z_0}(p)}$-Big, ${\bf I_{z_0}(p,d)}$-Small, PhaseRetrieval,
{\bf $\bf I_{2}-$Big$\bf +I_{3}-$Small}}. The boundaries of the scatterers used in our numerical experiments are parameterized as follows.
\begin{small}
\begin{eqnarray}
\label{circle}&\mbox{\rm Circle:}&\quad \mathbf{x}(t)\ =(a,b)+0.1\ (\cos t, \sin t),\quad 0\leq t\leq2\pi,\\
\label{kite}&\mbox{\rm Kite:}&\quad \mathbf{x}(t)\ =\ (\cos t+0.65\cos 2t-0.65, 1.5\sin t),\quad 0\leq t\leq2\pi,
\end{eqnarray}
\end{small}
with $(a,b)$ being the location.

The direct problem is solved  by the boundary integral
equation method \cite{JiLiuXi}.
Define $\theta_l:=2\pi l/N,\,l=0,1,\cdots,N-1$, let ${\bf d}_l=(\cos\theta_l,\sin\theta_l)^T$ and ${\bf\hx}_j=(\cos\theta_j, \sin\theta_j)^T$ for $j,l=0,1,\cdots, N-1$. In our simulations, we compute the far field patterns $w_{\Om\cup\{\mathbf z_0\},ss}^\infty({\bf\hx}_j, {\bf d}_l, {\bf q},\tau)$, ${\bf q\in\mathcal{Q}},\, j,l=0,1,\cdots, N-1$,
for $N$ equidistantly distributed incident directions and $N$ observation directions.
We further perturb this data by random noise
\ben
&&\Big|w_{\Om\cup\{\mathbf z_0\},ss}^{\infty,\delta}({\bf\hx}_j, {\bf d}_l,{\bf q},\tau)\Big| \cr
&=& |w_{\Om\cup\{\mathbf z_0\},ss}^{\infty}({\bf\hx}_j, {\bf d}_l,{\bf q},\tau)| (1+\delta*e_{rel}), \quad {\bf q\in\mathcal{Q}},\, j,l=0,1,\cdots, N-1,
\enn
where $e_{rel}$ is a uniformly distributed random number in the open interval $(-1,1)$.
The value of $\delta$ used in our code is the relative error level.
We also consider absolute error in {\bf Example PhaseRetrieval}. In this case, we perturb the phaseless data
\ben
&&\Big|w_{\Om\cup\{\mathbf z_0\},ss}^{\infty,\delta}({\bf\hx}_j, {\bf d}_l,{\bf q},\tau)\Big|\cr
&=& \max\Big\{0,\,|w_{\Om\cup\{\mathbf z_0\},ss}^{\infty}({\bf\hx}_j, {\bf d}_l,{\bf q},\tau)|+\delta*e_{abs}\Big\},\quad{\bf q\in\mathcal{Q}},\, j,l=0,1,\cdots, N-1,
\enn
where $e_{abs}$ is again a uniformly distributed random number in the open interval $(-1,1)$.
Here, the value $\delta$ denotes the total error level in the measured data.

In the following experiments, we take $N=512$,  the Lam$\acute{\mbox{e}}$ constants $\lambda=1$ and $\mu=1$,
and the circular frequency $\om=8\pi$.\\

\textbf{Example ${\bf I_{z_0}(p)}$-Big}. This example checks the validity of our method for scatterers with
different source points. We consider the benchmark example with a kited domain.
Figure \ref{NDSKite} shows the results with three source points $\mathbf{z_0}=(2,4)^{T}, (4,4)^{T}$
and $(12,12)^{T}$.
As expected, the indicator ${\bf I_{z_0}(p)}$ takes a large value on $\pa \Om\cup\pa \Om(\mathbf{z_0})$, where $\Om(\mathbf{z_0})$
is the symmetric domain of $\Om$ with respect to the source point $\bf z_0$. The symmetric domain of $\mathbf{z_0}=(12,12)$
is outside of the sampling space. Note that $\Om({\bf z_0})$ changes as the source point $\mathbf{z_0}$ changes,
thus it is very easy to pick the correct domain $\Om$ by considering the indicator ${\bf I_{\mathbf{z_0}}}$
with different source points, or we can just choose $\mathbf{z_0}$ far enough.
As shown in Figures \ref{NDSKite}, the left hand kite should be the one desired. \\

\begin{figure}[htbp]
  \centering
  \subfigure[\textbf{$\mathbf{z_0}=(2,4)^{T}$.}]{
    \includegraphics[width=1.5in,height=1.4in]{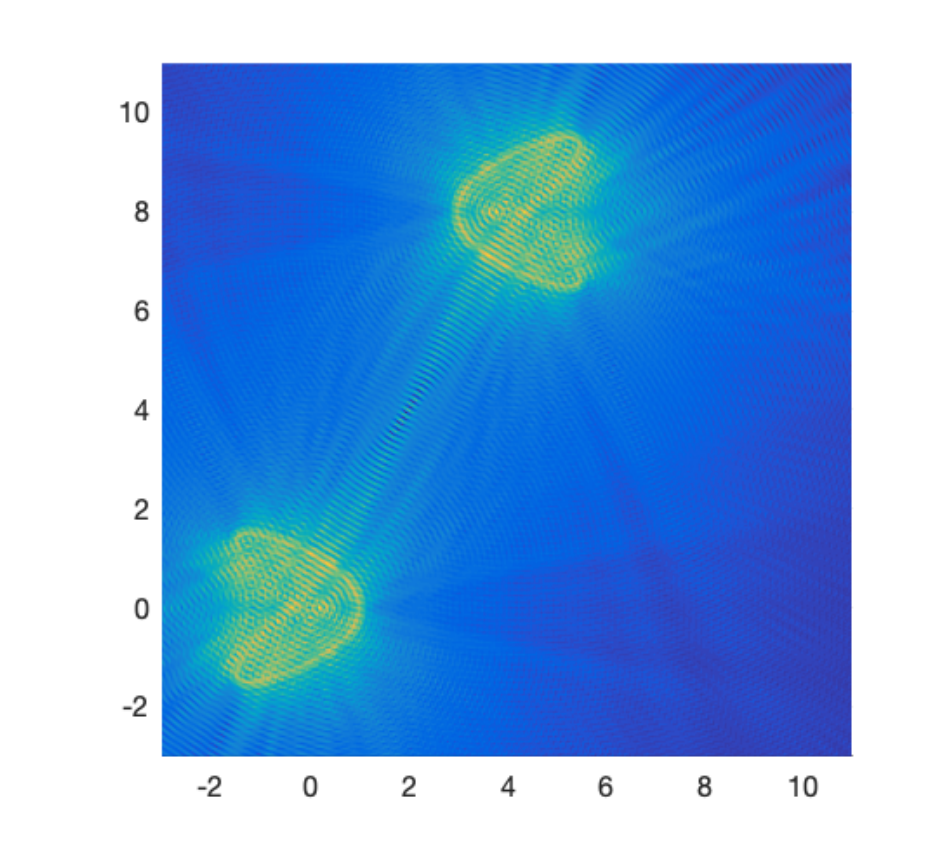}}
  \subfigure[\textbf{$\mathbf{z_0}=(4,4)^{T}$.}]{
    \includegraphics[width=1.5in,height=1.4in]{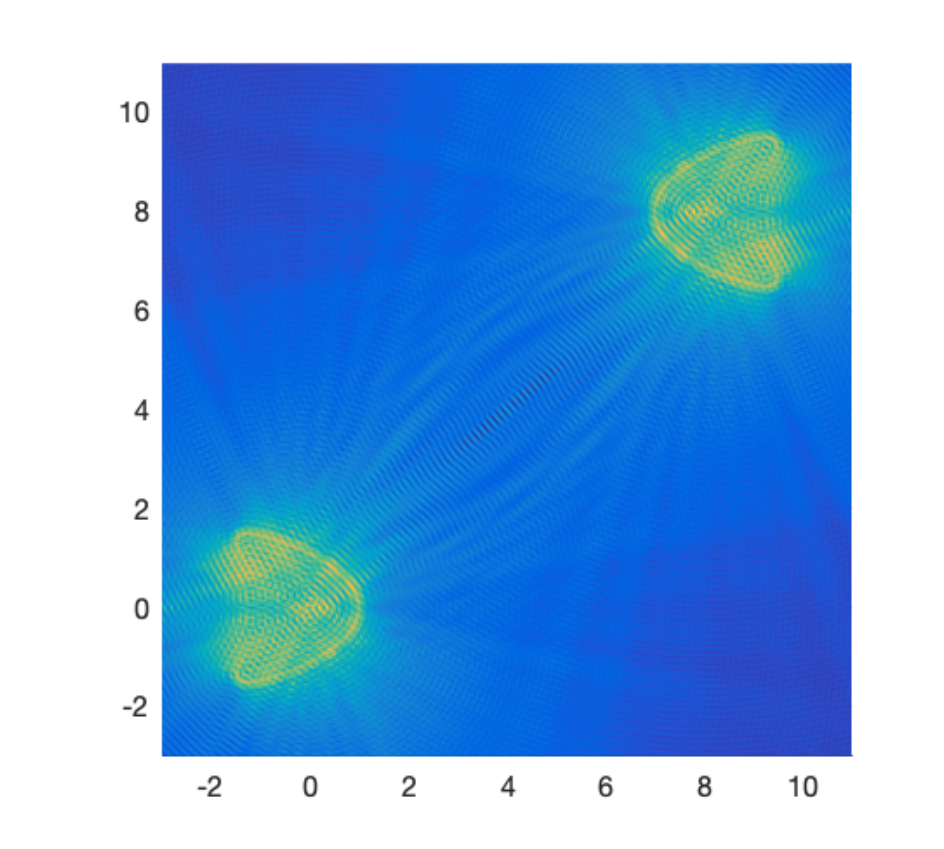}}
  \subfigure[\textbf{$\mathbf{z_0}=(12,12)^{T}$.}]{
    \includegraphics[width=1.5in,height=1.4in]{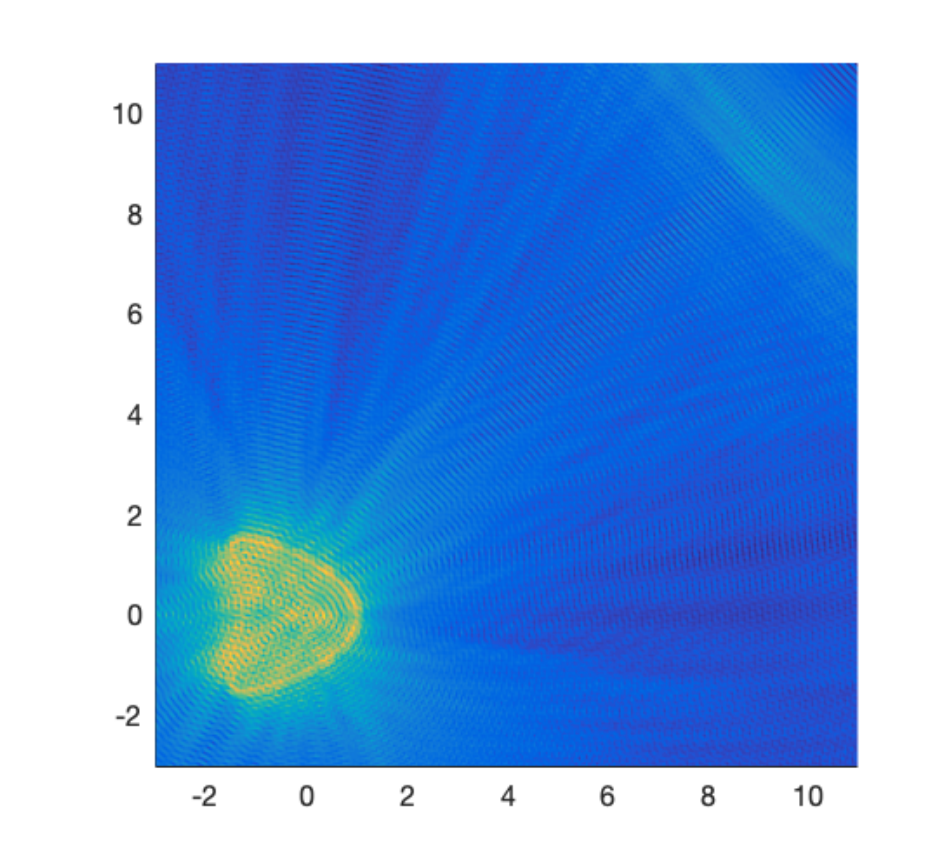}}
\caption{{\bf Example ${\bf I_{z_0}(p)}$-Big.}\, Reconstruction of kite shaped domain with different source points.}
\label{NDSKite}
\end{figure}

\textbf{Example ${\bf I_{z_0}(p,d)}$-Small}. In this example, the scatterer is a combination of two
mini disks with radius $0.1$, one centered at
$(a,b)=(3,3)$ and the other at $(a,b)=(1,1)$.  Figure \ref{NDSsmall} shows
the reconstructions using ${\bf I_{z_0}(p,d)}$, ${\bf d}=(1,0)^{T}$ with the same source points as in
the \textbf{Example ${\bf I_{z_0}(p)}$-Big}.
\begin{figure}[htbp]
  \centering
  \subfigure[\textbf{$\mathbf{z_0}=(2,4)^T$.}]{
    \includegraphics[width=1.55in,height=1.4in]{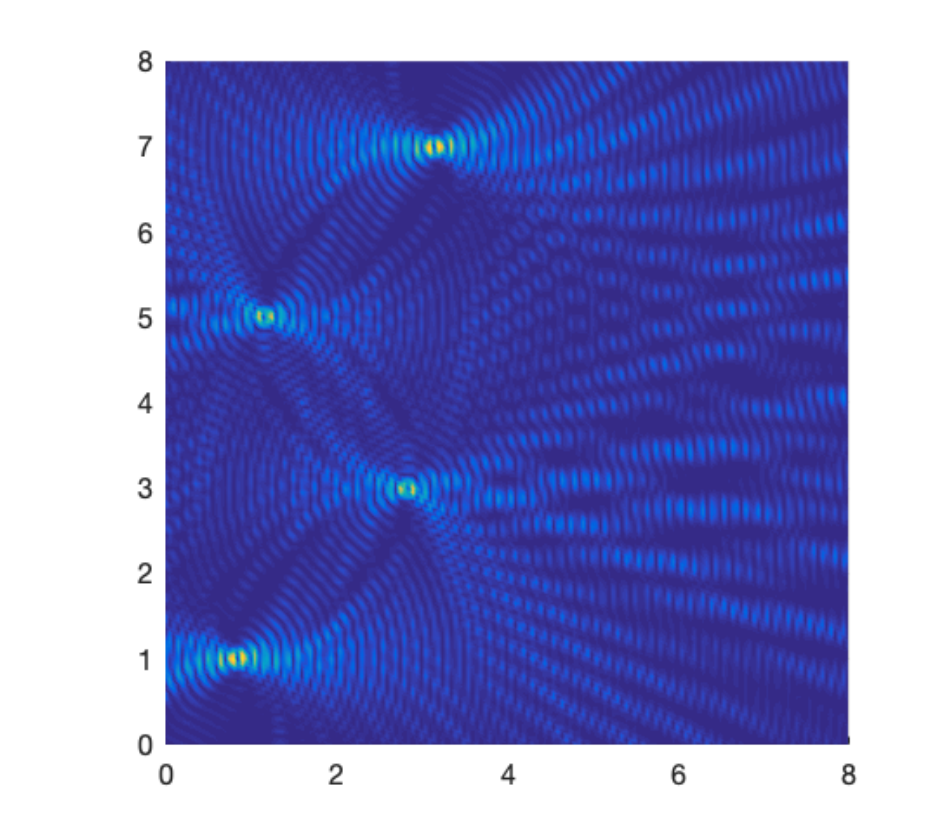}}
  \subfigure[\textbf{$\mathbf{z_0}=(4,4)^T$.}]{
    \includegraphics[width=1.5in,height=1.4in]{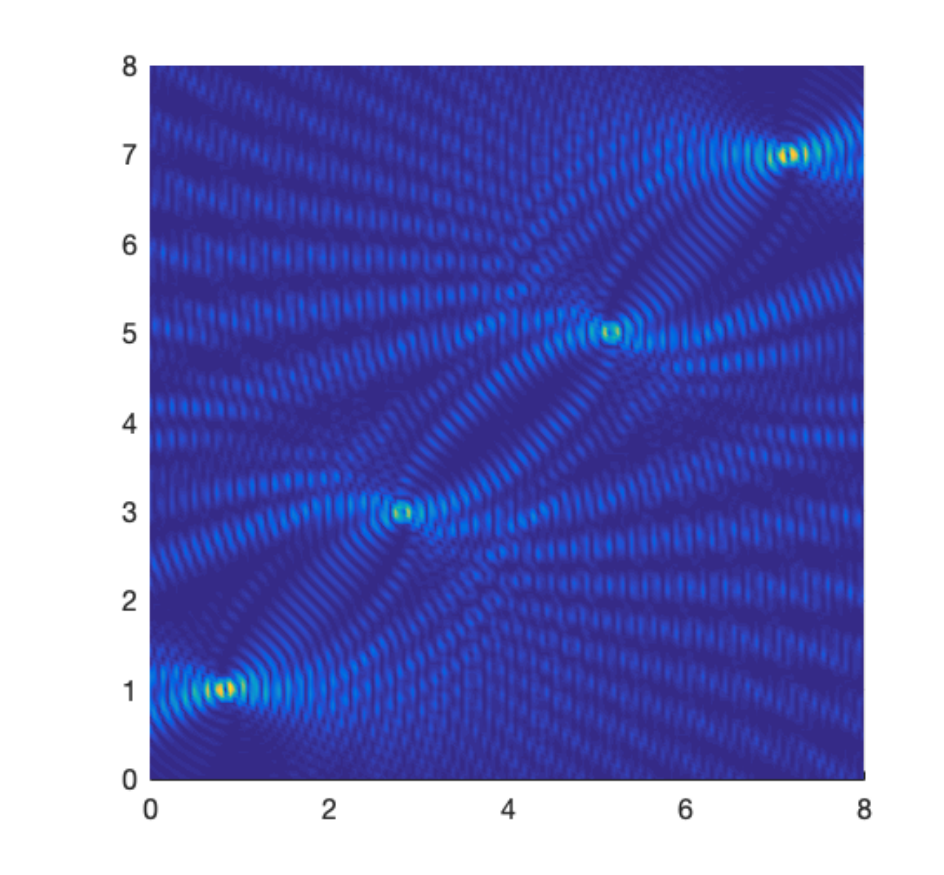}}
  \subfigure[\textbf{$\mathbf{z_0}=(12,12)^T$.}]{
    \includegraphics[width=1.43in,height=1.4in]{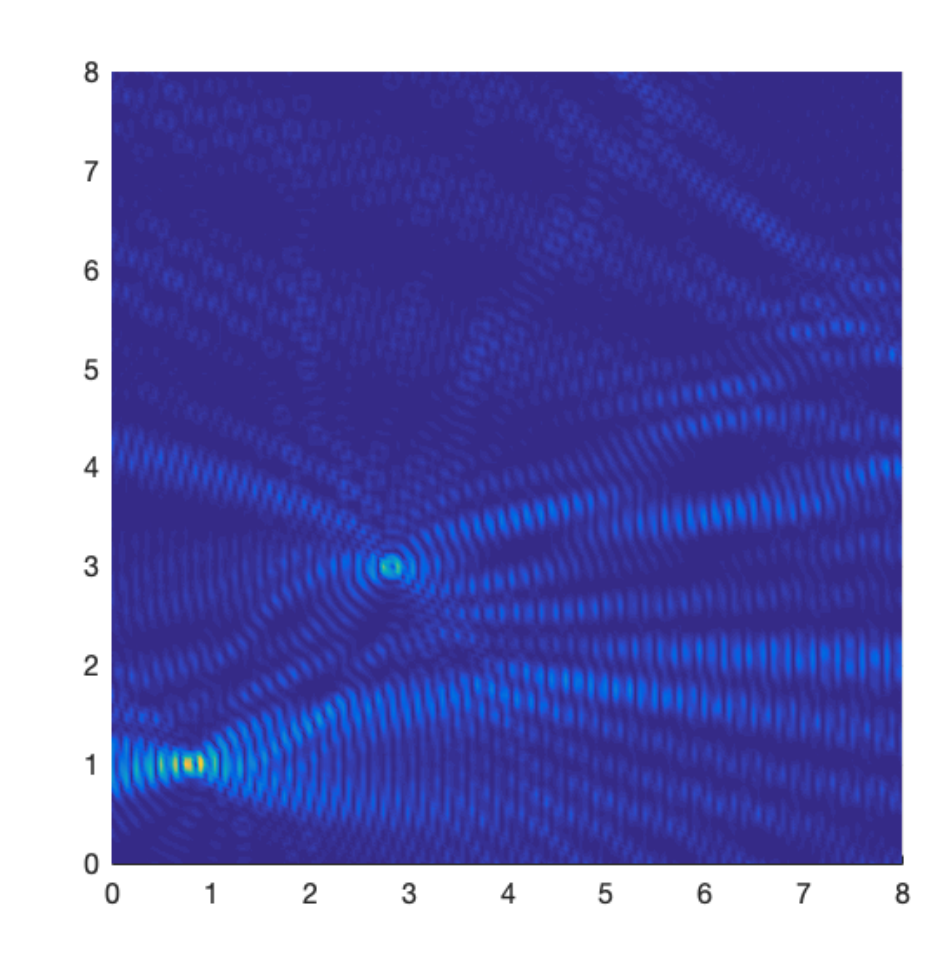}}
\caption{{\bf Example ${\bf I_{z_0}(p,d)}$-Small.}\, Reconstruction of two small disks
by using ${\bf I_{z_0}(p,d)}$with different source points. Here, $\mathbf{d}=(1,0)^T$.}
\label{NDSsmall}
\end{figure}

\textbf{Example PhaseRetrieval}.
This example is designed to check the phase retrieval scheme proposed in Subsection \ref{PR}.
The underlying scatterer is chosen to be a kite shaped domain.
For comparison, we consider the real part of far field pattern at a fixed incident direction $\mathbf{d}=(1,0)^T$.
Figure \ref{zero} shows the results without measurement noise by using three different source
points $(2,4)^T, \,(4,4)^T$ and $(12,12)^T$. In particular, the source point
$(2,4)^T$ is very close to the kite shaped domain. However, Figure \ref{zero}(a) shows that the
multiple scattering is very week. Of course, Figures \ref{zero}(b)-(c) show that the interaction
between the source point and the kite shaped domain decreases as the source point away from the target.
Figures \ref{relative}-\ref{absolute} show the results with relative error and absolute error
considered, respectively.
We find that our phase retrieval scheme is quite robust with respect to noise. This also verifies the theory
provided in Theorem \ref{obstacle-stability}.\\

\begin{figure}[htbp]
  \centering
  \subfigure[\textbf{${\bf z_0}=(2,4)^T$.}]{
    \includegraphics[width=1.5in,height=1in]{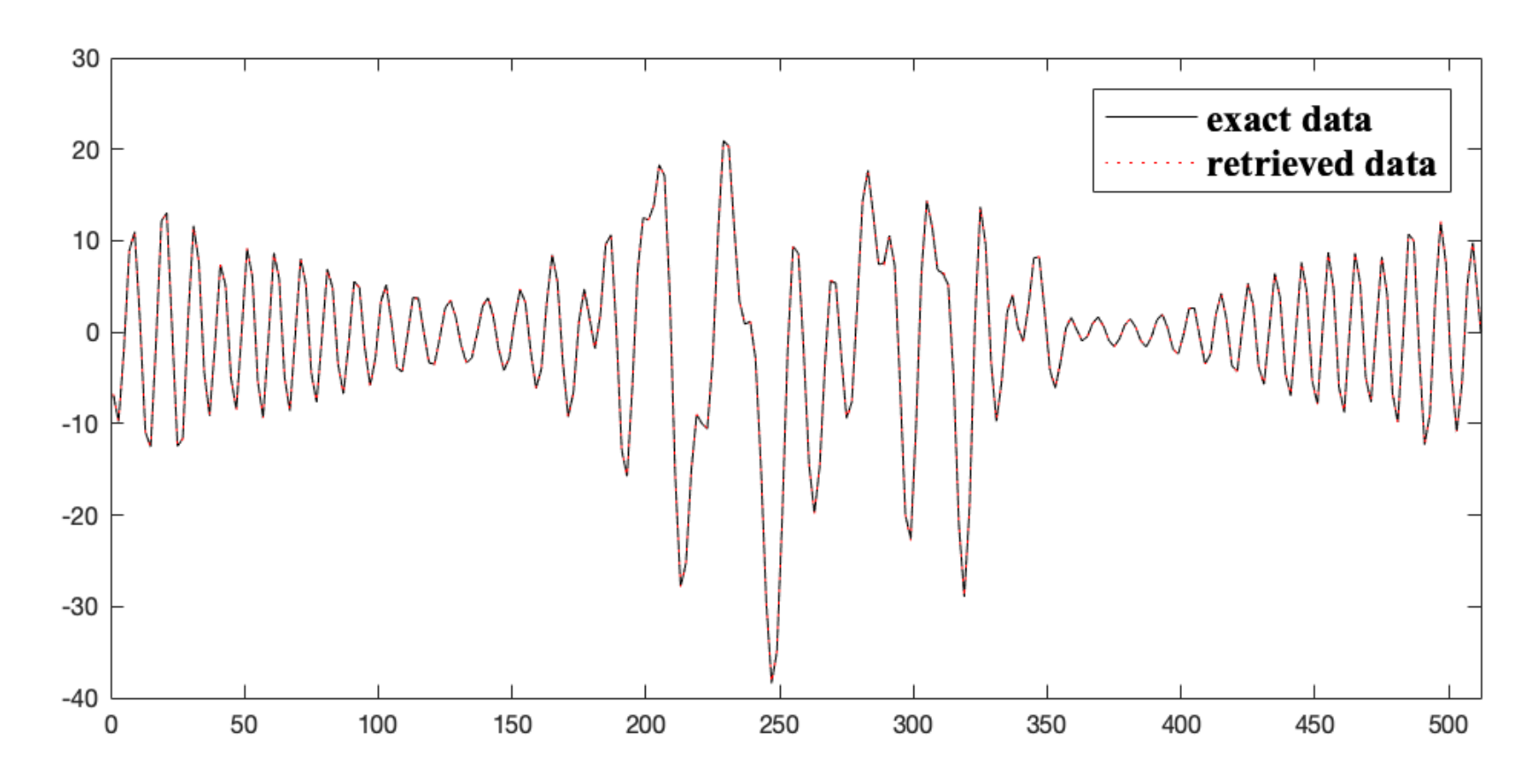}}
  \subfigure[\textbf{${\bf z_0}=(4,4)^T$.}]{
    \includegraphics[width=1.5in,height=0.99in]{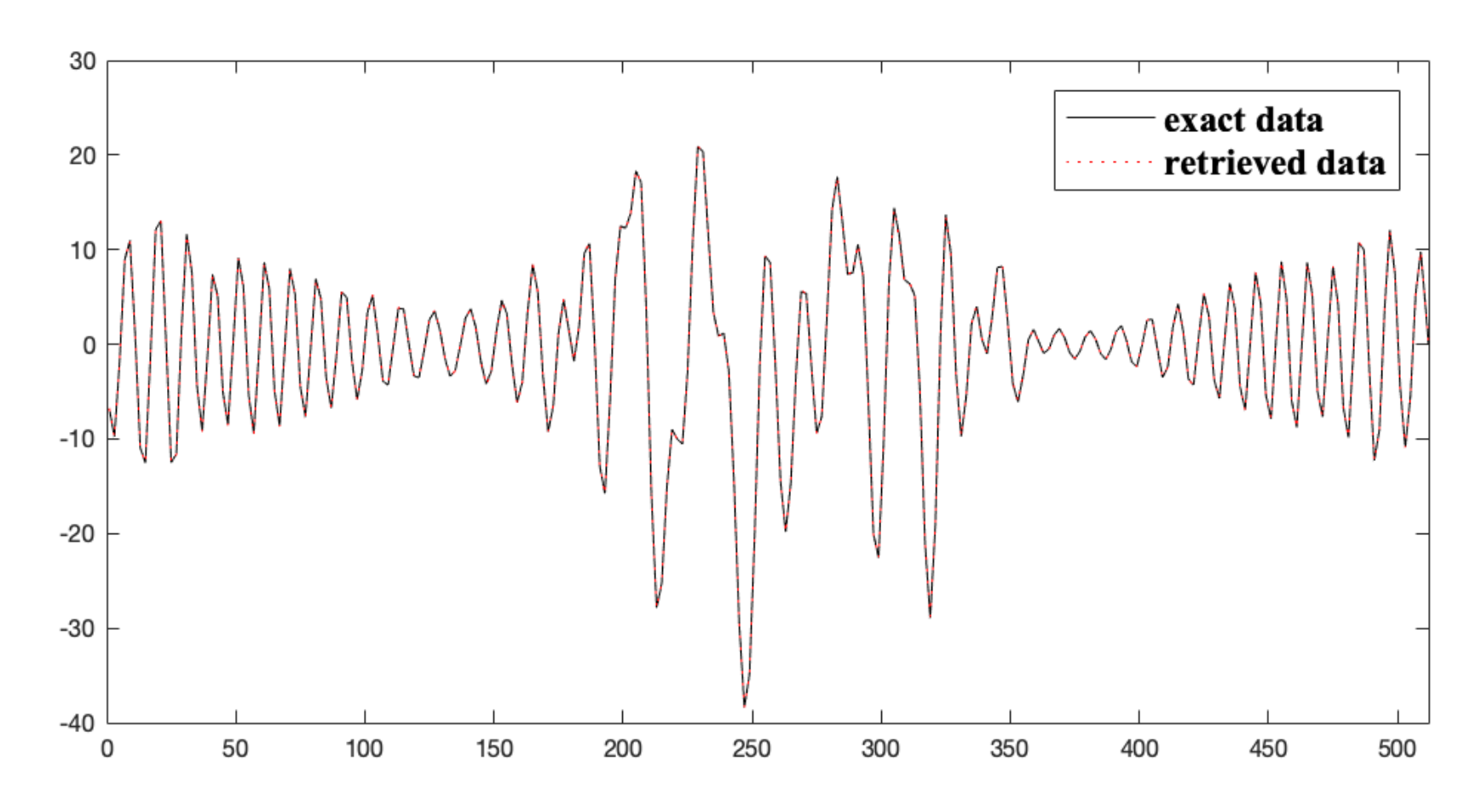}}
  \subfigure[\textbf{${\bf z_0}=(12,12)^T$.}]{
    \includegraphics[width=1.5in,height=0.99in]{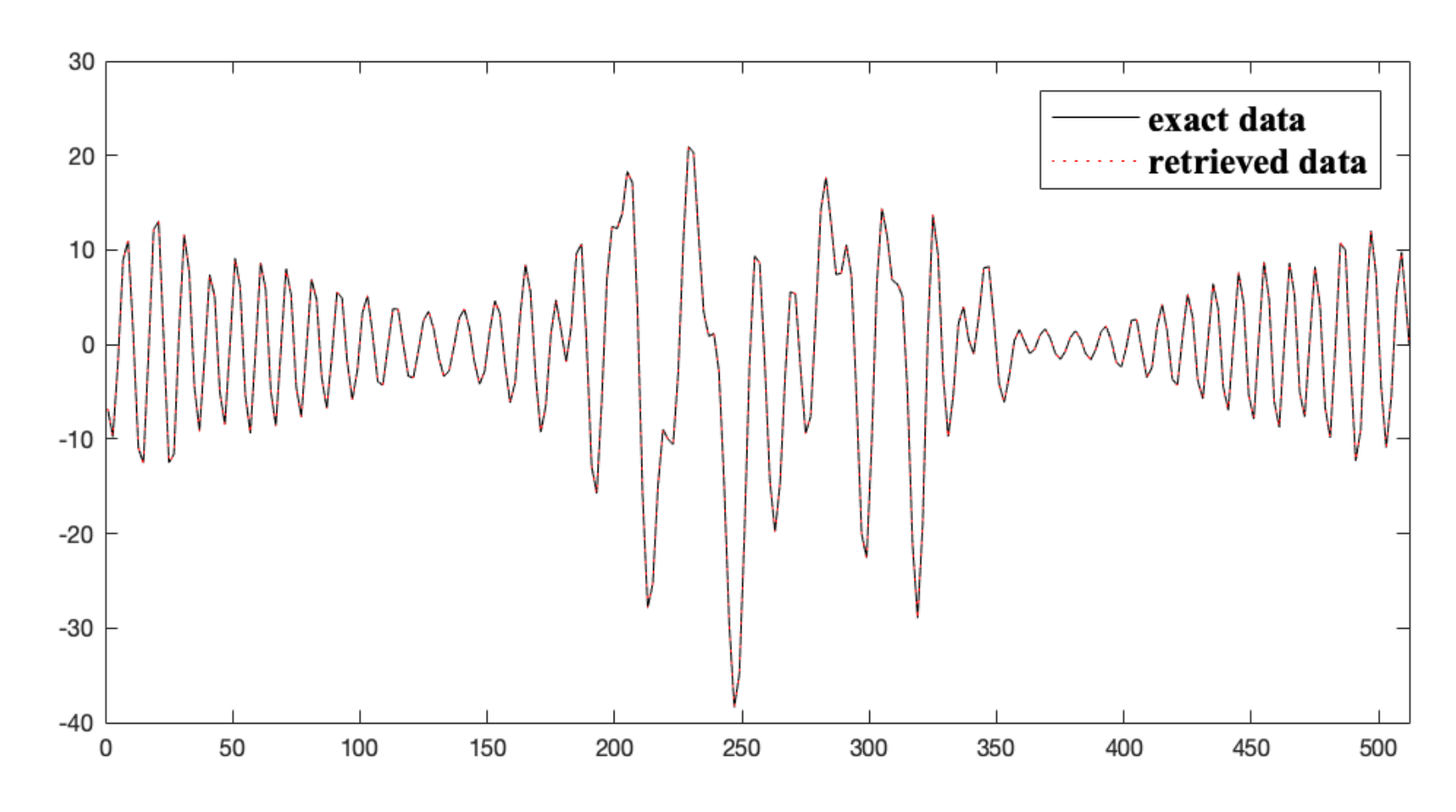}}
\caption{{\bf Example PhaseRetrieval.}\, Phase retrieval for the real part of the far field pattern without error using different source points at a fixed direction $\mathbf{d}=(1,0)^T$.}
\label{zero}
\end{figure}

\begin{figure}[htbp]
  \centering
  \subfigure[\textbf{$0$  noise.}]{
    \includegraphics[width=1.5in,height=1in]{pics/PRerror24noise0.pdf}}
  \subfigure[\textbf{$10\%$ noise.}]{
    \includegraphics[width=1.5in,height=0.99in]{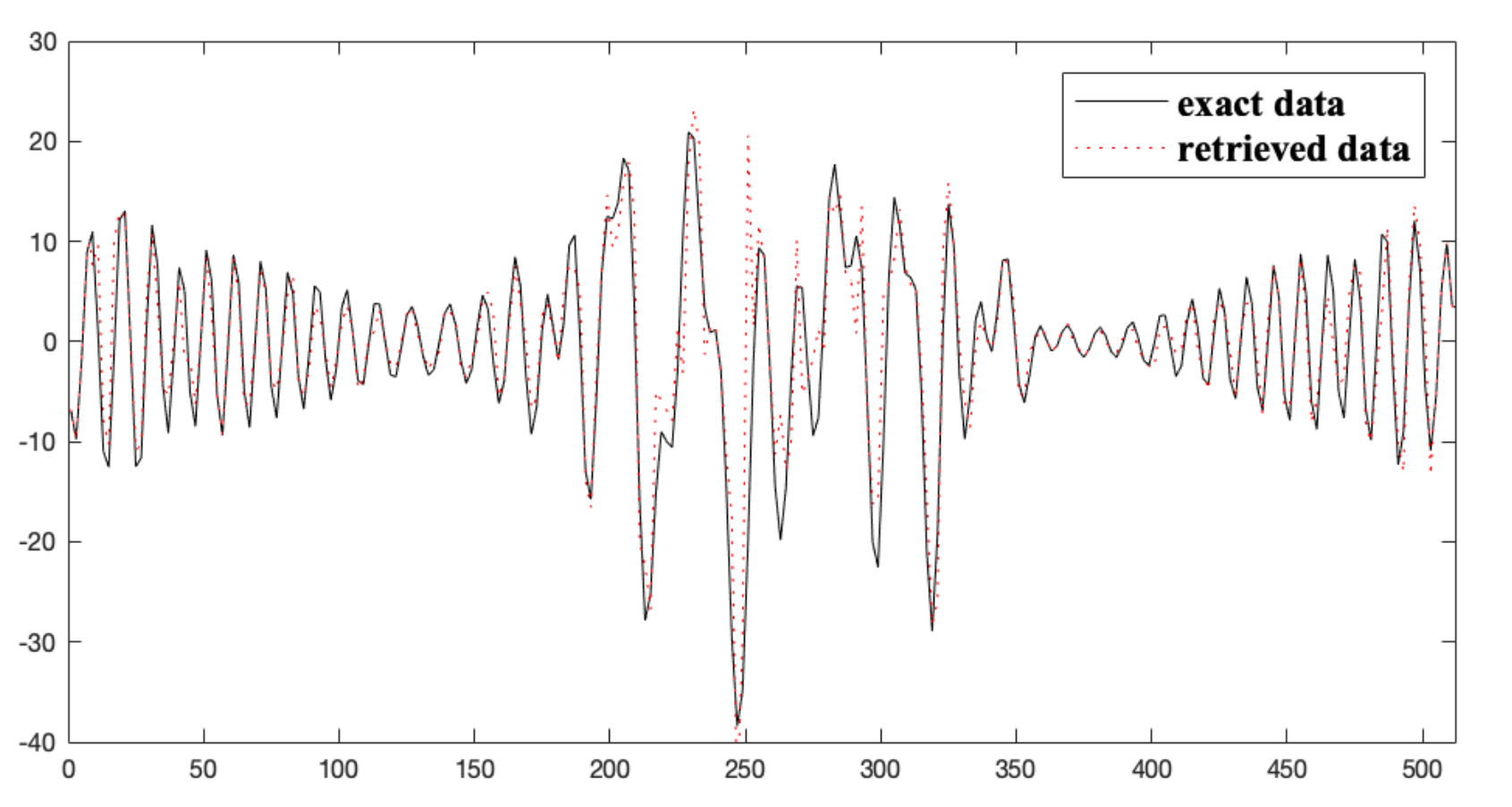}}
  \subfigure[\textbf{$30\%$ noise.}]{
    \includegraphics[width=1.5in,height=0.99in]{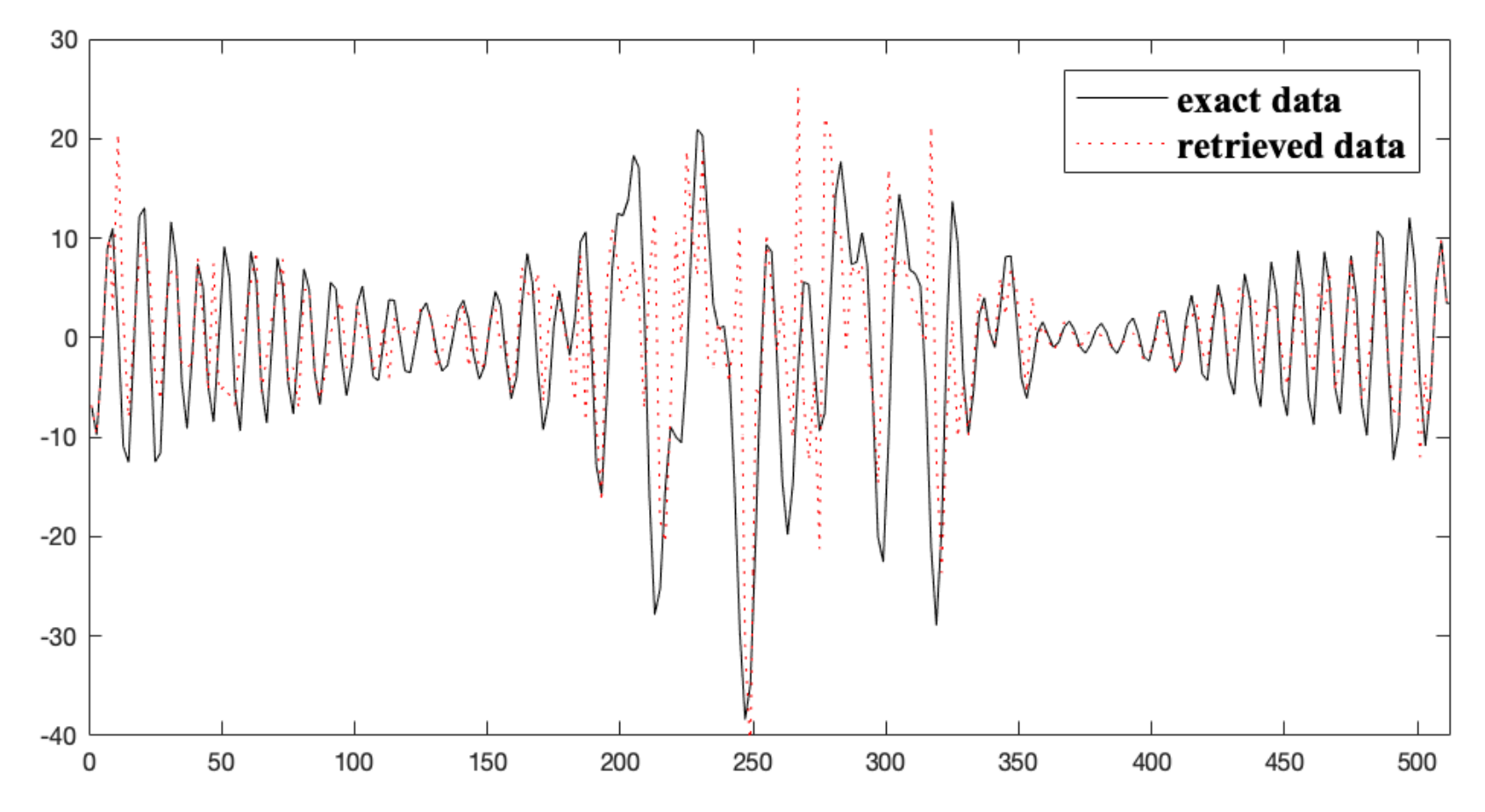}}
\caption{{\bf Example PhaseRetrieval.}\, Phase retrieval for the real part of the far field pattern with different relative error and ${\bf z_0}=(2,4)^T$ at a fixed direction $\mathbf{d}=(1,0)^T$.}
\label{relative}
\end{figure}

\begin{figure}[htbp]
  \centering
  \subfigure[\textbf{$0$ noise.}]{
    \includegraphics[width=1.5in,height=1in]{pics/PRerror24noise0.pdf}}
  \subfigure[\textbf{$0.1$ noise.}]{
    \includegraphics[width=1.5in,height=0.99in]{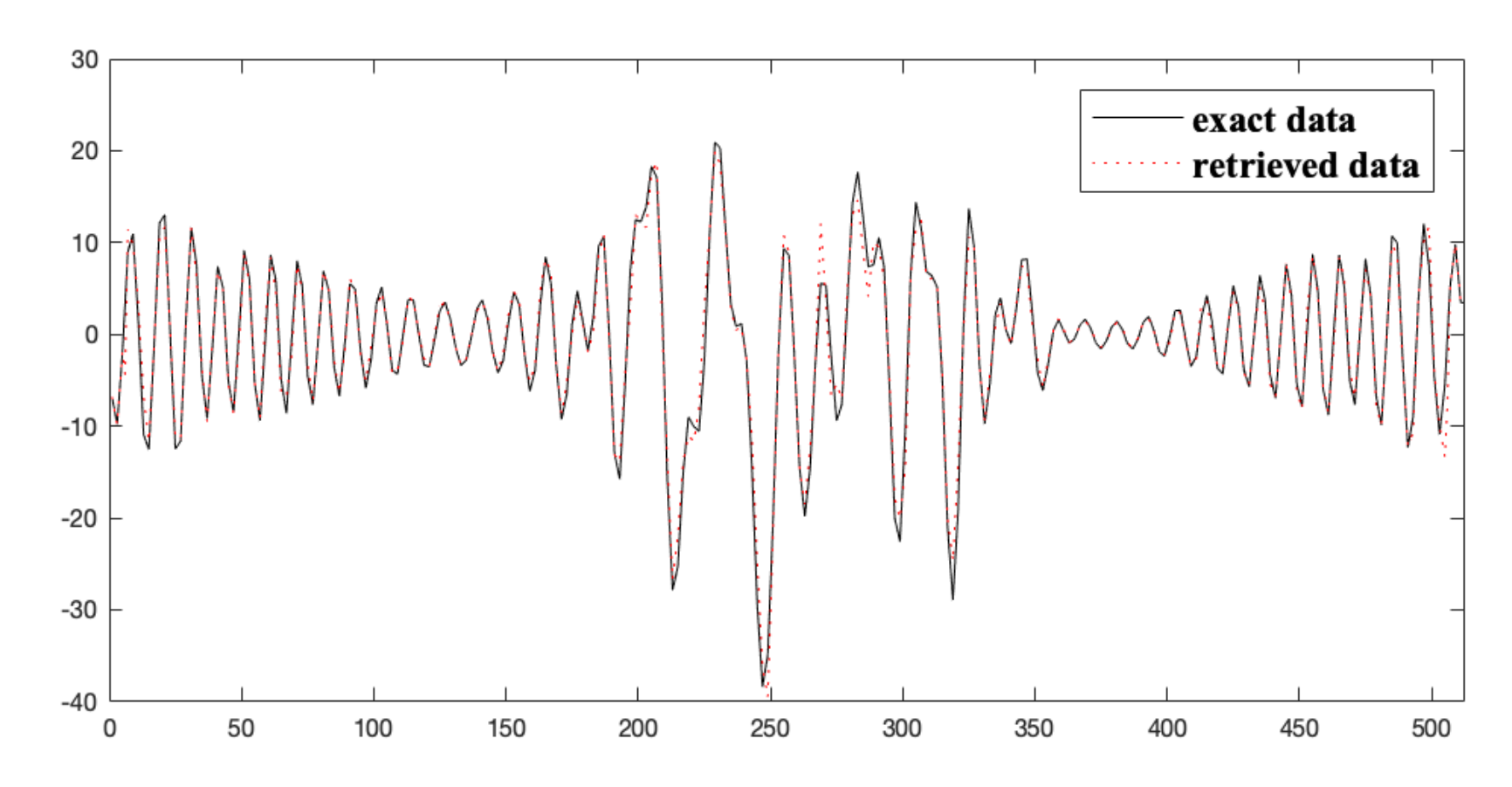}}
  \subfigure[\textbf{$0.3$ noise.}]{
    \includegraphics[width=1.5in,height=0.99in]{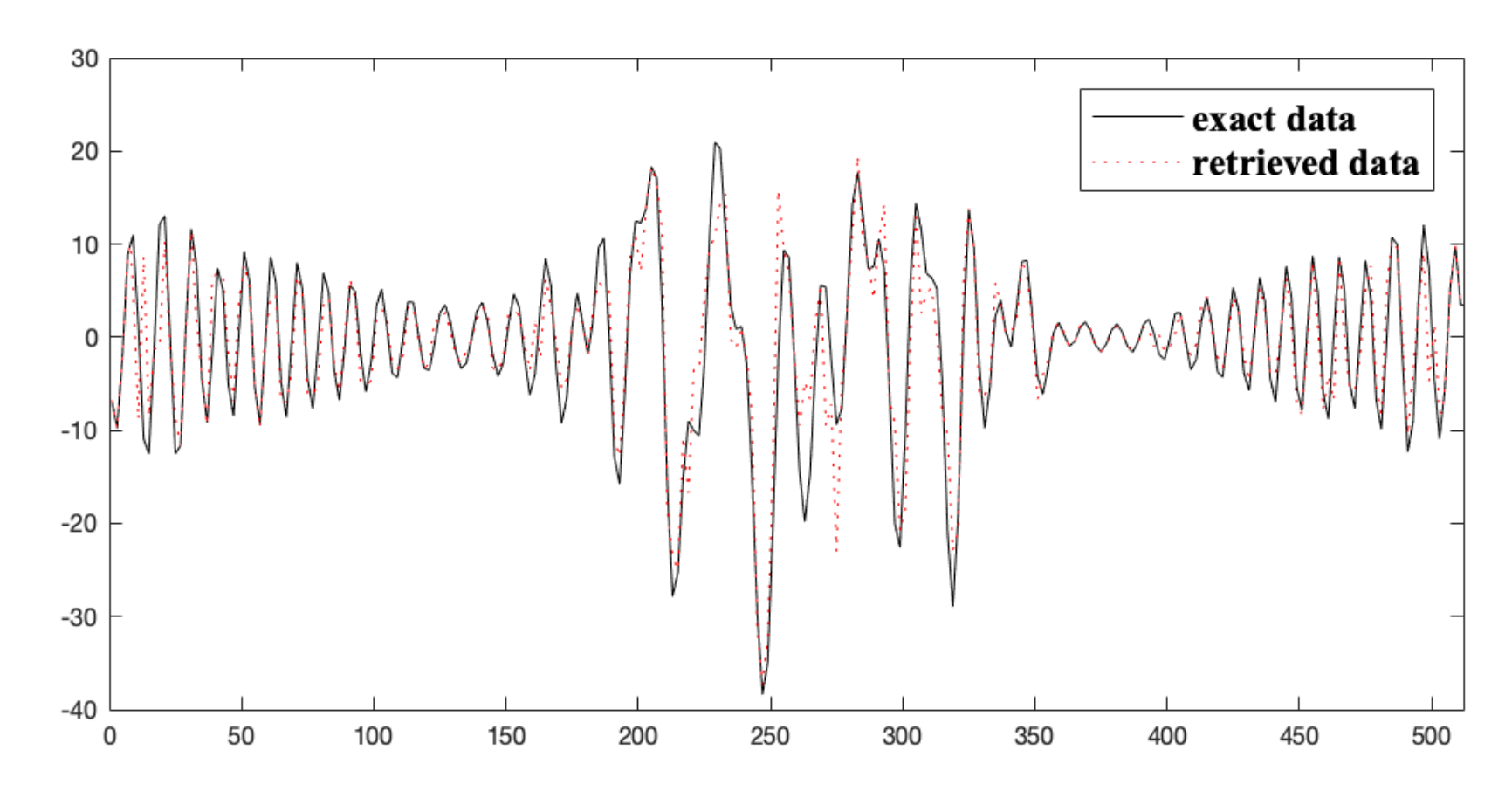}}
\caption{{\bf Example PhaseRetrieval.}\, Phase retrieval for the real part of the far field pattern with different absolute error and ${\bf z_0}=(2,4)^T$ at a fixed direction $\mathbf{d}=(1,0)^T$.}
\label{absolute}
\end{figure}

\textbf{Example {\bf $\bf I_{2}-$Big$\bf +I_{3}-$Small}}.
In Figures \ref{PRexample}, the scatterers are the same as the \textbf{Example ${\bf I_{z_0}(p)}$-Big} and \textbf{Example ${\bf I_{z_0}(p,d)}$-Small}.
We choose the same source point ${\bf z_0}=(2,4)^T$, no false domain appears in the reconstructions. Figures \ref{PRexample}(a) uses all the incident directions, while
\ref{PRexample}(b) uses one incident direction $\mathbf{d}=(1,0)^T$.

\begin{figure}[htbp]
  \centering
  \subfigure[\textbf{Kite.}]{
    \includegraphics[width=1.8in,height=1.6in]{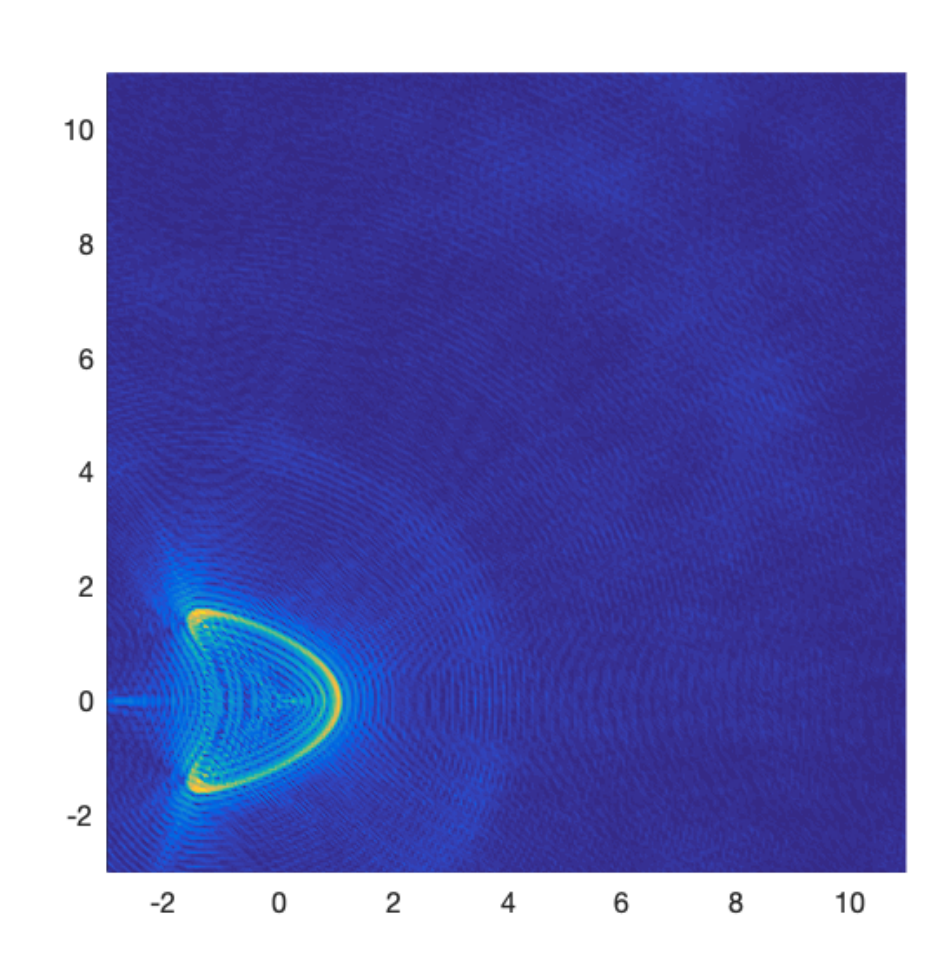}}
  \subfigure[\textbf{Two circles.}]{
    \includegraphics[width=2.2in,height=1.6in]{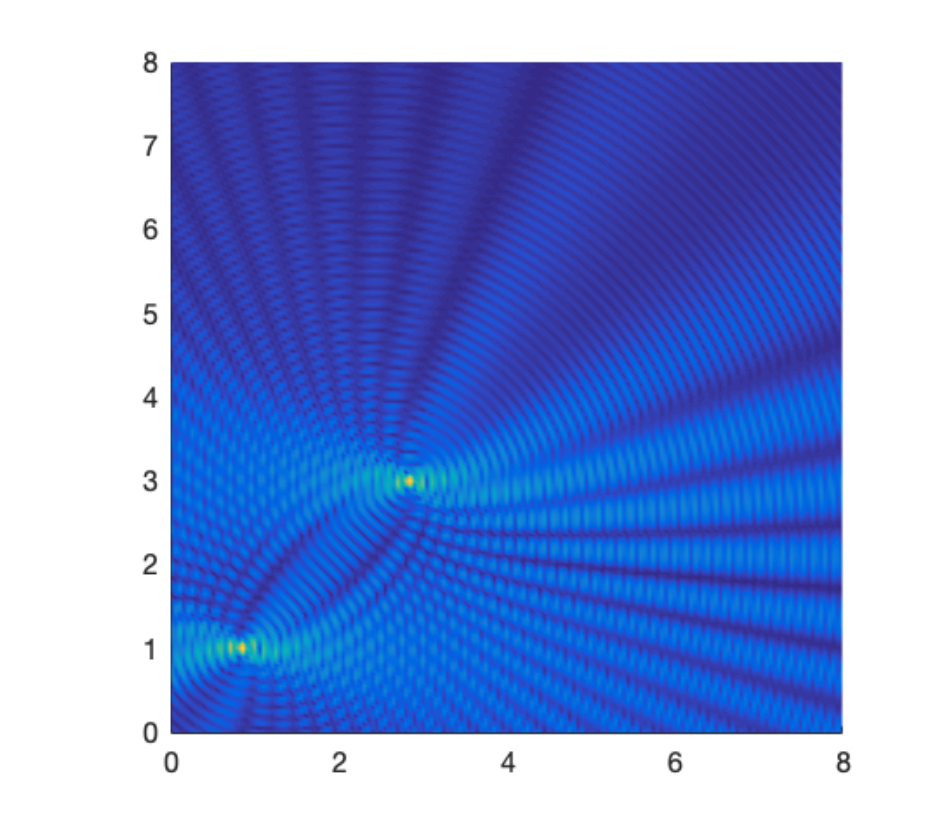}}
\caption{{\bf $\bf I_{2}-$Big$\bf +I_{3}-$Small}\, Reconstruction of different shapes.}
\label{PRexample}
\end{figure}

\subsection{Phaseless inverse source problems}
The forward problems are computed the same as in \cite{AlaHuLiuSun}. In all examples, for ${\bf\hx} \in \Theta$, we consider multiple frequency far field
data $u^\infty_{\mathbf{F},s}({\bf\hx},k_j),\quad j=1,\cdots,N,$ where $N=20,k_{min}=0.5,k_{max}=20$ such that  $k_j=(j-0.5)\Delta k, \Delta k=\frac{k_{max}}{N}$.

Then the phaseless data are stored in the matrices $M_{\bf F,z_0}=\Big(|u_{{\bf F\cup\{z_0\}},s}^{\infty}({\bf\hx}, k_j,\tau)|\Big), \hx\in\Theta, j=1,\cdots,N$.
We further perturb $M_{\bf F, z_0}$ by random relative error and absolute error as before.  Three shapes are considered: a rectangle given by $(1, 2) \times(1,1.6)$,
a L-shaped domain given by  $(0,2) \times(0,2) \setminus (1/16, 2) \times(1/16,2)$ and an equilateral triangle with vertices $(-2,0), (1,0),(-1/2,3/2\sqrt{3})$.
Here, we use
\ben
{\bf F=(x^2+y^2+5, x^2-y^2+5)}^T.
\enn

 \textbf{${\bf I^{\Theta}_{z_0,S}}$ with one and two observation directions} We first consider the case of one observation using different $\mathbf{z_0}$. The support of
$\bf F$ is the rectangle. In Fig. \ref{ObservationoneS}, we plot the indicators using $\mathbf{\hx}= (0,1)^T$ and three source points $\mathbf{z_0} = (4,1.3)^T, (4,4)^T$ and $(12,12)^T$. The picture clearly shows that the source and its the point symmetric domain (with respect to $\bf z_0$) lies in a strip, which is perpendicular to the observation direction.
Next we consider two observation directions $(1,0)^T$ and $(0,1)^T$, we plot the indicators in Fig. \ref{ObservationtwoS}.  Since the observation directions are perpendicular to each other, the strips are perpendicular to each other too.\\

\begin{figure}[htbp]
  \centering
  \subfigure[\textbf{${\bf z_0}=(4,1.3)^T$.}]{
    \includegraphics[width=1.5in,height=1.5in]{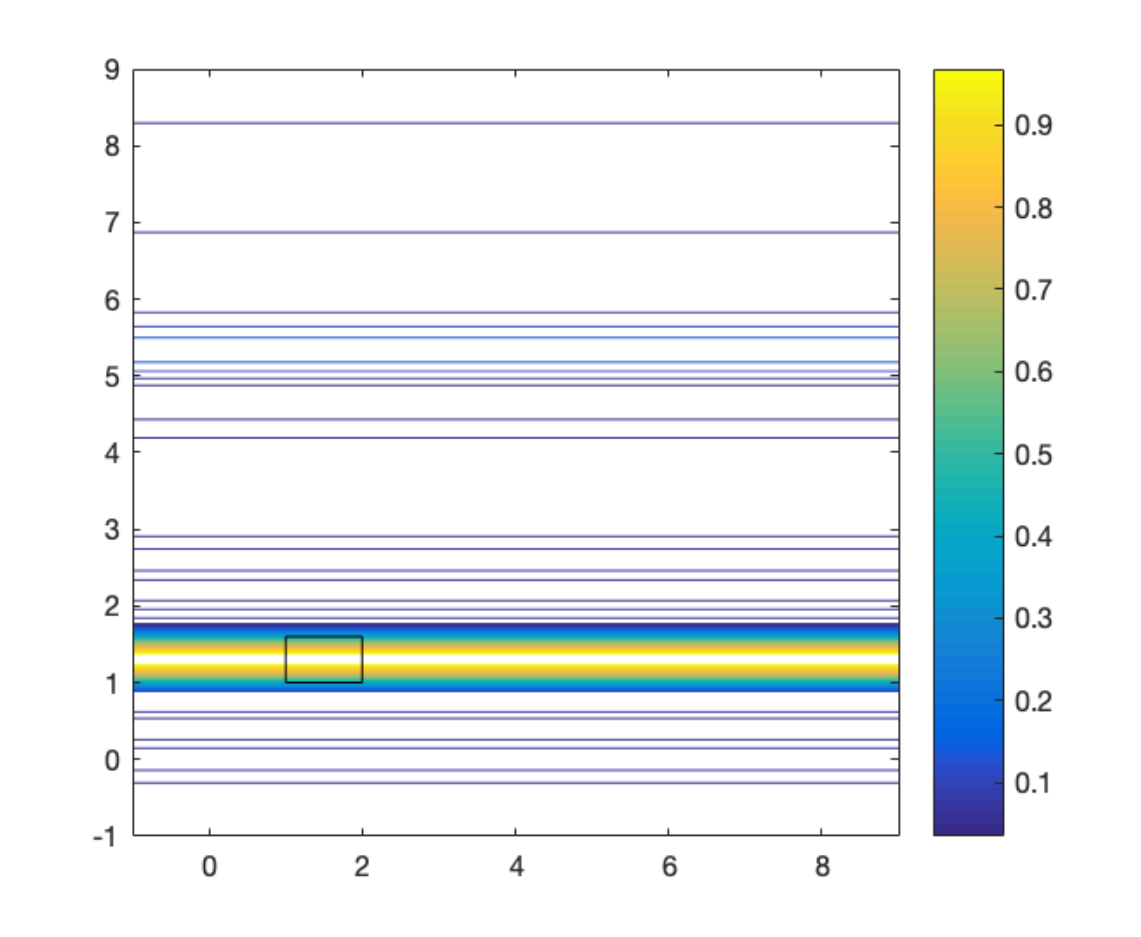}}
  \subfigure[\textbf{${\bf z_0}=(4,4)^T$.}]{
    \includegraphics[width=1.5in,height=1.5in]{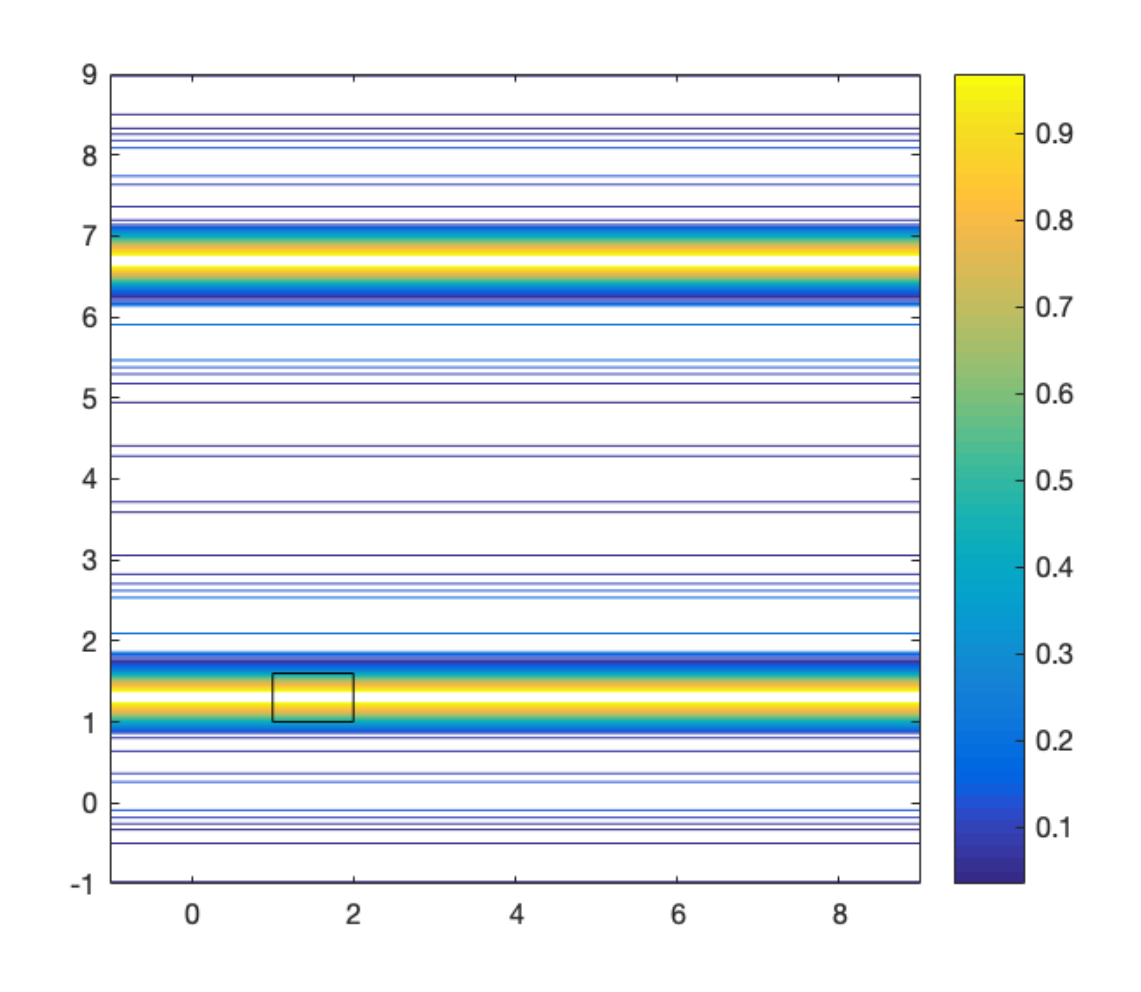}}
  \subfigure[\textbf{${\bf z_0}=(12,12)^T$.}]{
    \includegraphics[width=1.5in,height=1.5in]{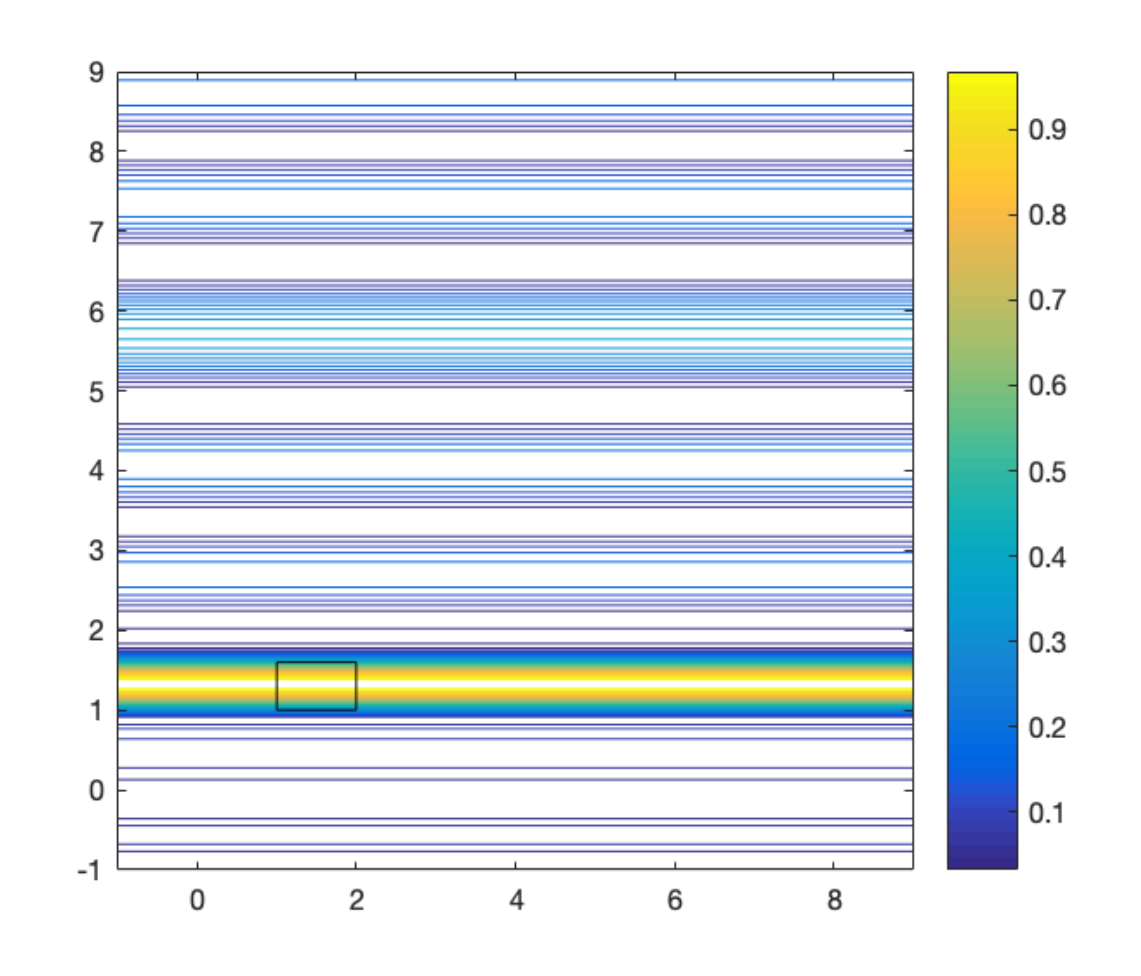}}
\caption{${\bf I^{\Theta}_{z_0,S}}$ with one observation direction for rectangle.}
\label{ObservationoneS}
\end{figure}

\begin{figure}[htbp]
  \centering
  \subfigure[\textbf{${\bf z_0}=(4,1.3)^T$.}]{
    \includegraphics[width=1.5in,height=1.5in]{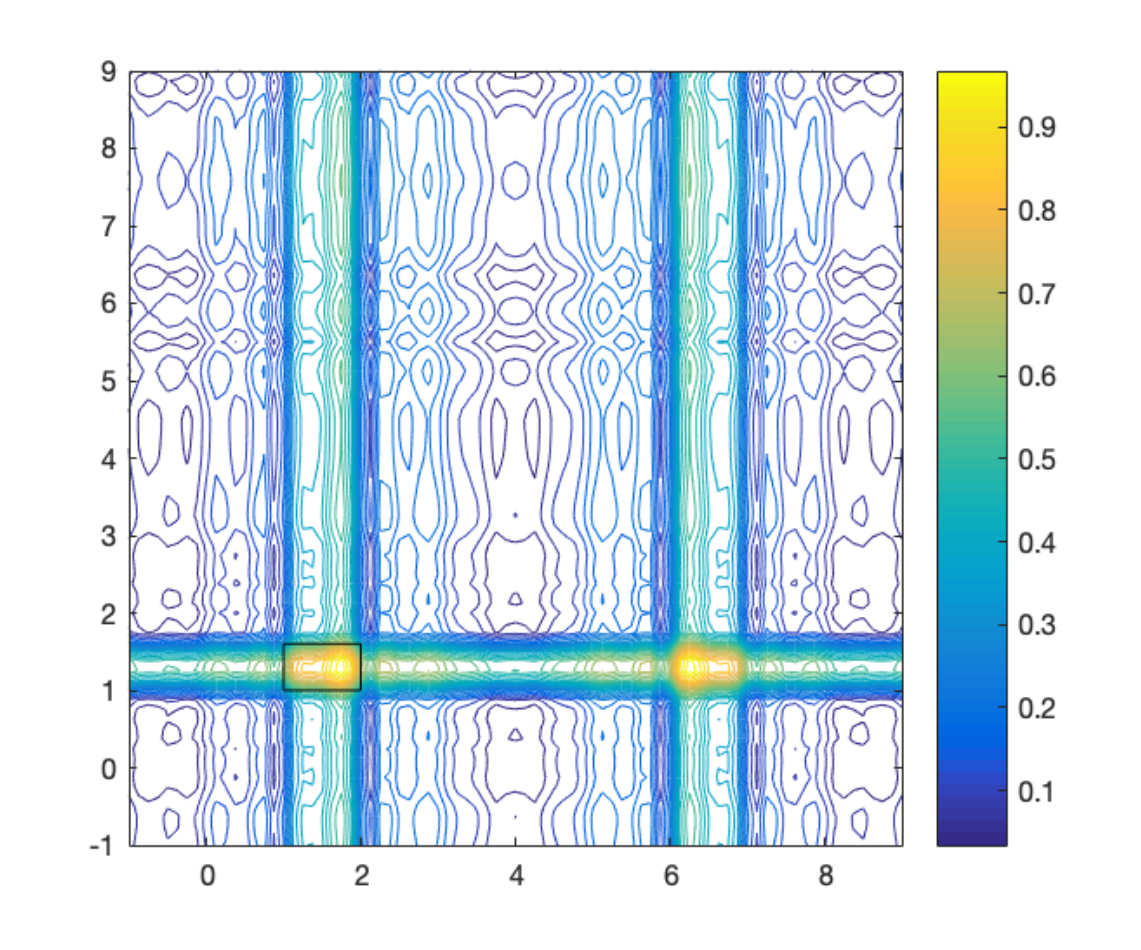}}
  \subfigure[\textbf{${\bf z_0}=(4,4)^T$.}]{
    \includegraphics[width=1.6in,height=1.5in]{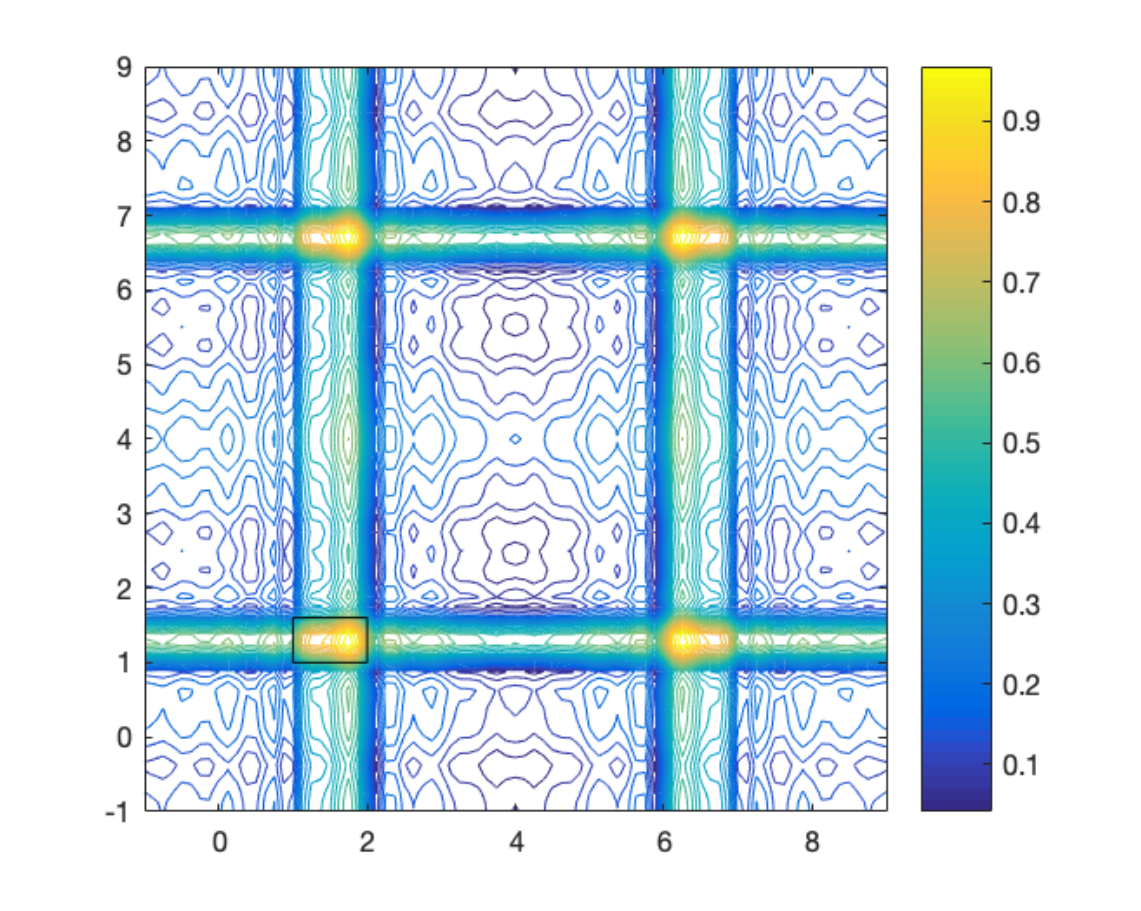}}
  \subfigure[\textbf{${\bf z_0}=(12,12)^T$.}]{
    \includegraphics[width=1.5in,height=1.5in]{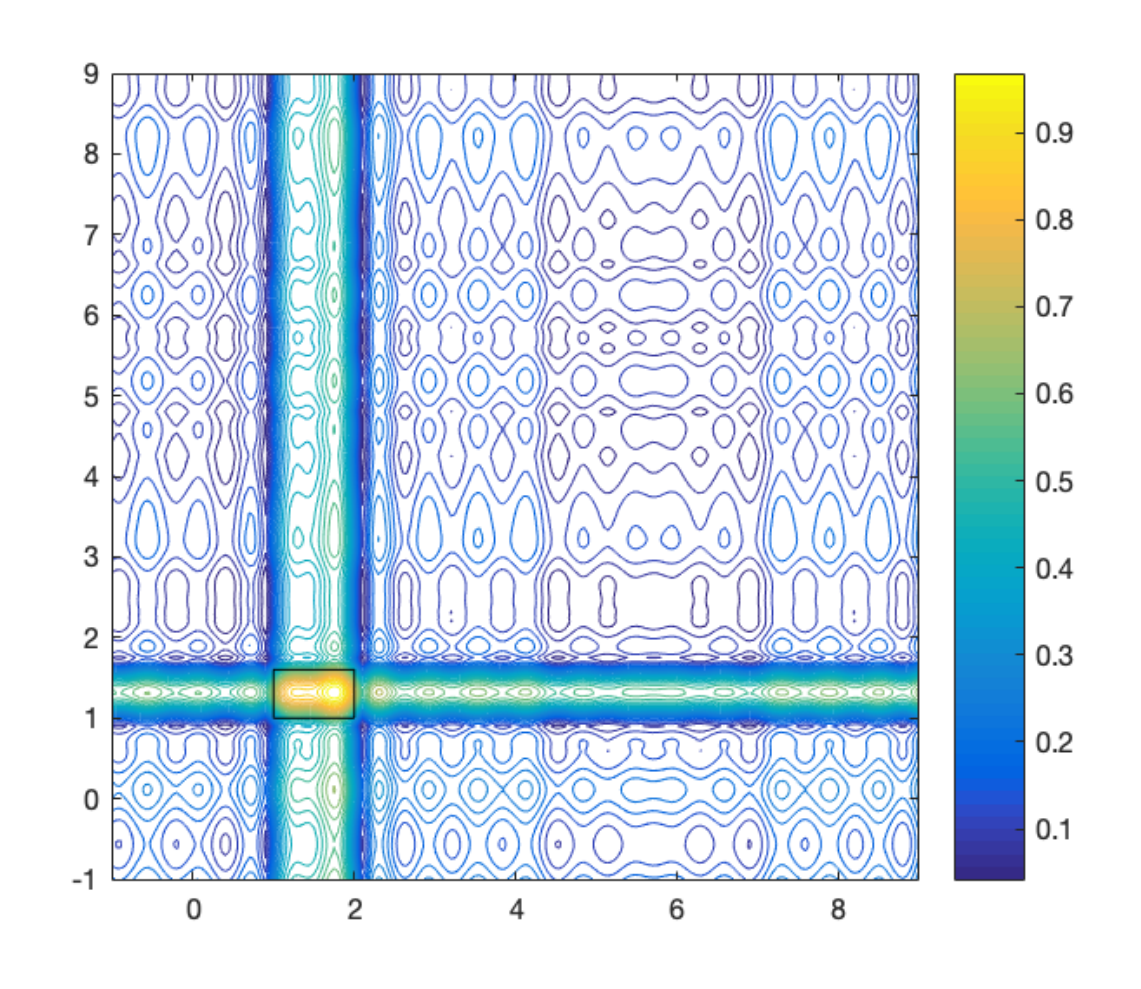}}
\caption{${\bf I^{\Theta}_{z_0,S}}$ with two observation directions for rectangle.}
\label{ObservationtwoS}
\end{figure}

\textbf{${\bf I^{\Theta}_{z_0,S}}$ with multiple observation directions}
We use the rectangle and the L-shaped domain in this example.
Now we use 20 observation angles $\theta_j, j=1,\cdots, 20$ such that  $\theta_j=-\pi/2+j\pi/20$. Note that $\theta_j\in(-\pi/2,\pi/2)$. Fig. \ref{ObservationmulS} gives the results for
rectangle with different ${\bf z_0}$. Fig. \ref{ObservationmulLS} gives the results for
the L-shaped domain. The locations and sizes
of support of ${\bf F}$ are reconstructed correctly.\\

\begin{figure}[htbp]
  \centering
  \subfigure[\textbf{${\bf z_0}=(4,1.3)^T$.}]{
    \includegraphics[width=1.5in,height=1.5in]{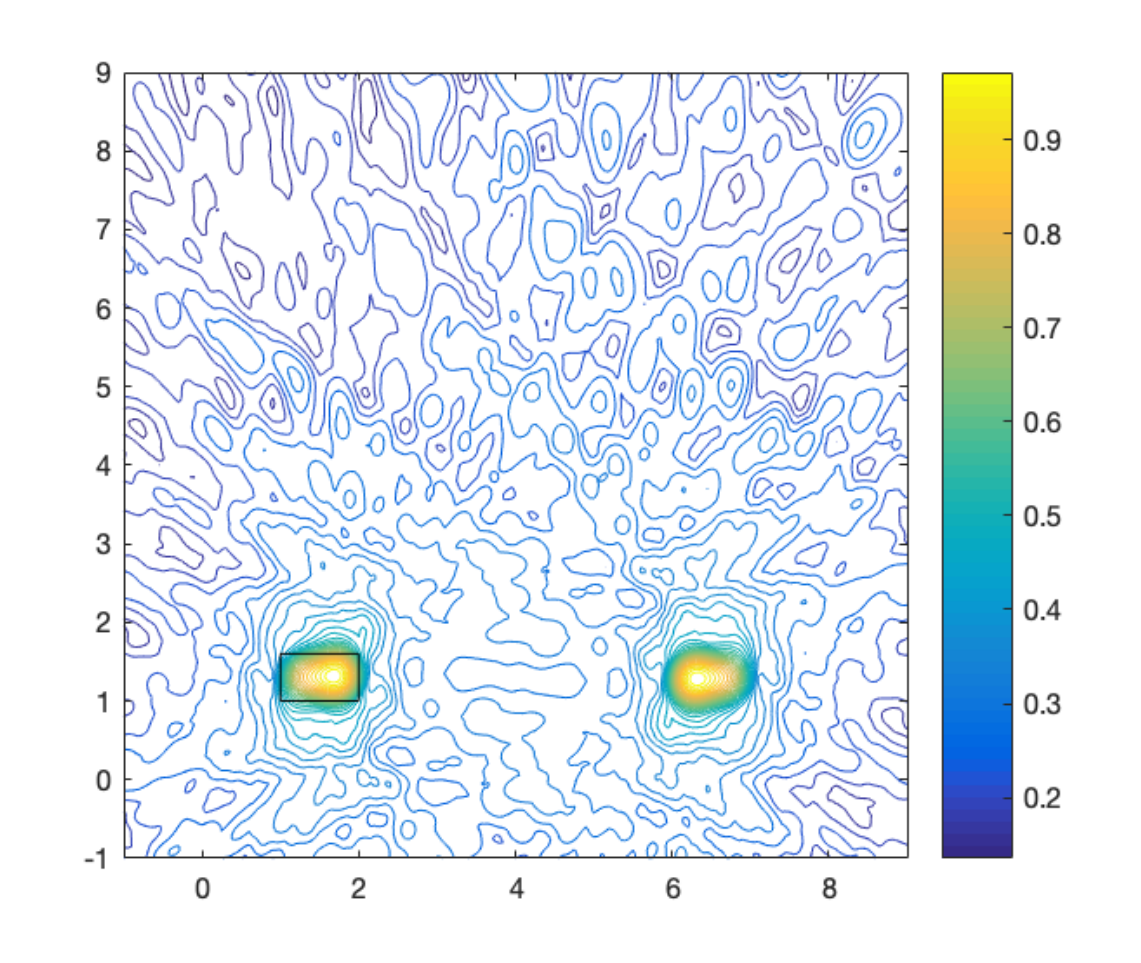}}
  \subfigure[\textbf{${\bf z_0}=(4,4)^T$.}]{
    \includegraphics[width=1.6in,height=1.5in]{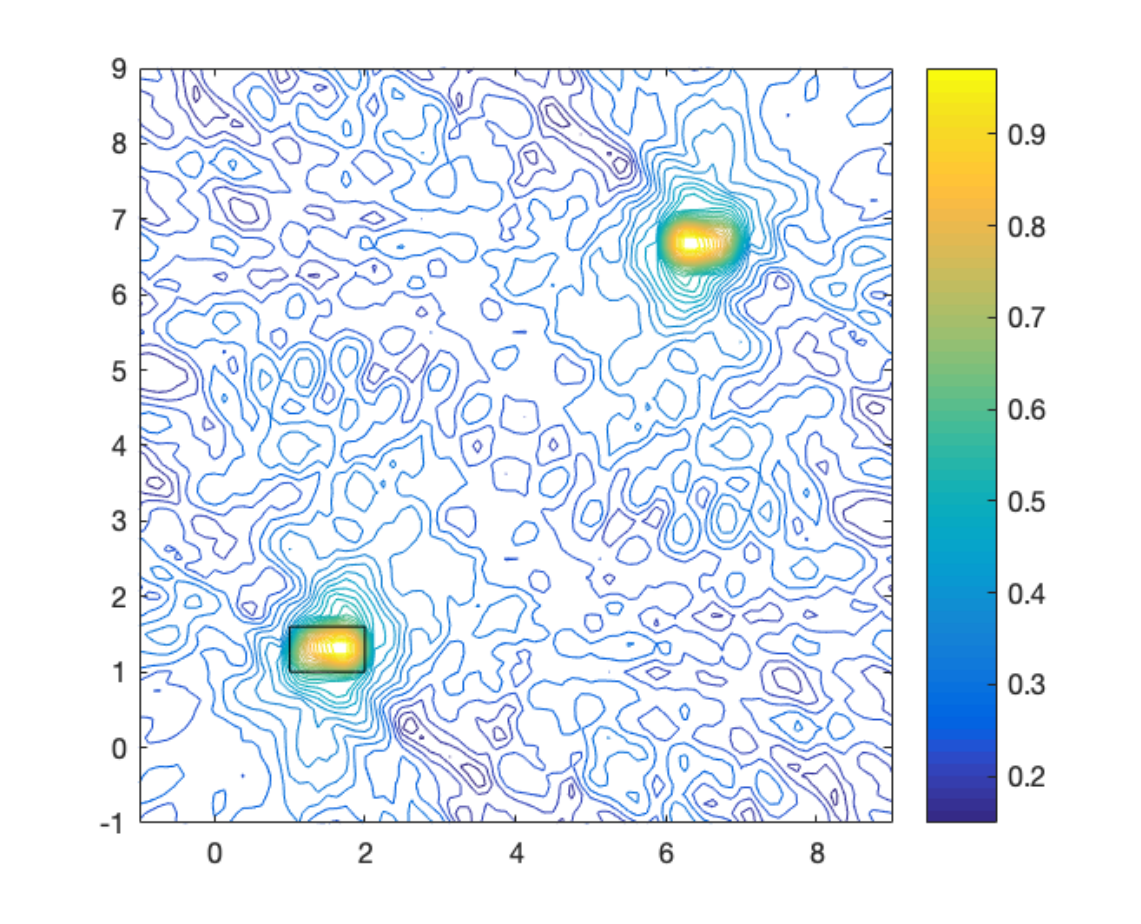}}
  \subfigure[\textbf{${\bf z_0}=(12,12)^T$.}]{
    \includegraphics[width=1.5in,height=1.5in]{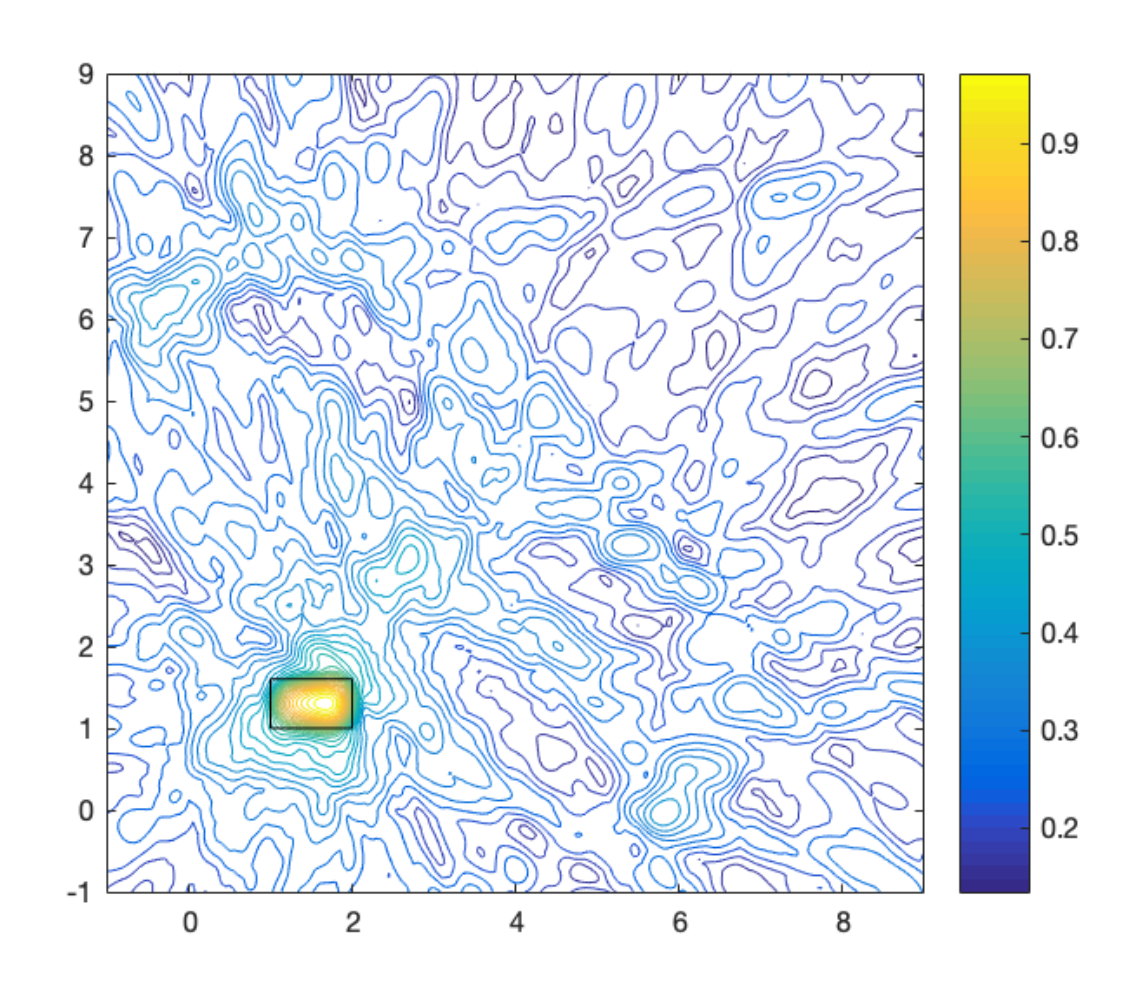}}
\caption{${\bf I^{\Theta}_{z_0,S}}$ with multiple observation directions for rectangle.}
\label{ObservationmulS}
\end{figure}

\begin{figure}[htbp]
  \centering
  \subfigure[\textbf{${\bf z_0}=(3,31/32)^T$.}]{
    \includegraphics[width=1.5in,height=1.5in]{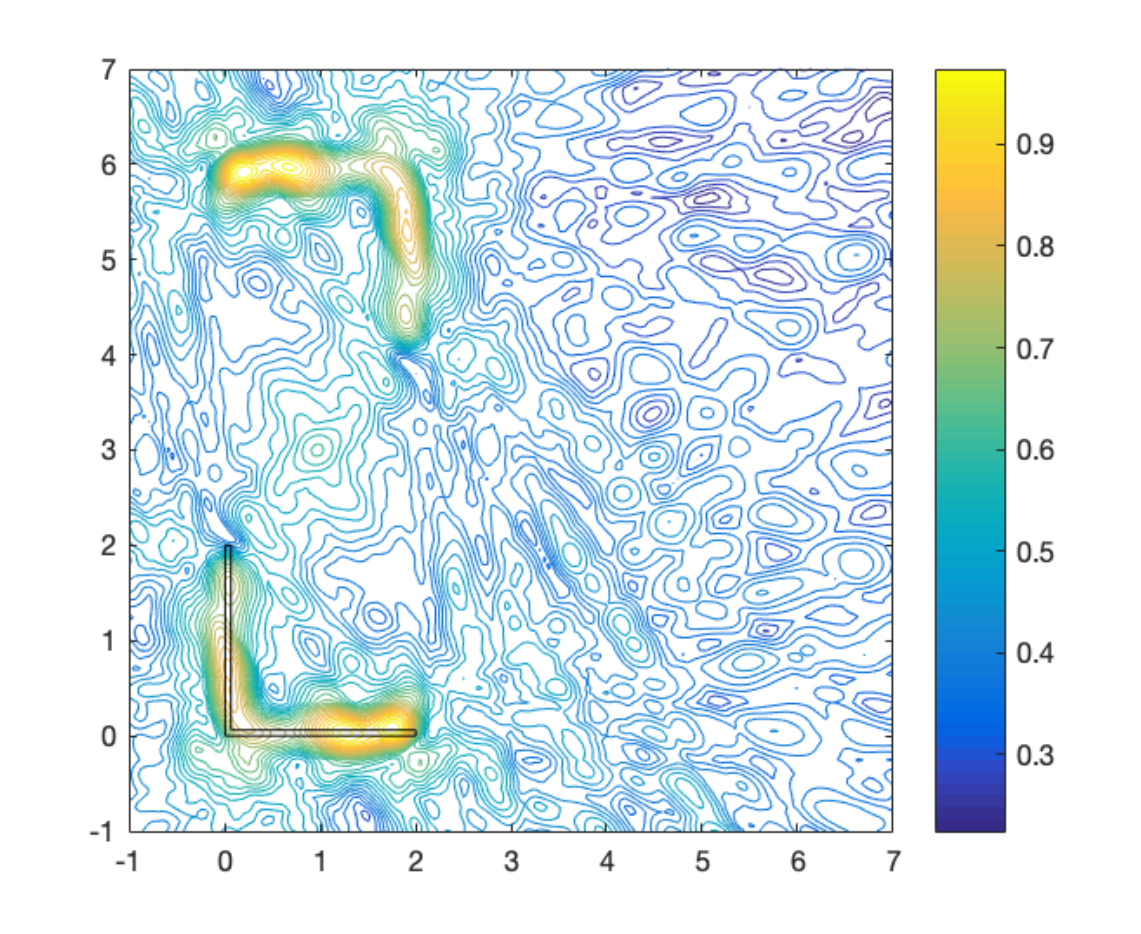}}
  \subfigure[\textbf{${\bf z_0}=(3,3)^T$.}]{
    \includegraphics[width=1.6in,height=1.5in]{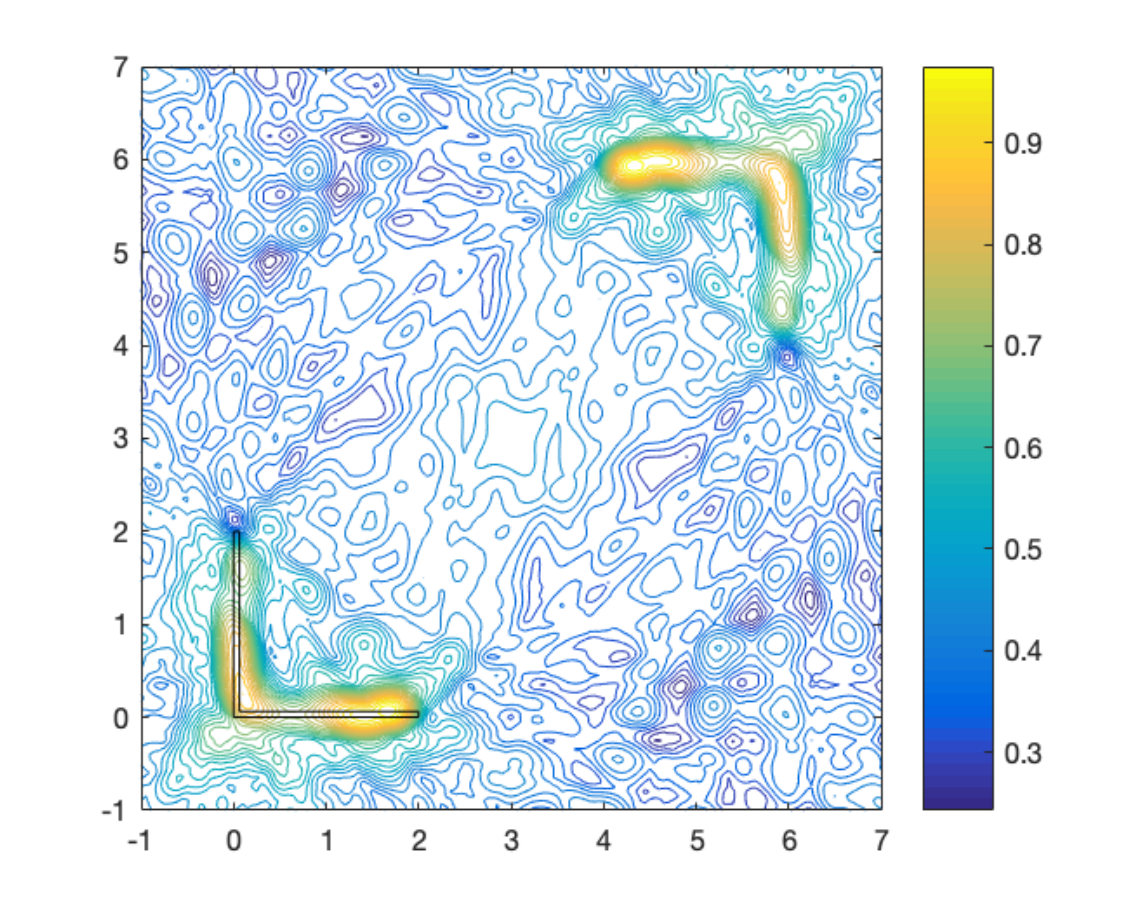}}
  \subfigure[\textbf{${\bf z_0}=(12,12)^T$.}]{
    \includegraphics[width=1.5in,height=1.5in]{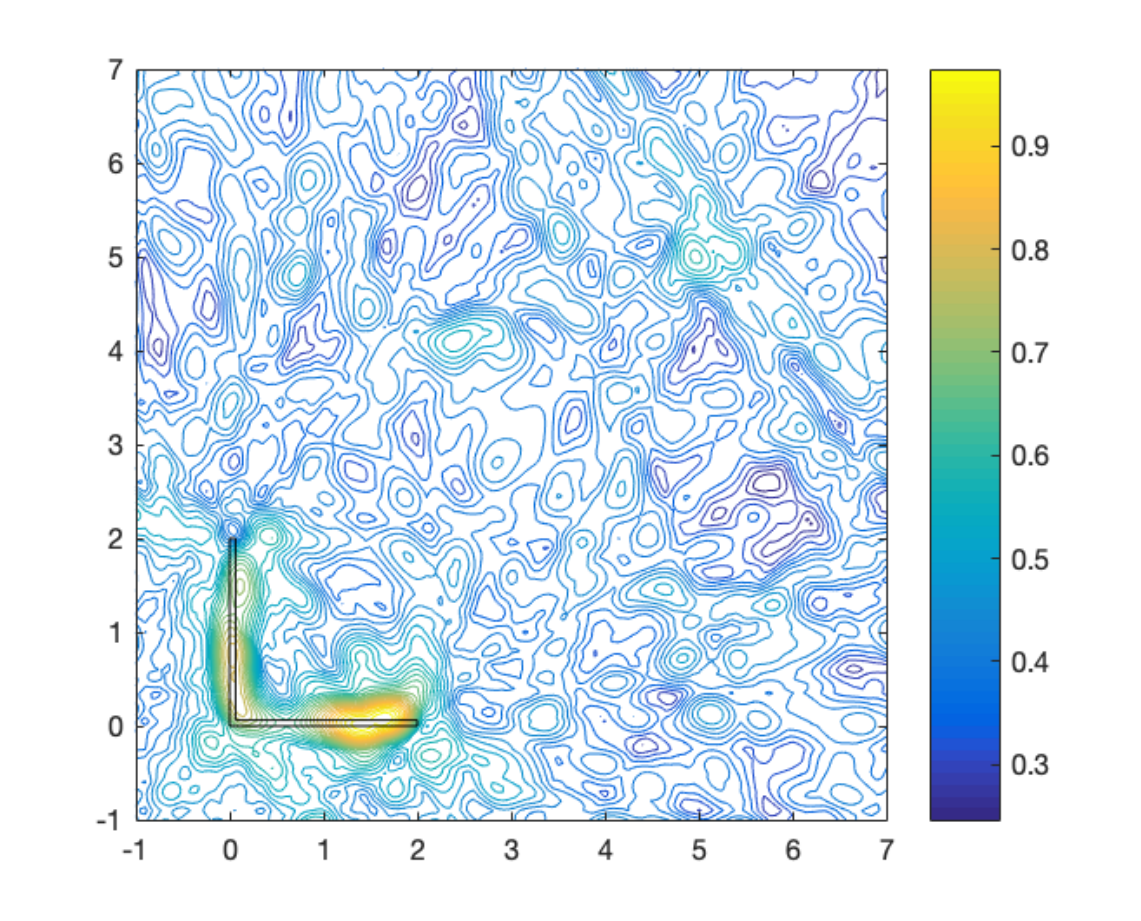}}
\caption{${\bf I^{\Theta}_{z_0,S}}$ with multiple observation directions for L-shaped domain.}
\label{ObservationmulLS}
\end{figure}

\textbf{The validation the phase retrieval scheme}
This example is designed to check the phase retrieval scheme
proposed before. The underlying scatterer is the rectangle.
In Fig. \ref{relativeS} and Fig. \ref{absoluteS} , we compare the phase retrieval data with the exact one, the real part of far field
pattern at a fixed direction $\mathbf{\hx}=(0,1)^T$ is given. We observe that the phaseless data
are reconstructed very well for small relative error level. These two figures also show that our phase
retrieval scheme is robust to noise.\\

\begin{figure}[htbp]
  \centering
  \subfigure[\textbf{$10\%$ noise.}]{
    \includegraphics[width=1.5in,height=1in]{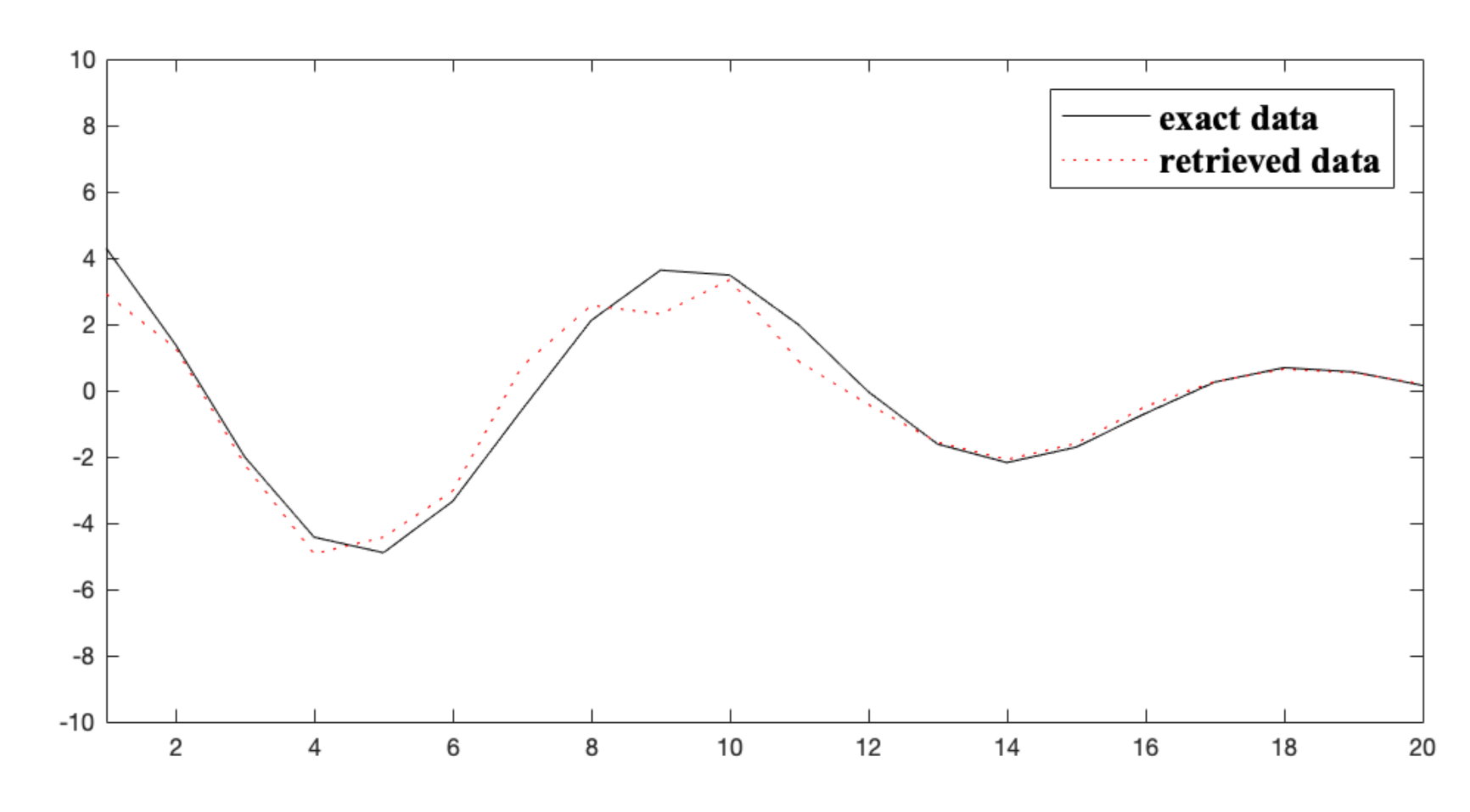}}
  \subfigure[\textbf{$30\%$ noise.}]{
    \includegraphics[width=1.5in,height=0.99in]{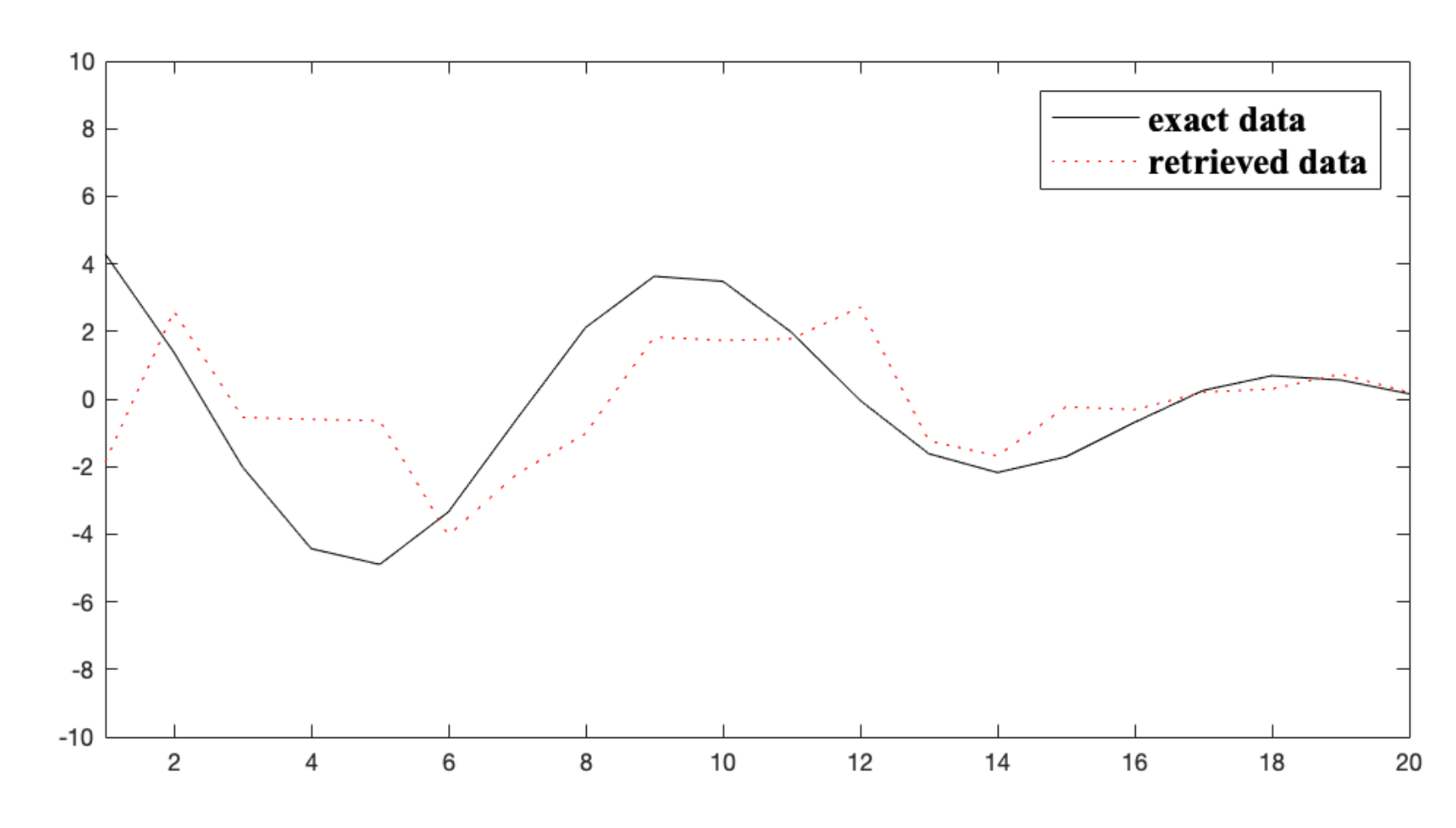}}
  \subfigure[\textbf{$50\%$ noise.}]{
    \includegraphics[width=1.5in,height=0.99in]{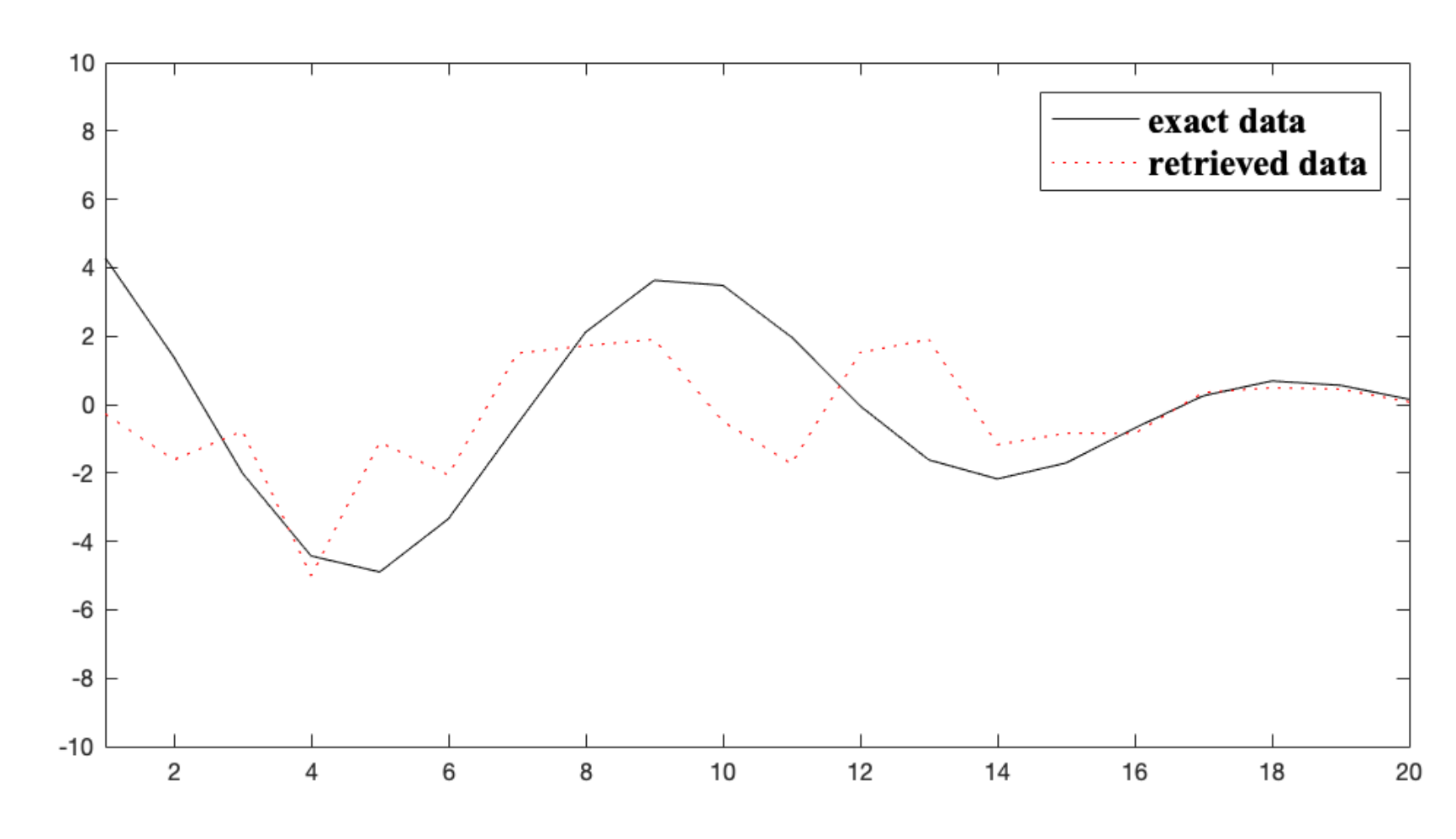}}
\caption{{\bf Example PhaseRetrieval.}\, Phase retrieval for the real part of the far field pattern with relative error at a fixed direction $\mathbf{\hx}=(0,1)^T$.}
\label{relativeS}
\end{figure}

\begin{figure}[htbp]
  \centering
  \subfigure[\textbf{0.1 noise.}]{
    \includegraphics[width=1.5in,height=1in]{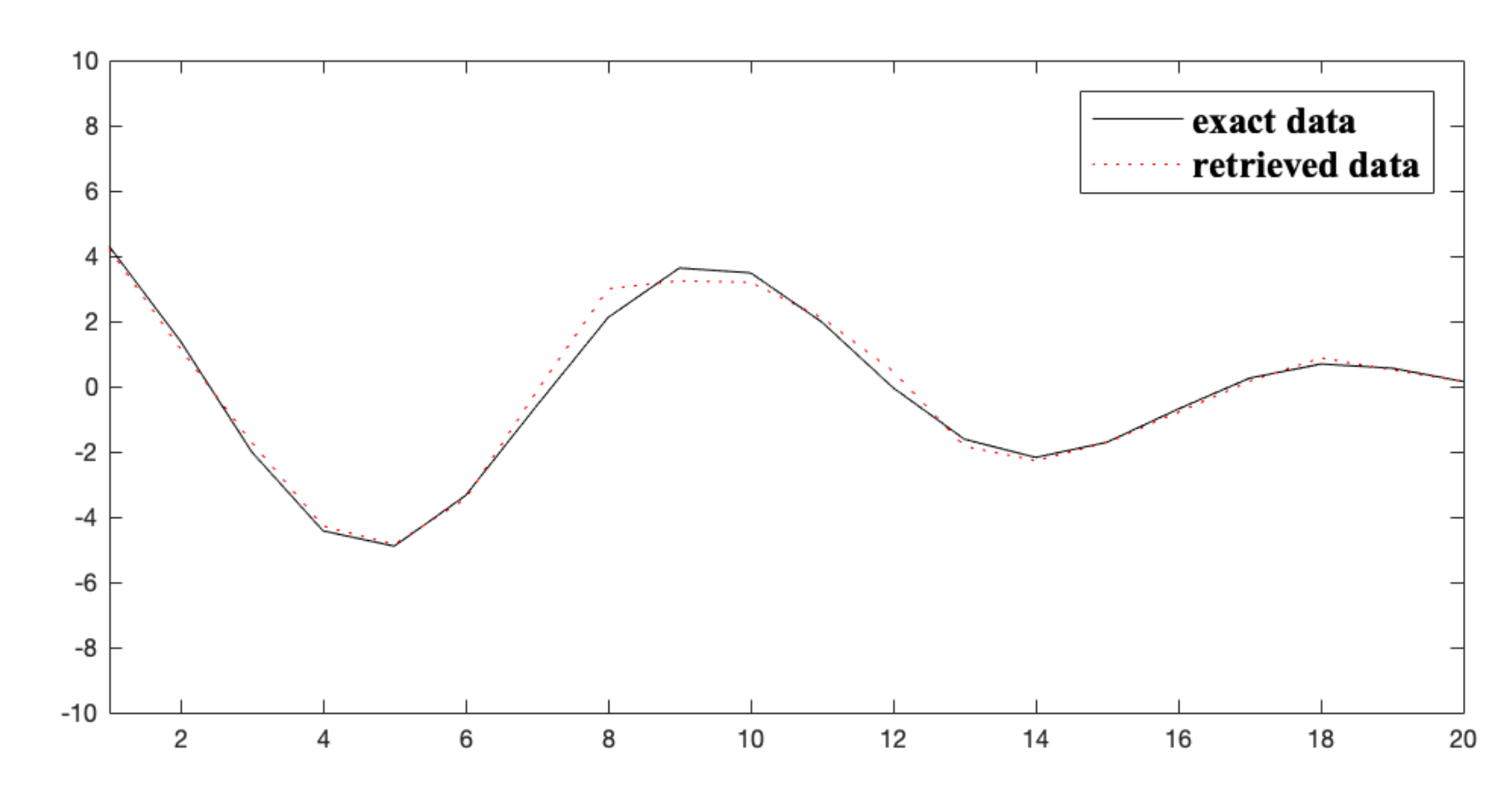}}
  \subfigure[\textbf{0.3 noise.}]{
    \includegraphics[width=1.5in,height=1in]{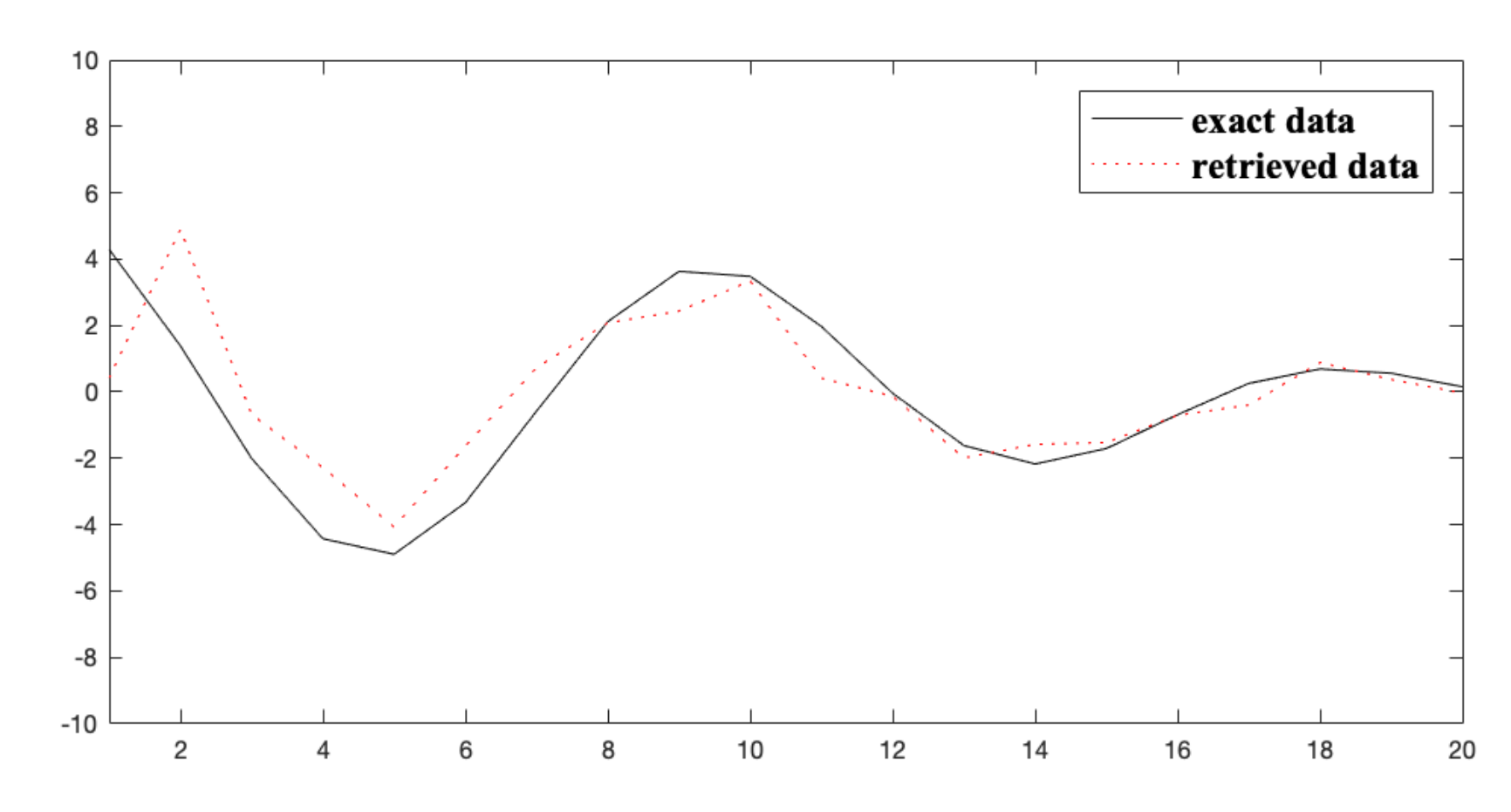}}
  \subfigure[\textbf{0.5 noise.}]{
    \includegraphics[width=1.5in,height=1in]{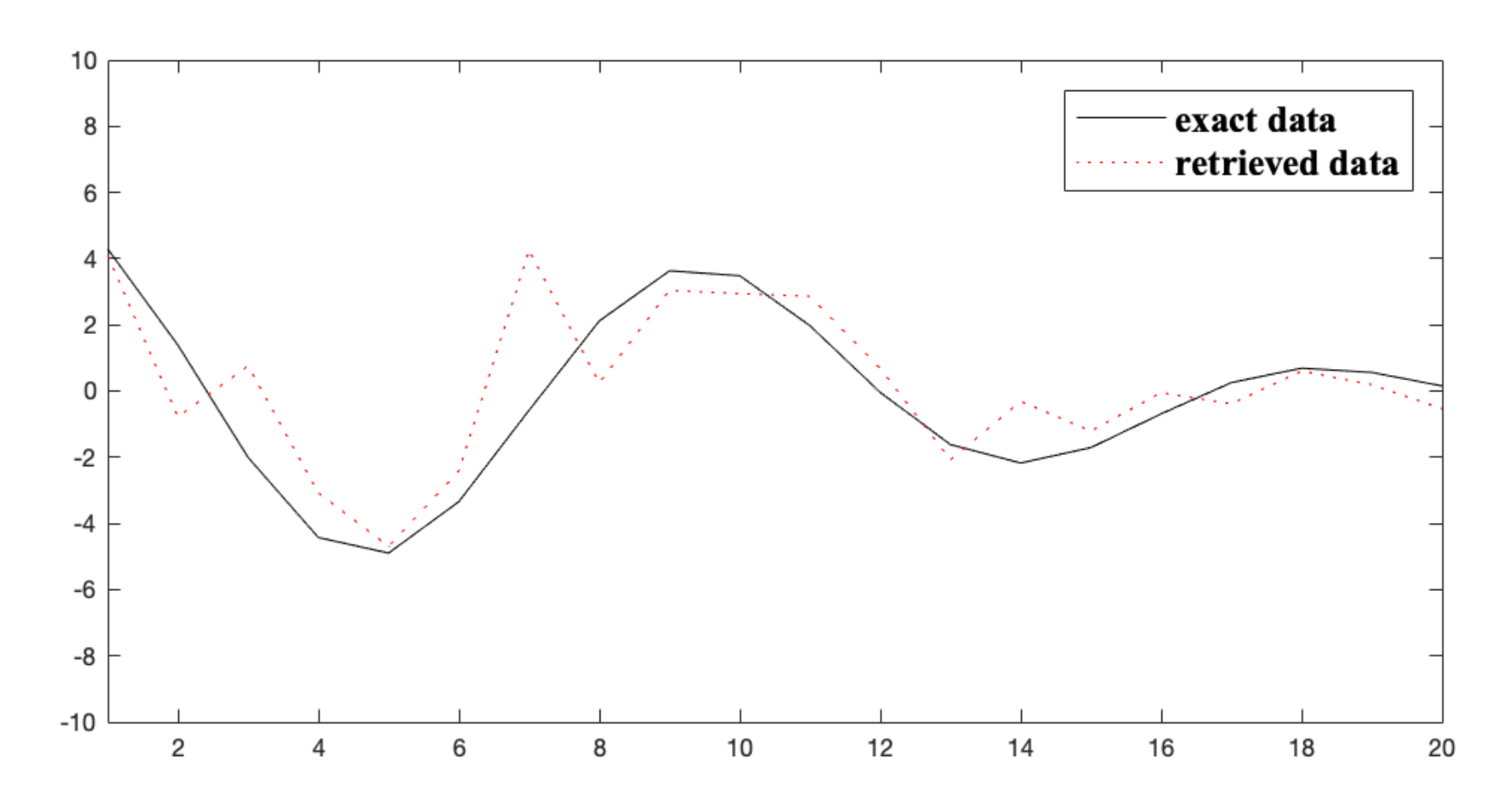}}
\caption{{\bf Example PhaseRetrieval.}\, Phase retrieval for the real part of the far field pattern with absolute error at a fixed direction $\mathbf{\hx}=(0,1)^T$.}
\label{absoluteS}
\end{figure}

\textbf{${\bf I_{\Theta}}(\mathbf p)$ for extended objects }
In this example, we take $\mathbf{z_0}=(4,4)^T$.  We show the reconstructions of extended objects with the same $20$ observation
directions. One is the L-shaped domain, the other is the triangle. Fig. \ref{PRexampleS} give the reconstructions.
\begin{figure}[htbp]
  \centering
  \subfigure[\textbf{L-shaped domain.}]{
    \includegraphics[width=2in,height=1.3in]{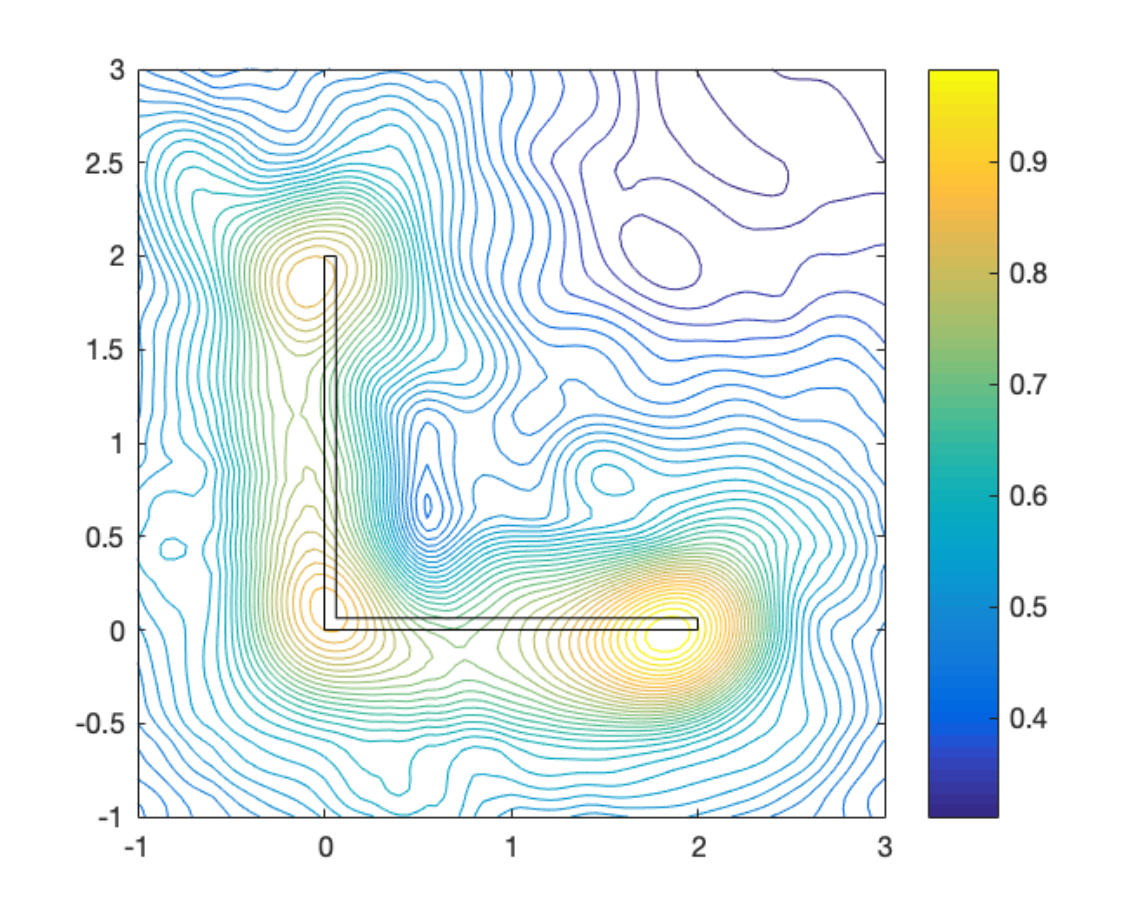}}
  \subfigure[\textbf{Triangle.}]{
    \includegraphics[width=2in,height=1.3in]{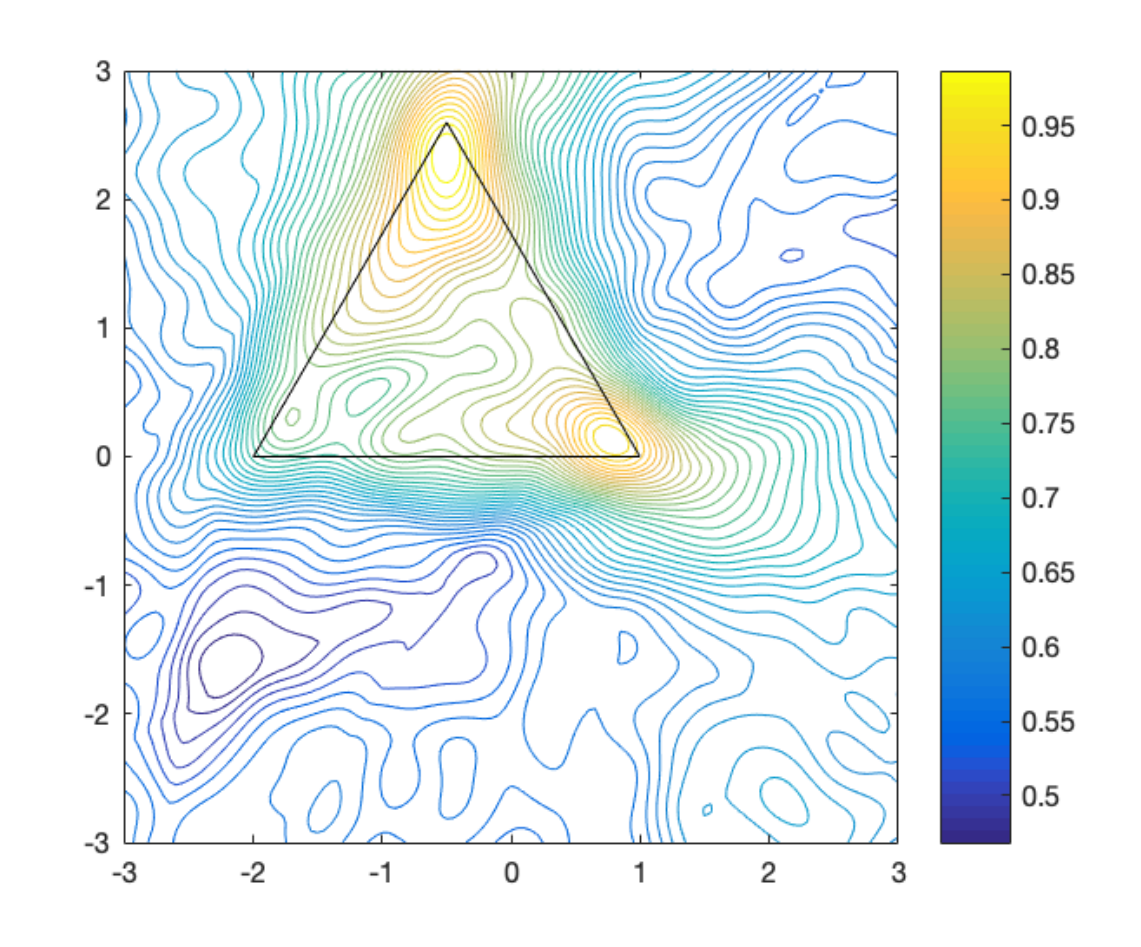}}
\caption{{\bf Example Extended Objects.}\, Reconstructions the phase retrieval scheme using multiple directions for different domains with $\mathbf{z_0}=(4,4)^T$.}
\label{PRexampleS}
\end{figure}

\section{Concluding remarks}
\label{sec:ConcludingRemarks}
\setcounter{equation}{0}
In this paper, we study systematically the inverse elastic scattering problems with phaseless far field data. By considering simultaneously the scattering of point sources, we establish some uniqueness results with phaseless far field data, propose a simple and stable phase retrieval technique and some direct sampling methods for shape reconstructions. The theoretical investigations are then complemented by numerical examples which exploit generated synthetic far-field
data for a variety of surfaces in two dimensions. The elaborated numerical reconstructions reveal that the phase retrieval technique is quite robust to noise and the proposed direct sampling methods are capable of identifying unknown objects  effectively, even only spare data are used.

Numerically, we observe that if the indicator ${\bf I_{\mathbf{z_0}}}\mathbf{(p,d)}$ given in \eqref{Indicator01} is replaced by
\be
&&{\bf \tilde{I}_{\mathbf{z_0}}}\mathbf{(p,d)}\cr
&:=&\sum_{\mathbf{q}\in\mathcal{Q}}\left|\int_{\mathbb{S}} \mathcal {F}_{\bf z_0}(\mathbf{\hx,d,q},\tau_1)\cos[k_s\mathbf{\hx\cdot(p-z_0)-k_s p\cdot d}]ds(\mathbf{\hx})\right|^2,\,\mathbf{d}\in\mathbb{S},\,\mathbf{p}\in\R^2,\qquad
\en
Then the false domain $\Om(\mathbf{z_0})$ disappear surprisingly. Unfortunately, we are not clear of the theory basis for this fact.

We are also interested in the phaseless total fields taken on some measurement surface containing the unknown objects. For the source scattering problems, the phase retrieval technique is also applicable. However, since the point sources are also radiating solutions to the Navier equation, to retrieve the phased data stably, we have to choose the source points close to the measurement surface. We will address this problem in a forthcoming paper.

\section*{Acknowledgement}
The research of X. Ji is partially supported by the NNSF of China with Grant Nos. 11271018 and 91630313,
and National Centre for Mathematics and Interdisciplinary Sciences, CAS.
The research of X. Liu is supported by the NNSF of China under grant 11571355 and the Youth Innovation Promotion Association, CAS.

\end{document}